%% file: IM_APAL1_arxiv.tex
\documentclass[11pt]{amsart}
\usepackage{amsfonts}
\usepackage{amssymb}
\usepackage{amsmath}
\usepackage{amsthm}
\usepackage[dvipdfm]{graphicx,color}
\usepackage{ascmac}
\usepackage{moreverb}
\usepackage{fancybox}
\usepackage{fancyvrb}
\usepackage{booktabs}
\usepackage{comment}
\usepackage{stmaryrd}
\usepackage{txfonts}
\usepackage[all]{xy}
\usepackage{calc}

\newlength{\hoffsettmp}
\newlength{\voffsettmp}
\makeindex

\theoremstyle{plain}
\newtheorem{theorem}{Theorem}

\newtheorem{cor}[theorem]{Corollary}
\newtheorem{prop}[theorem]{Proposition}

\theoremstyle{definition}
\newtheorem{definition}[theorem]{Definition}
\newtheorem{example}[theorem]{Example}

\theoremstyle{definition}
\newtheorem*{remark}{Remark}

\newtheorem*{notation}{Notation}

\newtheorem*{ack}{Acknowledgements}

\newcommand{\lrangle}[1]{\langle #1 \rangle}
\newcommand{\pair}[1]{( #1 )}
\newcommand{\res}{\upharpoonright}
\newcommand{\dg}[1]{\mathbf{#1}}
\newcommand{\fr}{\mbox{}^{\smallfrown}}

\newcommand{\lsup}{\otimes}
\newcommand{\linf}{\oplus}

\newcommand{\binf}{\bigoplus}

\newcommand{\nn}{\mathbb{N}}

\newcommand{\tie}{\triangledown}
\newcommand{\btie}{\bigtriangledown}
\newcommand{\htie}{\blacktriangledown}
\newcommand{\bhtie}{\mbox{\Large{$\blacktriangledown$}}}
\newcommand{\lcm}{\tie_{\omega}}
\newcommand{\blcm}{\btie_{\omega}\mbox{}}
\newcommand{\cls}{\tie_\infty}
\newcommand{\bcls}{\btie_\infty\mbox{}}
\newcommand{\hjump}[1]{\htie(#1)}

\newcommand{\cmeet}{\binf\mbox{}^{\longrightarrow}}
\newcommand{\ntie}{\fr}
\newcommand{\shft}{\mbox{}\leftharpoonup}
\newcommand{\concat}{\bigsqcap}

\newcommand{\lrceil}[1]{{#1}^{\bullet}}
\newcommand{\bhk}[1]{\llbracket #1 \rrbracket}
\newcommand{\bigbhk}[1]{\left\llbracket #1 \right\rrbracket}

\title[Inside the Muchnik Degrees I]{Inside the Muchnik Degrees I:\\Discontinuity, Learnability and Constructivism}

\author{K.~Higuchi}
\address{Department of Mathematics and Informatics, Chiba University, 1-33 Yayoi-cho, Inage, Chiba, Japan}
\email{khiguchi@g.math.s.chiba-u.ac.jp}

\author{T.~Kihara}
\address{School of Information Science, Japan Advanced Institute of Science and Technology, Nomi 923-1292, Japan}
\email{kihara@jaist.ac.jp}


\begin{document}



\begin{abstract}
Every computable function has to be continuous.
To develop computability theory of discontinuous functions, we study low levels of the arithmetical hierarchy of nonuniformly computable functions on Baire space.
First, we classify nonuniformly computable functions on Baire space from the viewpoint of learning theory and piecewise computability.
For instance, we show that mind-change-bounded-learnability is equivalent to finite $(\Pi^0_1)_2$-piecewise computability (where $(\Pi^0_1)_2$ denotes the difference of two $\Pi^0_1$ sets), error-bounded-learnability is equivalent to finite $\Delta^0_2$-piecewise computability, and learnability is equivalent to countable $\Pi^0_1$-piecewise computability (equivalently, countable $\Sigma^0_2$-piecewise computability).
Second, we introduce disjunction-like operations such as the coproduct based on BHK-like interpretations, and then, we see that these operations induce Galois connections between the Medvedev degree structure and associated Medvedev/Muchnik-like degree structures.
Finally, we interpret these results in the context of the Weihrauch degrees and Wadge-like games.
\end{abstract}

\maketitle

\input{NRMP_fullproof1.tex}

\tableofcontents

\input{NRMP_fullproof1-2a.tex}
\input{NRMP_fullproof1-2b.tex}

\input{NRMP_fullproof1-2c.tex}
\input{NRMP_fullproof1-2d.tex}
\input{NRMP_fullproof5.tex}

\begin{ack}\upshape
The authors were partially supported by Grant-in-Aid for JSPS fellows.
The second author (Kihara) would like to thank Douglas Cenzer, Hajime Ishihara, Dick de Jongh, Arno Pauly, and Albert Visser, for valuable comments and helpful discussion, and the second author also would like to thank Makoto Tatsuta and Yoriyuki Yamagata for introducing him to the syntactical study on Limit Computable Mathematics.
The second author is also grateful to Sam Sanders who helped his English writing.
Finally, the authors would like to thank the anonymous referees for their valuable comments and suggestions.
\end{ack}

\addcontentsline{toc}{section}{Bibliography}
\bibliographystyle{plain}
\bibliography{IMref}

\printindex

\end{document}

%% file: NRMP_fullproof1.tex
\section{Summary}

\subsection{Introduction}

Imagine the floor function, a real function that takes the integer part of an input.
Although it seems easy to draw a rough graph of the floor function, it is {\em not} computable with respect to the standard real number representation \cite{Wei}, because computability automatically induces topological continuity.
One way to study the floor function in computability theory is to ``{\em computabilize}'' it by changing the representation/topology of the real space (see, for instance, \cite{YT}).
However, it is also important to enhance our knowledge of the noncomputability/discontinuity level of such seemingly computable functions without changing representation/topology.
Our main objective is to study low levels of the arithmetical/Baire hierarchy of functions on Baire space from the viewpoint of approximate computability/continuity and piecewise computability/continuity.

We postulate that a {\em nearly computable} function shall be, at the very least, {\em nonuniformly computable}, where a function $f$ is said to be nonuniformly computable if for every input $x$, there exists an algorithm $\Psi_x$ that computes $f(x)$ using $x$ as an oracle, where we do not require the map $x\mapsto\Psi_x$ to be computable.
The notion of nonuniform computability naturally arises in Computable Analysis \cite{BG,Zie4}.
However, of course, most nonuniformly computable discontinuous functions are far from being computable.
Then, what type of discontinuous functions are recognized as being nearly computable?
A nearly computable/continuous function has to be approximated using computable/continuous functions.
For instance, a Baire function appears to be {\em dynamically approximated} by a sequence of continuous functions and a piecewise continuous ($\sigma$-continuous) function appears to be {\em statically approximated} by countably many continuous functions.

There have been many challenges \cite{BP,Zie1,Zie2,Zie3,Zie4,YMT,YT} in developing computability theory of (nonuniformly computable) discontinuous functions using the notion of {\em learnability} (dynamical-approximation) and {\em piecewise computability} (statical-approximation).
Indeed, one can show the equivalence of effective learnability and $\Pi^0_1$-piecewise computability: the class of functions that are computable with finitely many mind changes is exactly the class of functions that are decomposable into countably many computable functions with $\Pi^0_1$ domains.
In this paper, we introduce various concepts of dynamic-approximability, and then, we characterize these concepts as static-approximability.

Now, we focus our attention on the concepts lying between (uniform) computability and nonuniform computability.
In 1950-60th, Medvedev \cite{Med} and Muchnik \cite{Muc} introduced the degree structure induced by uniform and nonuniform computability to formulate semantics for the intuitionistic propositional calculus based on Kolmogorov's idea of interpreting each proposition as a problem.
The degree structure induced by the Medvedev (Muchnik) reduction forms a Brouwer algebra (the dual of a Heyting algebra), where the (intuitionistic) disjunction is interpreted as the coproduct of subsets of Baire space.

Our objective is to reveal the hidden relationship between the hierarchy of nonuniformly computable functions and the hierarchy of disjunction operations.
When a certain suitable disjunction-like operation such as the coproduct is introduced, we will see that one can recover the associated degree structure from the disjunction operation.
As a consequence, we may understand the noncomputability feature of functions by observing the degree-theoretic behavior of associated disjunction operations.
This phenomenon can be explained by using the terminology of Galois connections or adjoint functors.
For instance, one can introduce a disjunction operation on Baire space using the limit-BHK interpretation of {\em Limit Computable Mathematics} \cite{Hay} (abbreviated as {\sf LCM}), a type of constructive mathematics based on Learning Theory, whose positive arithmetical fragment is characterized as Heyting arithmetic with the recursive $\omega$-rule and the $\Sigma^0_1$ law of excluded middle \cite{BY,TB}.
Then, the ``limit-BHK disjunction'' includes all the information about the reducibility notion induced by learnable functions on Baire space.

Furthermore, in this paper, we introduce more complicated disjunction-like operations using BHK-like interpretations represented as ``dynamic proof models'' or ``nested models''.
For instance, a dynamic disjunction along a well-founded tree realizes the concept of learnability with ordinal-bounded mind changes, and a dynamic disjunction along an ill-founded tree realizes the concept of decomposability into countably many computable functions along a $\Sigma^0_2$ formula.

We also interpret these results in the context of the Weihrauch degrees and Wadge-like games.
We introduce a partial interpretation of nonconstructive principles including ${\sf LLPO}$ and ${\sf LPO}$ in the Weihrauch degrees and characterize the noncomputability/discontinuity level of nearly computable functions using these principles.

\subsection{Results}

In section $2$, we introduce the notion of $(\alpha,\beta|\gamma)$-computability for partial functions on $\nn^\nn$, for each ordinal $\alpha,\beta,\gamma\leq\omega$.
Then, the notion of $(\alpha,\beta|\gamma)$-computability induces just seven classes closed under composition.
\index{$[\mathfrak{C}_T]^\alpha_\beta$}%
\index{$[\mathfrak{C}_T]^\alpha_{\beta\mid\gamma}$}%
\begin{itemize}
\item
$[\mathfrak{C}_T]^1_1$ denotes the set of all partial computable functions on $\nn^\nn$.
\item
$[\mathfrak{C}_T]^1_{<\omega}$ denotes the set of all partial functions on $\nn^\nn$ learnable with bounded mind changes.
\item
$[\mathfrak{C}_T]^1_{\omega|<\omega}$ denotes the set of all partial functions on $\nn^\nn$ learnable with bounded errors.
\item
$[\mathfrak{C}_T]^1_{\omega}$ denotes the set of all partial learnable functions on $\nn^\nn$.
\item 
$[\mathfrak{C}_T]^{<\omega}_{1}$ denotes the set of all partial $k$-wise computable functions on $\nn^\nn$ for some $k\in\nn$.
\item 
$[\mathfrak{C}_T]^{<\omega}_{\omega}$ denotes the set of all partial functions on $\nn^\nn$ learnable by a team.
\item 
$[\mathfrak{C}_T]^\omega_1$ denotes the set of all partial nonuniformly computable functions on $\nn^\nn$ (i.e., all functions $f$ satisfying $f(x)\leq_Tx$ for any $x\in{\rm dom}(f)$).
\end{itemize}

We will see that the following inclusions hold.
\begin{center}\footnotesize
\begin{tabular}{ccccc}
 &\rotatebox[origin=r]{30}{$\subset$}&$[\mathfrak{C}_T]^{<\omega}_1$&\rotatebox{-30}{$\subset$}& \\
$[\mathfrak{C}_T]^1_1\;\subset\;[\mathfrak{C}_T]^1_{<\omega}\;\subset\;[\mathfrak{C}_T]^1_{\omega|<\omega}$& & & & $[\mathfrak{C}_T]^{<\omega}_\omega\;\subset\;[\mathfrak{C}_T]^\omega_1$\\
 &\rotatebox[origin=r]{-30}{$\subset$}&$[\mathfrak{C}_T]^1_\omega$&\rotatebox{30}{$\subset$}& 
\end{tabular}
\end{center}
These notions are characterized as the following piecewise computability notions, respectively.
\index{${\rm dec}^X_x[\Gamma]$}%

\begin{itemize}
\item 
${\rm dec}^1_{\rm p}[-]$ also denotes the set of all partial computable functions on $\nn^\nn$.
\item 
${\rm dec}^{<\omega}_{\rm d}[\Pi^0_1]$ denotes the set of all partial functions on $\nn^\nn$ that are decomposable into finitely many partial computable functions with $(\Pi^0_1)_2$ domains, where a $(\Pi^0_1)_2$ set is the difference of two $\Pi^0_1$ sets.
\item 
${\rm dec}^{<\omega}_{\rm p}[\Delta^0_2]$ denotes the set of all partial functions on $\nn^\nn$ that are decomposable into finitely many partial computable functions with $\Delta^0_2$ domains.
\item 
${\rm dec}^{\omega}_{\rm p}[\Pi^0_1]$ denotes the set of all partial functions on $\nn^\nn$ that are decomposable into countably many partial computable functions with $\Pi^0_1$ domains.
\item 
${\rm dec}^{<\omega}_{\rm p}[-]$ denotes the set of all partial functions on $\nn^\nn$ that are decomposable into finitely many partial computable functions.
\item 
${\rm dec}^{<\omega}_{\rm p}{\rm dec}^{\omega}_{\rm p}[\Pi^0_1]$ denotes the set of all partial functions on $\nn^\nn$ that are decomposable into finitely many partial $\Pi^0_1$-piecewise computable functions.
\item 
${\rm dec}^{\omega}_{\rm p}[-]$ denotes the set of all partial functions on $\nn^\nn$ that are decomposable into countably many partial computable functions.
\end{itemize}

\begin{center}\footnotesize
\begin{tabular}{ccccc}
 &\rotatebox[origin=r]{30}{$\subset$}&${\rm dec}^{<\omega}_{\rm p}[-]$&\rotatebox{-30}{$\subset$}& \\
${\rm dec}^1_{\rm p}[-]\;\subset\;{\rm dec}^{<\omega}_{\rm d}[\Pi^0_1]\;\subset\;{\rm dec}^{<\omega}_{\rm p}[\Delta^0_2]$& & & & ${\rm dec}^{<\omega}_{\rm p}{\rm dec}^{\omega}_{\rm p}[\Pi^0_1]\;\subset\;{\rm dec}^{\omega}_{\rm p}[-]$\\
 &\rotatebox[origin=r]{-30}{$\subset$}&${\rm dec}^{\omega}_{\rm p}[\Pi^0_1]$&\rotatebox{30}{$\subset$}& 
\end{tabular}
\end{center}

In Section $3$, we formalize the disjunction operations.
Medvedev interpreted the intuitionistic disjunction as the coproduct (direct sum) $\oplus:\mathcal{P}(\nn^\nn)\times\mathcal{P}(\nn^\nn)\to\mathcal{P}(\nn^\nn)$.
We will introduce the following disjunction operations $\bhk{\cdot\vee\cdot}^*_{*}:\mathcal{P}(\nn^\nn)\times\mathcal{P}(\nn^\nn)\to\mathcal{P}(\nn^\nn)$:
\index{$\bhk{\cdot\vee\cdot}^*_{*}$}%
\begin{itemize}
\item 
$\bhk{\cdot\vee\cdot}^3_{\sf LCM[n]}$ is the disjunction operation on $\mathcal{P}(\nn^\nn)$ induced by the backtrack BHK-interpretation with mind-changes $<n$.
\item 
$\bhk{\cdot\vee\cdot}^2_{\sf LCM}$ is the disjunction operation on $\mathcal{P}(\nn^\nn)$ induced by the two-tape BHK-interpretation with finitely many mind-changes.
\item 
$\bhk{\cdot\vee\cdot}^3_{\sf LCM}$ is the disjunction operation on $\mathcal{P}(\nn^\nn)$ induced by the backtrack BHK-interpretation with finitely many mind-changes.
\item 
$\bhk{\cdot\vee\cdot}^2_{\sf CL}$ is the disjunction operation on $\mathcal{P}(\nn^\nn)$ induced by the two-tape BHK-interpretation permitting unbounded mind-changes.
\end{itemize}
Then, the direct sum $\oplus$ is characterized as the {\sf LCM} disjunction without mind-changes $\bhk{\cdot\vee\cdot}^3_{\sf LCM[1]}$.
In section $5$, we also introduce more complicated disjunction operations, which will play key roles in Part II.

In section $4$, we study the interaction between the disjunction operations and the learnable/piecewise computable functions.
We will construct new operations by iterating the disjunction operations introduced in Section $3$ in the following way:

\begin{center}\footnotesize
\begin{tabular}{ccccc}
 &\rotatebox[origin=r]{30}{$\geq$}&$\bigoplus_{m\in\nn}\bhk{\bigvee^{(m)}P}^{2}_{{\sf CL}}$&\rotatebox{-30}{$\geq$}& \\
$P\;\geq\;\bigoplus_{m\in\nn}\bhk{P\vee P}^{3}_{{\sf LCM}[m]}\;\geq\;\bigoplus_{m\in\nn}\bhk{\bigvee^{(m)}P}^{2}_{{\sf LCM}}$& & & & $\bigoplus_{m\in\nn}\bhk{\bigvee^{(m)}\bhk{P\vee P}^{3}_{{\sf LCM}}}^{2}_{{\sf CL}}\;\geq\;\bigcup_{m\in\nn}\bhk{\bigvee^{(m)}P}^{2}_{{\sf CL}}$\\
 &\rotatebox[origin=r]{-30}{$\geq$}&$\bhk{P\vee P}^{3}_{{\sf LCM}}$&\rotatebox{30}{$\geq$}& 
\end{tabular}
\end{center}

Every such operation induces a functor from the associated Medvedev/Muchnik-like degree structure to the Medvedev degree structure.
The main result is that every such functor is left adjoint to the canonical map from the Medvedev degree structure onto the associated degree structure.

In section $6$, we will see that how our classes of nonuniformly computable functions relate to the arithmetical hierarchy of non-intuitionistic principles such as {\em the law of excluded middle} ({\sf LEM}), {\em the lessor limited principle of omniscience} or {\em de Morgan's law} ({\sf LLPO}), and {\em the double negation elimination} ({\sf DNE}).
The arithmetical hierarchy of non-intuitionistic principles is illustrated as follows:

\begin{center}\small
\begin{tabular}{ccccccccccc}
& & & & &\rotatebox[origin=r]{30}{---}&${\Sigma^0_2\text{\sf -LLPO}}$&\rotatebox{-30}{---}& & & \\
${\sf HA}$&---&${\Sigma^0_1\text{\sf -LEM}}$&---&${\Delta^0_2\text{\sf -LEM}}$& & & & ${\Sigma^0_2\text{\sf -LEM}}$ &---&${\Sigma^0_3\text{\sf -DNE}}$\\
& & & & &\rotatebox[origin=r]{-30}{---}&${\Sigma^0_2\text{\sf -DNE}}$&\rotatebox{30}{---}& & & 
\end{tabular}
\end{center}

Here, $\Gamma$-{\sf LEM} represents the sentence $\varphi\vee\neg\varphi$ for $\Gamma$-sentences $\varphi$; $\Gamma$-{\sf LLPO} represents the sentence $\neg(\varphi\wedge\psi)\rightarrow\neg\varphi\vee\neg\psi$ for $\Gamma$-sentences $\varphi,\psi$; and $\Gamma$-{\sf DNE} represents the sentence $\neg\neg\varphi\rightarrow\varphi$ for $\Gamma$-sentences $\varphi$.
We interpret these principles as partial multi-valued functions on $\nn^\nn$, and then we characterize our notions of nonuniform computability by using these principles in the context of the Weihrauch degrees.
We also characterize our notions by Wadge-like games.

\subsection{Notations and Conventions}\label{subsec:1:notation}

For any sets $X$ and $Y$, we say that {\em $f$ is a function from $X$ to $Y$} (written $f:X\to Y$) if the domain ${\rm dom}(f)$ of $f$ includes $X$, and the range ${\rm range}(f)$ of $f$ is included in $Y$.
We also use the notation $f:\subseteq X\to Y$ to denote that $f$ is a partial function from $X$ to $Y$, i.e., the domain ${\rm dom}(f)$ of $f$ is included in $X$, and the range ${\rm rng}(f)$ of $f$ is also included in $Y$.

\index{$\sigma^-$}\index{tree!immediate predecessor in}\index{$[\sigma]$}\index{tree}\index{$[T]$}%
\index{tree!extendible in}\index{$T^{ext}$}\index{tree!leaf of}\index{tree!dead end of}%
For basic terminology in Computability Theory, see Soare \cite{Soa}.
For $\sigma\in\nn^{<\nn}$, we let $|\sigma|$ denote the length of $\sigma$.
For $\sigma\in\nn^{<\nn}$ and $f\in\nn^{<\nn}\cup\nn^\nn$, we say that $\sigma$ is {\em an initial segment} of $f$ (denoted by $\sigma\subset f$) if $\sigma(n)=f(n)$ for each $n<|\sigma|$.
Moreover, $f\res n$ denotes the unique initial segment of $f$ of length $n$.
let $\sigma^-$ denote an immediate predecessor node of $\sigma$, i.e. $\sigma^-=\sigma\res (|\sigma|-1)$.
We also define $[\sigma]=\{f\in \nn^\nn:f\supset\sigma\}$.
A {\em tree} is a subset of $\nn^{<\nn}$ closed under taking initial segments.
For any tree $T\subseteq \nn^{<\nn}$, we also let $[T]$ be the set of all infinite paths of $T$, i.e., $f$ belongs to $[T]$ if $f\res n$ belongs to $T$ for each $n\in\nn$. 
A node $\sigma\in T$ is {\em extendible} if $[T]\cap[\sigma]\not=\emptyset$.
Let $T^{ext}$ denote the set of all extendible nodes of $T$.
We say that $\sigma\in T$ is {\em a leaf} or {\em a dead end} if there is no $\tau\in T$ with $\tau\supsetneq\sigma$.

\index{concatenation}\index{$\fr$}\index{$\bigsqcap$}%
For any set $X$, the tree $X^{<\nn}$ of finite words on $X$ forms a monoid under concatenation $\fr$.
Here {\em the concatenation of $\sigma$ and $\tau$} is defined by $(\sigma\fr \tau)(n)=\sigma(n)$ for $n<|\sigma|$ and $(\sigma\fr\tau)(|\sigma|+n)=\tau(n)$ for $n<|\tau|$.
We use symbols $\fr$ and $\bigsqcap$ for the operation on this monoid, where $\bigsqcap_{i\leq n}\sigma_i$ denotes $\sigma_0\fr\sigma_1\fr\dots\fr\sigma_n$.
To avoid confusion, the symbols $\times$ and $\prod$ are only used for a product of sets.
We often consider the following three left monoid actions of $X^{<\nn}$:
The first one is the set $X^\nn$ of infinite words on $X$ with an operation $\fr:X^{<\nn}\times X^\nn\to X^\nn$; $(\sigma\fr f)(n)=\sigma(n)$ for $n<|\sigma|$ and $(\sigma\fr f)(|\sigma|+n)=f(n)$ for $n\in\nn$.
The second one is the set $\mathcal{T}(X)$ of subtrees $T\subseteq X^{<\nn}$ with an operation $\fr:X^{<\nn}\times\mathcal{T}(X)\to\mathcal{T}(X)$; $\sigma\fr T=\{\sigma\fr\tau:\tau\in T\}$.
The third one is the power set $\mathcal{P}(X^\nn)$ of $X^{\nn}$ with an operation $\fr:X^{<\nn}\times\mathcal{P}(X^\nn)\to\mathcal{P}(X^\nn)$; $\sigma\fr P=\{\sigma\fr f:f\in P\}$.

\index{Pi01 set@$\Pi^0_1$ set}\index{Pi01 set@$\Pi^0_1$ set!effective enumeration of@effective enumeration of}\index{Pi01 set@$\Pi^0_1$ set!corresponding tree of@corresponding tree of}\index{$\Phi(\sigma;n)$}%
\index{Pi01 set@$\Pi^0_1$ set!computable sequence of}\index{Pi01 set@$\Pi^0_1$ set!uniform sequence of}%
\index{Pi01 set@$\Pi^0_1$ set!special@special}\index{$f\oplus g$}\index{$P\linf Q$}\index{$P\lsup Q$}%
We say that a set $P\subseteq\nn^\nn$ is $\Pi^0_1$ if there is a computable relation $R$ such that $P=\{f\in\nn^\nn:(\forall n)R(n,f)\}$ holds.
Equivalently, $P=[T_P]$ for some computable tree $T_P\subseteq\nn^{<\nn}$.
Let $\{\Phi_e\}_{e\in\nn}$ be an effective enumeration of all Turing functionals (all partial computable functions\footnote{In some context, a function $\Phi$ is sometimes called partial computable if it can be extended to some $\Phi_e$. In this paper, however, we do not need to distinguish our definition as being different from this definition.}) on $\nn^\nn$.
Then the $e$-th $\Pi^0_1$ subset of $2^\nn$ is defined by $P_e=\{f\in 2^\nn:\Phi_e(f;0)\uparrow\}$.
Note that $\{P_e\}_{e\in\nn}$ is an effective enumeration of all $\Pi^0_1$ subsets of Cantor space $2^\nn$.
If (an index $e$ of) a $\Pi^0_1$ set $P_e\subseteq 2^\nn$ is given, then $T_e=\{\sigma\in 2^{<\nn}:\Phi_e(\sigma;0)\uparrow\}$ is called {\em the corresponding tree for $P_e$}.
Here $\Phi(\sigma;n)$ for $\sigma\in\nn^{<\nn}$ and $n\in\nn$ denotes the computation of $\Phi$ with an oracle $\sigma$, an input $n$, and step $|\sigma|$.
Whenever a $\Pi^0_1$ set $P$ is given, we assume that an index $e$ of $P$ is also given.
If $P\subseteq 2^\nn$ is $\Pi^0_1$, then the corresponding tree $T_P\subseteq 2^{<\nn}$ of $P$ is computable, and $[T_P]=P$.
Moreover, the set $L_P$ of all leaves of the computable tree $T_P$ is also computable.
We also say that a sequence of $\{P_i\}_{i\in I}$ of $\Pi^0_1$ subsets of a space $X$ is {\em computable} or {\em uniform} if the set $\{(i,f)\in I\times X:f\in P_i\}$ is again a $\Pi^0_1$ subset of the product space $I\times X$.
A set $P\subseteq\nn^\nn$ is {\em special} if $P$ is nonempty and $P$ has no computable member.
For $f,g\in\nn^\nn$, $f\oplus g$ is defined by $(f\oplus g)(2n)=f(n)$ and $(f\oplus g)(2n+1)=g(n)$ for each $n\in\nn$.
For $P,Q\subseteq\nn^\nn$, put $P\linf Q=(\lrangle{0}\fr P)\cup(\lrangle{1}\fr Q)$ and $P\lsup Q=\{f\oplus g:f\in P\;\&\;g\in Q\}$.

%% file: NRMP_fullproof1-2a.tex
\section{Nonuniformly Computable Discontinuous Functions}

\subsection{Piecewise Computable Functions}

Our main objective in the paper is to study the intermediate notion of {\em (uniform) computability} and {\em nonuniform computability}.
The concept of nonuniform computability can be rephrased as {\em countable computability}, i.e., partial functions that are decomposable into countably many computable functions.
One can expect that the class of nonuniformly computable functions is classified on the basis of the least cardinality and least complexity of the decomposition (see also Pauly \cite{Pauly10}).
For instance, if a partial function $\Gamma:\subseteq\nn^\nn\to\nn^\nn$ is decomposable into $k$ many computable functions, we say that it is {\em $k$-wise computable} or {\em $(k,1)$-computable}, and if $\Gamma$ is decomposable into countably many (finitely many, resp.) computable functions with uniformly $\Lambda$-definable domains, we say that it is {\em countable (finite, resp.) $\Lambda$-piecewise computable}, where $\Lambda$ is a lightface pointclass.

An important subclass of the piecewise computable functions consists of partial functions that are identifiable in the limit (\cite{Gol}).
The relationship between the computability with {\em trial-and-error} (limit computability or effective learnability) and the subhierarchy of the level $\Delta^0_2$ has been common knowledge among recursion theorists since the last fifty years or so (see also Shoenfield \cite{Shoen59}, Gold \cite{Gol}, Putnam \cite{Putnam65}, and Ershov \cite{Ershov68}).
A basic observation (see Theorem \ref{thm:5:red-eq-dis}) regarding the concept of type-two learnability (see also de Brecht-Yamamoto \cite{Brecht-Yamamoto,Bre-Yama}) is that a partial function on $\nn^\nn$ is $\Pi^0_1$-piecewise computable if and only if it is identifiable in the limit or learnable in the following sense:
a partial function $\Gamma:\subseteq\nn^\nn\to\nn^\nn$ will be called {\em learnable} or {\em $(1,\omega)$-computable} if there is a computable function $\Psi:\subseteq\nn^{<\nn}\to\nn$ such that $\Phi_{\lim_{n\to\infty}\Psi(f\res n)}(f)=\Gamma(f)$ for every $f\in{\rm dom}(\Gamma)$, where recall that $\{\Phi_e\}_{e\in\nn}$ is a fixed enumeration of all partial computable functions.
Such a $\Psi$ is called a {\em learner}.
\index{learner}%

\index{learner!dominate}%
We say that partial function $\widehat{\Psi}:\subseteq\nn^{<\nn}\to\nn$ dominates $\Psi:\subseteq\nn^{<\nn}\to\nn$ as a learner if $\lim_s\widehat{\Psi}(f\res s)$ converges to $\lim_s\Psi(f\res s)$ whenever $\lim_s\Psi(f\res s)$ converges.
We say that $\{\Psi_e\}_{e\in\nn}$ enumerates all learners if every partial function $\Psi:\subseteq\nn^{<\nn}\to\nn$ is dominated by some $\Psi_e$ as a learner.
To get a nice enumeration of all learners, we first check the following proposition.

\begin{prop}\label{prop:1-2:learn-trick}
There is an effective enumeration $\{\Psi_e\}_{e\in\nn}$ of all learners that consists of total functions $\Psi_e:\nn^{<\nn}\to\nn$.
\end{prop}

\begin{proof}\upshape
For the $e$-th partial computable function $\Phi_e:\subseteq\nn^{<\nn}\to\nn$ and an index $k$, we effectively define a total computable function $\Psi_{\lrangle{e,k}}:\nn^{<\nn}\to\nn$ that dominates $\Phi_e$ as a learner.
We define $\Phi$ by $\Phi(\lrangle{})=k$ and $\Phi(\sigma)=\Phi_e(\sigma)$ for all nonempty string $\sigma$. 
Given $\sigma\in\nn^{<\nn}$, put $\sigma^*=\max\{\tau\subseteq\sigma:\Phi(\tau)[|\sigma|]\downarrow\}$.
Then define $\Psi_{\lrangle{e,k}}(\sigma)=\Phi(\sigma^*)$ for every $\sigma\in\nn^{<\nn}$.
If $\lim_s\Phi_e(f\res s)$ converges then clearly $\lim_s\Psi_{\lrangle{e,k}}(f\res s)$ also converges to the same value.
Hence, $\Psi_{\lrangle{e,k}}$ dominates $\Phi_e$.
\end{proof}

The set $\{\Psi_e\}_{e\in\nn}$ in Proposition \ref{prop:1-2:learn-trick} is referred as {\em the effective enumeration of all learners}, and $\Psi_e$ is called {\em the $e$-th learner}.
\index{learner!effective enumeration of}\index{learner!$e$-th learner}%

\begin{remark}
We urge the reader not to confuse the notions $\Psi(\sigma)$ and $\Phi(\sigma)$ for a learner $\Psi$ and a computable function $\Phi$ (on $\nn^\nn$).
In the former case, $\Psi(\sigma)$ simply denotes the output (the inference) of the learner $\Psi$ based on the current input $\sigma$.
In the latter case, however, we use $\sigma$ as an initial segment of some oracle information, and so really $\Phi(\sigma)$ denotes a string $\lrangle{\Phi(\sigma;0),\Phi(\sigma;1),\Phi(\sigma;2),\dots}$.
\index{$\Phi(\sigma)$}%
\end{remark}

\begin{notation}
Let $\Psi:\nn^{<\nn}\to\nn$ be a learner.
For any string $\sigma\in\nn^{<\nn}$, the set of {\em mind-change locations of the learner $\Psi$ on the informant $\sigma$} (denoted by ${\tt mcl}_\Psi(\sigma)$) is defined by
\index{learner!mind-change location of}\index{${\tt mcl}_\Psi$}%
\[{\tt mcl}_\Psi(\sigma)=\{n<|\sigma|:\Psi(\sigma\res n+1)\not=\Psi(\sigma\res n)\}.\]
We also define ${\tt mcl}_\Psi(f)=\bigcup_{n\in\nn}{\tt mcl}_\Psi(f\res n)$ for any $f\in\nn^{\nn}$.
Then, $\#{\tt mcl}_\Psi(f)$ denotes the {\em number of times that the learner $\Psi$ changes her/his mind on the informant $f$}.
\index{$\#{\tt mcl}_\Psi$}%
Moreover, the set of {\em indices predicted by the learner $\Psi$ on the informant $\sigma$} (denoted by ${\tt indx}_\Psi(\sigma)$) is defined by
\[{\tt indx}_\Psi(\sigma)=\{\Psi(\sigma\res n):n\leq|\sigma|\}.\]
\index{learner!indices predicted by}\index{${\tt indx}_\Psi(\sigma)$}%
We also define ${\tt indx}_\Psi(f)=\bigcup_{n\in\nn}{\tt indx}_\Psi(f\res n)$ for any $f\in\nn^{\nn}$.
\index{${\tt indx}_\Psi$}%
\end{notation}

We now introduce various subclasses of nonuniformly computable functions on $\nn^\nn$ based on Learning Theory.

\begin{definition}\label{def:1-2:nonunif_bas}
Let $D$ be a subset of Baire space $\nn^\nn$, and $\alpha,\beta,\gamma\leq\omega$ be ordinals.
\index{computable!$(\alpha,\beta\mid\gamma)$-computable}%
A function $\Gamma:D\to\nn^\nn$ is {\em $(\alpha,\beta|\gamma)$-computable} if there is a set $I\subseteq\nn$ of cardinality $\alpha$ such that, for any $g\in D$, there is an index $e\in I$ satisfying the following three conditions.
\begin{enumerate}
\item (Learnability)
$\lim_n\Psi_e(g\res n)$ converges, and $\Phi_{\lim_n\Psi_e(g\res n)}(g)=\Gamma(g)$.
\item 
\index{computable!$(\alpha,\beta\mid\gamma)$-computable!mind-change condition of}%
(Mind-Change Condition) $\#{\tt mcl}_{\Psi_e}(g)=\#\{n\in\nn:\Psi_e(g\res n+1)\not=\Psi_e(g\res n)\}<\beta$.
\item 
\index{computable!$(\alpha,\beta\mid\gamma)$-computable!error condition of}%
(Error Condition) $\#{\tt indx}_{\Psi_e}(g)=\#\{\Psi_e(g\res n):n\in\nn\}\leq\gamma$.
\end{enumerate}
\end{definition}

If $\gamma=\omega$, then we simply say that $\Gamma$ is {\em $(\alpha,\beta)$-computable} for $(\alpha,\beta|\gamma)$-computable function $\Gamma$.
\index{computable!$(\alpha,\beta)$-computable}%
Let $[\mathfrak{C}_T]^\alpha_{\beta}$ (resp.\ $[\mathfrak{C}_T]^\alpha_{\beta|\gamma}$) denote the set of all $(\alpha,\beta)$-computable (resp.\ $(\alpha,\beta|\gamma)$-computable) functions.
\index{$[\mathfrak{C}_T]^\alpha_{\beta}$}\index{$[\mathfrak{C}_T]^\alpha_{\beta\mid\gamma}$}%
Hereafter, the symbol $<\omega$ will be used in referring to ``some natural number $n$''.
For instance, $\Gamma$ is said to be $(<\omega,2|<\omega)$-computable if there are $a,c\in\nn$ such that it is $(a,2|c)$-computable.

\begin{remark}
Some of $(\alpha,\beta|\gamma)$-computability notions are related to learnability notions:
\index{learnable!with bounded mind-changes}\index{learnable!with bounded errors}%
\index{learnable}\index{computable!$k$-wise}\index{learnable!by a team}%
Every $(1,<\omega)$-computable function is {\em learnable with bounded mind-changes}; every $(1,\omega|<\omega)$-computable function is {\em learnable with bounded errors}; every $(1,\omega)$-computable function is {\em learnable}; every $(<\omega,1)$-computable function is {\em $k$-wise computable}; and every $(<\omega,\omega)$-computable function is {\em team-learnable}.
The concept of learnability in the context of real number computation has been studied by several researchers including Chadzelek-Hotz \cite{ChHo}, Ziegler \cite{Zie2,Zie3}, and de Brecht-Yamamoto \cite{Brecht-Yamamoto, Bre-Yama}.
The notion of mind-change is also related to the level of discontinuity studied by several researchers, for instance, Hertling \cite{Hertling96}, and Hemmerling \cite{Hemmerling08}.
See also Section \ref{section:gdd_wellfounded} for more information on the relationship between the notion of mind-changes and the level of discontinuity.
The notion of $k$-wise computability has been also studied by, for example, Pauly \cite{Pauly10} and Ziegler \cite{Zie4}.
\end{remark}

\begin{table}\caption{Seven Classes of Nonuniformly Computable Functions}\label{rtable}%
\begin{center}
\begin{tabular}{cc}\toprule
$[\mathfrak{C}_T]^1_1$ & (Uniformly) computable \\
$[\mathfrak{C}_T]^1_{<\omega}$ & Learnable with bounded mind changes \\
$[\mathfrak{C}_T]^1_{\omega|<\omega}$ & Learnable with bounded errors \\
$[\mathfrak{C}_T]^1_\omega$ & Learnable \\
$[\mathfrak{C}_T]^{<\omega}_1$ & $k$-wise computable for some $k\in\omega$\\
$[\mathfrak{C}_T]^{<\omega}_\omega$ & Learnable by a team \\
$[\mathfrak{C}_T]^\omega_1$ & Nonuniformly computable \\
\bottomrule
\end{tabular}
\end{center}
\end{table}

\begin{table}\label{table}\small
\begin{center}
\begin{tabular}{ccccccccccc}
 & & & & &\rotatebox[origin=r]{30}{$\subseteq$}&$[\mathfrak{C}_T]^{<\omega}_1=[\mathfrak{C}_T]^{<\omega}_{\omega|<\omega}$&\rotatebox{-30}{$\subseteq$}& & & \\
$[\mathfrak{C}_T]^1_1$&$\subseteq$&$[\mathfrak{C}_T]^1_{<\omega}$&$\subseteq$&$[\mathfrak{C}_T]^1_{\omega|<\omega}$& & & & $[\mathfrak{C}_T]^{<\omega}_{\omega}$ &$\subseteq$&$[\mathfrak{C}_T]^\omega_1=[\mathfrak{C}_T]^\omega_\omega$\\
 & & & & &\rotatebox[origin=r]{-30}{$\subseteq$}&$[\mathfrak{C}_T]^1_{\omega}$&\rotatebox{30}{$\subseteq$}& & & 
\end{tabular}
\end{center}
\caption{Seven monoids of nonuniformly computable functions}%
\end{table}

We first mention the topological interpretation of the learnability.
For a sequence $\{\sigma_n\}_{n\in\nn}\in(\nn^{<\nn})^{\nn}$ of strings, $\lim_n\sigma_n$ is defined by $(\lim_n\sigma_n)(m)=\lim_n(\sigma_n(m))$.
\index{$\lim_n\sigma_n$}%
If $\lim_n\sigma_n:\nn\to\nn$ is total, say $\lim_n\sigma_n=h\in\nn^\nn$, then we say that {\em $\lim_n\sigma_n\in\nn^\nn$ converges to $h$}.

\begin{prop}\label{prop:1-2:characterization}
Fix an ordinal $\alpha\leq\omega$.
A partial function $\Gamma:\subseteq\nn^\nn\to\nn^\nn$ is $(1,\alpha)$-computable if and only if there is a total computable function $\psi:\nn^{<\nn}\to\nn^{<\nn}$ such that $\lim_n\psi(g\res n)$ converges to $\Gamma(g)$, and $\#\{n\in\nn:\psi(g\res n+1)\not\supseteq\psi(g\res n)\}<\alpha$, for any $g\in{\rm dom}(\Gamma)$.
\end{prop}

\begin{proof}\upshape
Assume that $\Gamma$ is $(1,\alpha)$-computable via a learner $\Psi$.
We put $\psi(\sigma)=\Phi_{\Psi(\sigma)}(\sigma)$ for each $\sigma\in\nn^{<\nn}$.
Then the condition $\#{\tt mcl}_\Psi(g)<\alpha$ implies $\#\{n\in\nn:\psi(g\res n+1)\not\supseteq\psi(g\res n)\}<\alpha$, for any $g\in{\rm dom}(\Gamma)$.
Because if $\Psi(g\res n+1)=\Psi(g\res n)$, then $\psi(g\res n)=\Phi_{\Psi(g\res n)}(g\res n)\subseteq\Phi_{\Psi(g\res n)}(g\res n+1)=\psi(g\res n+1)$.
Thus, clearly, $\lim_n\psi(g\res n)$ converges to $\Phi_{\lim_n\Psi(g\res n)}(g)=\Gamma(g)$.

Assume that $\Gamma(g)=\lim_n\psi(g\res n)$ for any $g\in{\rm dom}(\Gamma)$ for some $\psi$ satisfying the condition in Proposition \ref{prop:1-2:characterization}.
We define a computable function $\Phi_{e(\sigma)}:\nn^\nn\to\nn^\nn$ for each $\sigma\in\nn^{<\nn}$.
For any $g\in\nn^\nn$, put $\Phi_{e(\sigma)}(g;n)=\psi(g\res s)(n)$ for each $n\in\nn$, where $s\geq|\sigma|$ is the least number such that $\psi(g\res s)(n)$ is defined.
Clearly, $\Phi_{e(\sigma)}$ is partial computable, and indeed, we can compute an index $e(\sigma)$ of $\Phi_{e(\sigma)}$ uniformly in $\sigma\in\nn^{<\nn}$.
Then, we define a learner $\Psi$ inductively.
Put $\Psi(\lrangle{})=e(\lrangle{})$.
Fix $\sigma\in\nn^{<\nn}$, and assume that $\Psi(\sigma^-)$ has already been defined.
If $\psi(\sigma)\supseteq\psi(\sigma^-)$, then set $\Psi(\sigma)=\Psi(\sigma^-)$.
If $\psi(\sigma)\not\supseteq\psi(\sigma^-)$, then set $\Psi(\sigma)=e(\sigma)$.
Clearly, the condition $\#\{n\in\nn:\psi(g\res n+1)\not\supseteq\psi(g\res n)\}<\alpha$ implies $\#{\tt mcl}_\Psi(g)<\alpha$, for any $g\in{\rm dom}(\Gamma)$.
In particular, $\lim_n\Psi(g\res n)$ converges to some index $e(\sigma)$ for any $g\in{\rm dom}(\Gamma)$.
Hence, $\Phi_{\lim_n\Psi(g\res n)}(g)=\bigcup_{n\geq|\sigma|}\psi(g\res n)=\lim_{n\in\nn}\psi(g\res n)=\Gamma(g)$, since $\{\psi(g\res n)\}_{n\geq|\sigma|}$ is an increasing sequence of strings.
\end{proof}

\begin{cor}[de Brecht-Yamamoto \cite{Brecht-Yamamoto}]
A partial function $\Gamma:\subseteq\nn^\nn\to\nn^\nn$ is $(1,\omega)$-computable if and only if there is a computable sequence $\{\Gamma_n\}_{n\in\nn}$ of partial computable functions which converges pointwise to $\Gamma$ on ${\rm dom}(\Gamma)$ with respect to the discrete topology on $\nn^\nn$.
\end{cor}

\begin{proof}\upshape
By Proposition \ref{prop:1-2:characterization}.
\end{proof}

\subsection{Seven Classes of Nonuniformly Computable Functions}

We first check several basic properties of $(\alpha,\beta|\gamma)$-computability to show the following theorem stating that the classes obtained from Definition \ref{def:1-2:nonunif_bas} closed under composition are exactly the classes listed in Table \ref{rtable}.

\begin{theorem}\label{thm:main:first1-2}
$\{[\mathfrak{C}_T]^{\alpha}_{\beta|\gamma}:\alpha,\beta,\gamma\in\nn\cup\{<\omega,\omega\}\}$ contains just seven monoids, $[\mathfrak{C}_T]^1_1$, $[\mathfrak{C}_T]^1_{<\omega}$, $[\mathfrak{C}_T]^1_{\omega|<\omega}$, $[\mathfrak{C}_T]^{<\omega}_1$, $[\mathfrak{C}_T]^1_\omega$, $[\mathfrak{C}_T]^{<\omega}_\omega$, and $[\mathfrak{C}_T]^\omega_1$.
\end{theorem}

\begin{prop}\label{prop:1-2:func}
Let $\Gamma$ be a partial function on Baire space $\nn^\nn$.
\begin{enumerate}
\item If $\Gamma$ is $(\alpha_0,\beta_0|\gamma_0)$-computable, $\alpha_0\leq\alpha_1$, $\beta_0\leq\beta_1$, and $\gamma_0\leq\gamma_1$, then $\Gamma$ is $(\alpha_1,\beta_1|\gamma_1)$-computable.
\item $\Gamma$ is $(\alpha,1)$-computable if and only if $\Gamma$ is $(\alpha,\beta|1)$-computable.
\item $\Gamma$ is $(\alpha,\beta)$-computable if and only if $\Gamma$ is $(\alpha,\beta|\beta)$-computable.
\item $\Gamma$ is $(1,1)$-computable if and only if $\Gamma$ is computable.
\item $\Gamma$ is $(\omega,1)$-computable if and only if $\Gamma$ is $(\omega,\omega)$-computable if and only if $\Gamma$ is nonuniformly computable, i.e., $\Gamma(g)\leq_Tg$ for any $g\in{\rm dom}(\Gamma)$, where recall that $\leq_T$ denotes the Turing reducibility. 
\end{enumerate}
\end{prop}

\begin{proof}\upshape
The items (1) and (2) easily follow from the definitions.
The item (3) follows from $\#{\tt indx}_{\Psi}(g)-1\leq\#{\tt mcl}_\Psi(g)$.

(4) If $\Gamma$ is computable via $\Phi_e$, then $\Gamma$ is $(1,1)$-computable via the singleton $\{i(e)\}$, where $\Psi_{i(e)}(\sigma)=e$ for any $\sigma\in\nn^{<\nn}$.
Assume that $\Gamma$ is $(1,1)$-computable via a singleton $\{e\}$.
Then $\Psi_e(\sigma)=\Psi_e(\lrangle{})$ for any $\sigma$ extendible to an element of ${\rm dom}(\Gamma)$, since $\#{\tt mcl}_{\Psi_e}(g)=0$ for any $g\in{\rm dom}(\Gamma)$.
Therefore, $\Gamma$ is computable via $\Phi_{\Psi_e(\lrangle{})}$.

(5) If $\Gamma$ is nonuniformly computable, then $\Gamma$ is $(\omega,1)$-computable via $\{i(e)\}_{e\in\nn}$, where $\Psi_{i(e)}(\sigma)=e$ for any $\sigma\in\nn^{<\nn}$.
Assume that $\Gamma$ is $(\omega,\omega)$-computable via $I$.
For any $g\in{\rm dom}(\Gamma)$, there is $e\in I$ such that $\lim_n\Psi_e(g\res n)$ converges to some value $p\in\nn$, and $\Phi_p(g)=\Gamma(g)$.
Thus, $\Gamma(g)\leq_Tg$ via $\Phi_p$.
\end{proof}

\begin{prop}\label{prop:1-2:func2}
For each $m,n\in\nn$, every $(m,\omega|n)$-computable function is $(m\cdot n,1)$-computable.
\end{prop}

\begin{proof}\upshape
Assume that $\Gamma:D\to\nn^\nn$ is $(m,\omega|n)$-computable with $m$-learners $\{\Psi^e\}_{e<m}$ with $n$-errors.
Now, we define an algorithm $\Phi^e_k$ for any $e<m$ and $k<n$, and we ensure the following property:
\[(\forall g\in D)(\exists e<m)(k<n)\;\Phi^e_k(g)=\Gamma(g).\]
The algorithm $\Phi^e_k$ proceeds as follows for $g$.
Recall that ${\tt indx}_{\Psi^e}(g)$ represents the set of all indices occurring in hypothesis of the learner $\Psi^e$.
We have an effective enumeration $d^e_0(g),d^e_1(g),\dots$ of all indices contained in ${\tt indx}_{\Psi^e}(g)$ uniformly in $g$.
Then, we set $\Phi^e_k(g)=\Phi_{d^e_k(g)}(g)$ if $d^e_k(g)$ is defined.
For any $g\in D$, there is $e<m$ such that $\lim_s\Psi^e(g\res s)$ converges to some correct computation $d$ of $\Gamma(g)$, i.e., $\Phi_d(g)=\Gamma(g)$.
Since $\#{\tt indx}_{\Psi^e}(g)<n$, we have $d^e_k(g)=d$ for some $k<n$.
Thus, for any $g\in D$, there are $e<m$ and $k<n$ such that $\Phi^e_k(g)=\Gamma(g)$.
Therefore, if $i^e_k$ is an index of $\Phi^e_k$ for each $e<m$ and $k<n$, then $\Gamma$ is $(m\cdot n,1)$-computable via an upper bound $\max\{i^e_k:e<m\;\&\;k<n\}$.
\end{proof}

\begin{cor}
$\Gamma$ is $(<\omega,\omega|<\omega)$-computable if and only if $\Gamma$ is $(<\omega,1)$-computable.
\end{cor}

\begin{proof}\upshape
Every $(<\omega,\omega|<\omega)$-computable function $\Gamma$ is $(m,\omega|n)$-computable for some $m,n<\omega$.
Therefore, by Proposition \ref{prop:1-2:func2}, $\Gamma$ is $(m\cdot n,1)$-computable.
In particular, $\Gamma$ is $(<\omega,1)$-computable, since $m\cdot n<\omega$.
\end{proof}

\begin{prop}\label{prop:1-3:monoid}
For each $i<2$, let $\Gamma_i$ be a partial $(\alpha_i,\beta_i|\gamma_i)$-computable function on Baire space $\nn^\nn$, where $\alpha_i,\beta_i,\gamma_i\leq\omega$ are ordinals.
Then $\Gamma_1\circ\Gamma_0$ is $(\alpha_0*\alpha_1,\beta_0*\beta_1|\gamma_0*\gamma_1)$-computable, where $*$ is the multiplication as the cardinals, or equivalently, $\kappa*\lambda=\min\{\kappa\cdot\lambda,\omega\}$ for ordinals $\kappa,\lambda\leq\omega$.
\end{prop}

\begin{proof}\upshape
For each $i<2$, since $\Gamma_i$ is $(\alpha_i,\beta_i|\gamma_i)$-computable, there is a collection of learners, $\{\Psi^i_j\}_{j<\alpha_i}$ and a cover $\{U^i_j\}_{j<\alpha_i}$ of ${\rm dom}(\Gamma_i)$ such that $\Gamma_i(f)=\Phi_{\lim_n\Psi^i_j(f\res n)}(f\res n)$ and $\#{\tt mcl}_{\Psi^i_j}(f)<\beta_i$ and $\#{\tt indx}_{\Psi^i_j}(f)<\gamma_i$, for any $j<\alpha_i$ and $f\in U^i_j$.
Fix $j<\alpha_0$ and $k<\alpha_1$.
Then $\Psi^*_{j,k}(\sigma)$ is defined as follows.
Let $J(\sigma)$ be the longest interval $[r,|\sigma|)$ satisfying $\Psi^0_j(\sigma\res r)=\Psi^0_j(\sigma)$, and define $J^+(\sigma)=J(\sigma)\setminus\{r\}$.
If $\#({\tt mcl}_{\Psi^1_k}\cap J^+(\sigma))<\beta_1$ and $\#({\tt indx}_{\Psi^1_k}\cap J(\sigma))<\gamma_1$, then put $\Psi^*_{j,k}(\sigma)=\Psi^1_k(\Phi_{\Psi^0_j(\sigma)}(\sigma))$.
Otherwise, put $\Psi^*_{j,k}(\sigma)=\Psi^*_{j,k}(\sigma^-)$.
For given $\sigma$, we compute an index $\Psi_{j,k}(\sigma)$, where $\Phi_{\Psi_{j,k}(\sigma)}(f)=\Phi_{\Psi^*_{j,k}(\sigma))}(\Phi_{\Psi^0_j(\sigma)}(f))$ for any $f$.

Note that $f\in{\rm dom}(\Gamma_1\circ\Gamma_0)$ if and only if $f\in{\rm dom}(\Gamma_0)$ and $\Gamma_0(f)\in{\rm dom}(\Gamma_1)$.
Therefore, for such $f$, there are $j<\alpha_0$ and $k<\alpha_1$ such that $f\in U^0_j$ and $\Gamma_0(f)\in U^1_k$.
Assume that $f\in{\rm dom}(\Gamma_1\circ\Gamma_0)\cap U^0_j$ and $\Gamma_0(f)\in U^1_k$.
It is easy to see that $\Psi^*_{j,k}$ is computable, $\#{\tt mlc}_{\Psi^*_{j,k}}(f)<\beta_0*\beta_1$ and $\#{\tt indx}_{\Psi^*_{j,k}}(f)<\gamma_0*\gamma_1$.
Moreover, there exist $s$ and $e_0$ such that $\Psi^0_j(f\res t)=\Psi^0_j(f\res s)=e_0$ for any $t\geq s$.
Fix such $s$.
Since $\Phi_{e_0}(f)=\Gamma_0(f)\in U^1_k$, for any $t\geq s$, $\#({\tt mcl}_{\Psi^1_k}\cap J^+(f\res t))<\beta_1$ and $\#({\tt indx}_{\Psi^1_k}\cap J(f\res t))<\gamma_1$, since $J(f\res t)=J(f\res s)$ and by our choice of $\Psi^1_k$.
Therefore, $\lim_n\Psi^*_{j,k}(f\res n)$ converges to $\lim_n\Psi^1_k(\Gamma_0(f\res n))$.
However, there exist $u\geq s$ and $e_1$ such that $\Psi^1_k(\Gamma_0(f\res v))=\Psi^1_k(\Gamma_0(f\res u))=e_1$ for any $v\geq u$, since $\{\Gamma_0(f\res u)\}_{u\geq s}$ is an increasing sequence of strings and $\Gamma_0(f)\in{\rm dom}(\Gamma_1)$.
Here $\Phi_{e_1}(\Gamma_0(f))=\Gamma_1(\Gamma_0(f))$.
Thus, 
\[\Phi_{\lim_n\Psi_{j,k}(f\res n)}(f)=\Phi_{\lim_n\Psi^*_{j,k}(f\res n)}(\Phi_{\lim_n\Psi^0_j(f\res n)}(f))=\Phi_{\lim_n\Psi^1_k(\Gamma_0(f)\res n)}(\Gamma_0(f))=\Gamma_1(\Gamma_0(f)).\]
Consequently, $\Gamma_1\circ\Gamma_0$ is $(\alpha_0*\alpha_1,\beta_0*\beta_1|\gamma_0*\gamma_1)$-computable, via $\{\Psi_{j,k}\}_{j<\alpha_0,k<\alpha_1}$.
\end{proof}

\begin{cor}\label{cor:1-3:monoid}
$[\mathfrak{C}_T]^\alpha_{\beta|\gamma}$ forms a monoid under composition, for any $\alpha,\beta,\gamma\in\{1,<\omega,\omega\}$.
\end{cor}

\begin{proof}\upshape
Straightforward from Proposition \ref{prop:1-3:monoid}.
\end{proof}

\begin{prop}\label{prop:1-4:monoid2}
$[\mathfrak{C}_T]^1_{<\omega}$ is the smallest monoid including $[\mathfrak{C}_T]^1_2$;
$[\mathfrak{C}_T]^1_{\omega|<\omega}$ is the smallest monoid including $[\mathfrak{C}_T]^1_{\omega|2}$.
$[\mathfrak{C}_T]^{<\omega}_1$ is the smallest monoid including $[\mathfrak{C}_T]^2_1$;
$[\mathfrak{C}_T]^{<\omega}_\omega$ is the smallest monoid including $[\mathfrak{C}_T]^2_\omega$.
\end{prop}

\begin{proof}\upshape
The first result is known, and indeed, it has also been proved in Mylatz's PhD thesis, but we also give a proof here for the sake of completeness.
We first show that every $(1,n+1)$-computable function $\Gamma$ can be represented as $\Gamma=\Gamma_1\circ\Gamma_0$ for some $(1,n)$-computable function $\Gamma_0$ and $(1,2)$-computable function $\Gamma_1$.
Let $\Psi$ be a learner for $\Gamma$.
We define a learner $\Psi_0$ for $\Gamma_0$ and a learner $\Psi_1$ for $\Gamma_1$.
For a given string $\sigma\in\nn^{<\nn}$, let $\sigma^*\subseteq\sigma$ be the longest initial segment of $\sigma$ satisfying $\#{\tt mcl}_\Psi(\sigma^*)<n$.
Then, on $\sigma$, the learner $\Psi_0$ guesses an index of the partial computable function $g\mapsto g\oplus\Phi_{\Psi(\sigma^*)}(g)$, i.e., $\Gamma_0(g)=\Phi_{\Psi_0(\sigma)}(g)=g\oplus\Phi_{\Psi(\sigma^*)}(g)$ for any $g\in\nn^\nn$.
Note that $\#{\tt mcl}_{\Psi_0}(g)<n$ for any $g\in\nn^\nn$.
Therefore, $\Gamma_0$ is $(1,n)$-computable.
For $\sigma\oplus\tau\in\nn^\nn$, if $\sigma^*=\sigma$ then the learner $\Psi_1$ guesses an index of the partial computable function $g\oplus h\mapsto h$.
If $\sigma^*\not=\sigma$, then $\Psi_1$ guesses an index of the partial computable function $g\oplus h\mapsto\Phi_{\Psi(\sigma)}(g)$, i.e., $\Phi_{\Psi_1(\sigma\oplus\tau)}(g\oplus h)=\Phi_{\Psi(\sigma)}(g)$.
Since $\Gamma$ is $(1,n+1)$-computable, and by the definition of $\sigma^*$, it is easy to see that $\Gamma_1$ is $(1,2)$-computable.
For $g\in\nn^\nn$, if $\#{\tt mcl}_\Psi(g)<n$, then 
\[\Gamma_1(\Gamma_0(g))=\Gamma_1(g\oplus\Gamma(g))=\Gamma(g).\]
If $\#{\tt mcl}_\Psi(g)=n$, then
\[\Gamma_1(\Gamma_0(g))=\Gamma_1(g\oplus\Phi_{\Psi(g^*)}(g))=\Gamma(g).\]
Consequently, $\Gamma=\Gamma_1\circ\Gamma_0$ as desired.

We next show that every $(1,\omega|n+1)$-computable function $\Gamma$ can be represented as $\Gamma=\Gamma_1\circ\Gamma_0$ for some $(1,\omega|n)$-computable function $\Gamma_0$ and $(1,\omega|2)$-computable function $\Gamma_1$.
Assume that $\Psi$ is a learner for $\Gamma$, and we enumerate $\#{\tt indx}_\Psi(\sigma)$ as $\{i^\sigma_m\}_{m\leq |\sigma|}$.
Here, if $m<n$ then $\Psi$ guesses $i^\sigma_m$ before $\Psi$ guesses $i^\sigma_n$ on some initial segment of $\sigma$.
Note that, if $\sigma\subseteq\tau$ and $i^\sigma_m$ is defined, then $i^\sigma_m=i^\tau_m$.
On $\sigma\in\nn^{<\nn}$, if $\Psi(\sigma)\not=i^\sigma_n$, then $\Psi_0$ guesses an index of the partial computable function $g\mapsto g\oplus\Phi_{\Psi(\sigma)}(g)$.
Otherwise, $\Psi_0$ guesses an index of the partial computable function $g\mapsto g\oplus\Phi_{i^\sigma_0}(g)$.
Then, the partial function $\Gamma_0$ identified by the learner $\Psi_0$ is $(1,\omega|n)$-computable.
On $\sigma\oplus\tau\in\nn^{<\nn}$ if $\Psi(\sigma)\not=i^\sigma_n$, then $\Psi_1$ guesses an index of the partial computable function $g\oplus h\mapsto h$.
Otherwise, $\Psi_1$ guesses an index of partial computable function $g\oplus h\mapsto\Phi_{\Psi(\sigma)}(g)$.

We show that every $(n+1,1)$-computable function $\Gamma$ can be represented as $\Gamma=\Gamma_1\circ\Gamma_0$ for some $(n,1)$-computable function $\Gamma_0$ and $(2,1)$-computable function $\Gamma_1$.
Assume that $\Gamma$ is $(n+1,1)$-computable via a collection $\{\Delta_i\}_{i\leq n}$ of partial computable functions.
For $g\in\nn^\nn$, if $\Gamma(g)=\Delta_i(g)$ for some $i<n$, then $\Gamma_0(g)=g\oplus\Delta_i(g)$.
Otherwise, we set $\Gamma_0(g)=g\oplus\Delta_0(g)$.
Then, clearly $\Gamma_0$ is $(n,1)$-computable via $\{\lambda g.g\oplus\Delta_i(g)\}_{i<n}$.
For $g\oplus h\in\nn^\nn$, if $\Gamma(g)=\Delta_i(g)$ for some $i<n$, then $\Gamma_1(g\oplus h)=h$.
Otherwise, we set $\Gamma_1(g\oplus h)=\Delta_n(g)$.
Clearly, $\Gamma_1$ is $(2,1)$-computable.
Note that, if $g\in{\rm dom}(\Gamma)$, then $\Gamma(g)=\Delta_i(g)$ for some $i\leq n$.
If $\Gamma(g)=\Delta_i(g)$ for some $i<n$, then $\Gamma_1(\Gamma_0(g))=\Gamma_1(g\oplus\Delta_i(g))=\Delta_i(g)$.
If $\Gamma(g)=\Delta_n(g)$, then $\Gamma_1(\Gamma_0(g))=\Gamma_1(g\oplus\Delta_0(g))=\Delta_n(g)$.
Therefore, $\Gamma(g)=\Gamma_1\circ\Gamma_0(g)$ for any $g\in{\rm dom}(\Gamma)$.
By the similar way, it is easy to see that every $(n+1,\omega)$-computable function $\Gamma$ can be represented as $\Gamma=\Gamma_1\circ\Gamma_0$ for some $(n,\omega)$-computable function $\Gamma_0$ and $(2,\omega)$-computable function $\Gamma_1$.
\end{proof}

\begin{proof}[Proof of Theorem \ref{thm:main:first1-2}]\upshape
By Proposition \ref{prop:1-2:func}, we have $[\mathfrak{C}_T]^1_{1|1}=[\mathfrak{C}_T]^1_{1|<\omega}=[\mathfrak{C}_T]^1_{1|\omega}=[\mathfrak{C}_T]^1_{<\omega|1}=[\mathfrak{C}_T]^1_{\omega|1}$; $[\mathfrak{C}_T]^{1}_{<\omega|<\omega}=[\mathfrak{C}_T]^1_{1|<\omega}$; and $[\mathfrak{C}_T]^\omega_{1|1}=[\mathfrak{C}_T]^\omega_{\beta|\gamma}$ for any $\beta,\gamma\in\{1,<\omega,\omega\}$.
Moreover, by Proposition \ref{prop:1-2:func} and Proposition \ref{prop:1-2:func2}, $[\mathfrak{C}_T]^{<\omega}_{1|1}=[\mathfrak{C}_T]^{<\omega}_{\beta|\gamma}$ whenever $\lrangle{\beta,\gamma}\not=\lrangle{\omega,\omega}$.
Therefore, by Proposition \ref{prop:1-3:monoid} and \ref{prop:1-4:monoid2}, we have just seven monoids, $[\mathfrak{C}_T]^1_1$, $[\mathfrak{C}_T]^1_{<\omega}$, $[\mathfrak{C}_T]^1_{\omega|<\omega}$, $[\mathfrak{C}_T]^{<\omega}_1$, $[\mathfrak{C}_T]^1_\omega$, $[\mathfrak{C}_T]^{<\omega}_\omega$, and $[\mathfrak{C}_T]^\omega_1$.
\end{proof}

\subsection{Degree Structures and Brouwer Algebras}
\label{sec:1-2a:deg_brow}

We will see some intuitionistic feature of our classes of nonuniformly computable functions.

\begin{definition}\label{def:1-2a:degree}
\index{$P\leq_\mathcal{F}Q$}%
Let $\mathcal{F}$ be a monoid consisting of partial functions $\Gamma:\subseteq\nn^\nn\to\nn^\nn$ under composition.
Then, $\mathcal{P}(\nn^\nn)$ is preordered by the relation $P\leq_\mathcal{F}Q$ indicating the existence of a function $\Gamma\in\mathcal{F}$ from $Q$ into $P$, that is, $P\leq_\mathcal{F}Q$ if and only if there is a partial function $\Gamma:\subseteq\nn^\nn\to\nn^\nn$ such that $\Gamma\in\mathcal{F}$ and $\Gamma(g)\in P$ for every $g\in Q$.
\index{$\mathcal{D}/\mathcal{F}$}\index{$\mathcal{P}/\mathcal{F}$}%
Let $\mathcal{D}/\mathcal{F}$ and $\mathcal{P}/\mathcal{F}$ denote the quotient sets $\mathcal{P}(\nn^\nn)/\equiv_\mathcal{F}$ and $\Pi^0_1(2^\nn)/\equiv_\mathcal{F}$, respectively.
Here, $\Pi^0_1(2^\nn)$ denotes the set of all nonempty $\Pi^0_1$ subsets of $2^\nn$.
For $P\in\mathcal{P}(\nn^\nn)$, the equivalence class $\{Q\subseteq\nn^\nn:Q\equiv_\mathcal{F}P\}\in\mathcal{D}/\mathcal{F}$ is called {\em the $\mathcal{F}$-degree} of $P$.
\index{degree}
\end{definition}

Recall from Corollary \ref{cor:1-3:monoid} that $\mathcal{F}=[\mathfrak{C}_T]^\alpha_{\beta|\gamma}$ forms a monoid for every $\alpha,\beta,\gamma\in\{1,<\omega,\omega\}$.

\begin{notation}
If $\mathcal{F}=[\mathfrak{C}_T]^\alpha_{\beta|\gamma}$ for some $\alpha,\beta,\gamma\in\{1,<\omega,\omega\}$, we write $\leq^\alpha_{\beta|\gamma}$, $\mathcal{D}^\alpha_{\beta|\gamma}$, and $\mathcal{P}^\alpha_{\beta|\gamma}$ instead of $\leq_\mathcal{F}$, $\mathcal{D}/\mathcal{F}$ and $\mathcal{P}/\mathcal{F}$.
\index{$\leq^\alpha_{\beta\mid\gamma}$}\index{$\mathcal{D}^\alpha_{\beta\mid\gamma}$}%
\index{$\mathcal{P}^\alpha_{\beta\mid\gamma}$}%
\end{notation}

\begin{remark}
By Proposition \ref{prop:1-2:func} (4) and (5), the preorderings $\leq^1_1$ and $\leq^\omega_1$ are equivalent to the Medvedev reducibility \cite{Med} and the Muchnik reducibility \cite{Muc}, respectively.
\end{remark}

We also introduce the truth-table versions of Definition \ref{def:1-2:nonunif_bas}.

\begin{definition}
Let $D$ be a subset of Baire space $\nn^\nn$, and $\alpha,\beta,\gamma\leq\omega$ be ordinals.
\index{truth-table!$(\alpha,\beta\mid\gamma)$-truth-table function}\index{truth-table!functional}%
A function $\Gamma:D\to\nn^\nn$ is {\em $(\alpha,\beta|\gamma)$-truth-table} if there are a set $I\subseteq\nn$ of cardinality $\alpha$, and a collection $\{p(e,k):e\in I\;\&\;k<\min\{\beta,\gamma\}\}$ of indices of {\em truth-table functionals} (i.e., ${\rm dom}(\Phi_{p(e,k)})=\nn^\nn$) such that
\begin{enumerate}
\item 
\index{truth-table!$(\alpha,\beta\mid\gamma)$-truth-table function!Popperian condition of}%
(Popperian Condition) for any $e\in I$ and $\sigma\in\nn^{<\nn}$, there is $k<z$ such that $\Psi_e(\sigma)=p(e,k)$.
\item $\Gamma$ is $(\alpha,\beta|\gamma)$-computable via the family $\{\Psi_e\}_{e\in I}$.
\end{enumerate}
Here, we do not assume the uniform computability of the collection $\{p(e,k):e\in I\;\&\;k<\min\{\beta,\gamma\}\}$.
If $\gamma=\omega$, then we simply say that $\Gamma$ is {\em $(\alpha,\beta)$-truth-table} for $(\alpha,\beta|\gamma)$-truth-table function $\Gamma$.
\index{truth-table!$(\alpha,\beta)$-truth-table function}%
Let $[\mathfrak{C}_{tt}]^\alpha_{\beta}$ (resp.\ $[\mathfrak{C}_{tt}]^\alpha_{\beta|\gamma}$) denote the set of all $(\alpha,\beta)$-truth-table (resp.\ $(\alpha,\beta|\gamma)$-truth-table) functions.
\index{$[\mathfrak{C}_{tt}]^\alpha_{\beta}$}\index{$[\mathfrak{C}_{tt}]^\alpha_{\beta\mid\gamma}$}
\end{definition}

\begin{remark}
It is easily checked that the truth-table versions of Proposition \ref{prop:1-2:func}, Proposition \ref{prop:1-3:monoid}, Corollary \ref{cor:1-3:monoid} and Proposition \ref{prop:1-4:monoid2} hold. 
\end{remark}

\begin{notation}
If $\mathcal{F}=[\mathfrak{C}_{tt}]^\alpha_{\beta|\gamma}$ for some $\alpha,\beta,\gamma\in\{1,<\omega,\omega\}$, we write $\leq^\alpha_{tt,\beta|\gamma}$, $\mathcal{D}^\alpha_{tt,\beta|\gamma}$, and $\mathcal{P}^\alpha_{tt,\beta|\gamma}$ instead of $\leq_\mathcal{F}$, $\mathcal{D}/\mathcal{F}$ and $\mathcal{P}/\mathcal{F}$.
\index{$\leq^\alpha_{tt,\beta\mid\gamma}$}\index{$\mathcal{D}^\alpha_{tt,\beta\mid\gamma}$}%
\index{$\mathcal{P}^\alpha_{tt,\beta\mid\gamma}$}%
\end{notation}

\begin{prop}
$\aleph_0=\#[\mathfrak{C}_r]^1_1=\#[\mathfrak{C}_r]^1_{<\omega}=\#[\mathfrak{C}_r]^1_{\omega|<\omega}=\#[\mathfrak{C}_{r}]^1_{\omega}<\#[\mathfrak{C}_{r}]^{<\omega}_1=\#[\mathfrak{C}_r]^{<\omega}_\omega=\#[\mathfrak{C}_r]^\omega_1=2^{2^{\aleph_0}}$, for each $r\in\{tt,T\}$.
\end{prop}

\begin{proof}\upshape
Every learner $\Psi$ determines just one learnable function $\Gamma\in[\mathfrak{C}_T]^1_\omega$.
Therefore, $[\mathfrak{C}_T]^1_\omega$ is countable.
For non-uniform computability, we first see $\#[\mathfrak{C}_T]^\omega_1\leq 2^{2^{\aleph_0}}$ since $\#(\nn^\nn)^{\nn^\nn}=2^{2^{\aleph_0}}$ by cardinal arithmetic.
On the other hand, every function $\Gamma:\nn^\nn\to\{0^\nn,1^\nn\}$ is $(<\omega,1)$-truth-table via two constant truth-table functionals $\Gamma_0(f)=0^\nn$ and $\Gamma_1(f)=1^\nn$ for any $f\in\nn^\nn$.
Therefore, $\#[\mathfrak{C}_{tt}]^{<\omega}_1\geq 2^{2^{\aleph_0}}$.
\end{proof}

\begin{prop}
For each $\alpha,\beta,\gamma\in\{1,<\omega,\omega\}$, the order structures $\mathcal{D}^\alpha_{\beta|\gamma}$, $\mathcal{D}^\alpha_{tt,\beta|\gamma}$, $\mathcal{P}^\alpha_{\beta|\gamma}$, and $\mathcal{P}^\alpha_{tt,\beta|\gamma}$ form lattices with top and bottom elements.
\end{prop}

\begin{proof}\upshape
It is easy to see that the product $\otimes$ and the sum $\oplus$ form supremum and infimum operations in these structures.
Moreover, every degree structure has top and bottom elements since it is coarser than $\mathcal{D}^1_1$, that has top and bottom elements.
\end{proof}

If a lattice $(L,\leq,\vee,\wedge)$ has the top element $1$, the bottom element $0$, and $\max\{c:c\wedge a\leq b\}$ (denoted by $a\rightarrow_L b$) exists for any $a,b\in L$, then $\mathcal{L}=(L,\leq,\vee,\wedge,\rightarrow_L,0,1)$ is called a {\em Heyting algebra}.
\index{Heyting algebra}%
An algebra $\mathcal{L}=(L,\leq,\vee,\wedge,\rightarrow,\bot,\top)$ is a {\em Brouwer algebra} if its dual $\mathcal{L}^{\rm op}=(L,\geq,\wedge,\vee,\leftarrow,\top,\bot)$ is a Heyting algebra.
\index{Brouwer algebra}%
Recall that the Medvedev lattice $\mathcal{D}^1_1$ and the Muchnik lattice $\mathcal{D}^\omega_1$ form Brouwer algebras \cite{Med,Muc}.

\begin{prop}\label{prop:1-2:Brouwerian}
The degree structures $\mathcal{D}^1_\omega$ and $\mathcal{D}^1_{tt,\omega}$ are Brouwerian.
\end{prop}

\begin{proof}\upshape
We just give a proof for $\mathcal{D}^1_\omega$, 
although it is straightforward to modify the proof for the truth-table version. 

Set $B(P,Q)=\{R\subseteq\nn^\nn:P\leq^1_\omega Q\otimes R\}$.
We need to a construct a function $\beta:\mathcal{P}(\nn^\nn)\times\mathcal{P}(\nn^\nn)\to\mathcal{P}(\nn^\nn)$ such that $\beta(P,Q)=\min B(P,Q)$ for any $P,Q\subseteq\nn^\nn$.
Let $\Lambda_e$ denote the $e$-th $(1,\omega)$-computable function, i.e., $\Lambda_e(g)=\Phi_{\lim_n\Psi_e(g\res n)}(g)$ for any $g\in{\rm dom}(\Lambda_e)$.
Define $\beta$ as follows.
\[\beta(P,Q)=\{e\fr g\in\nn^\nn:(\forall f\in Q)\;\Lambda_e(f\oplus g)\in P\}.\]

It is easy to see that $\beta(P,Q)\in B(P,Q)$ for any $P,Q\subseteq\nn^\nn$.
If $R\in B(P,Q)$, say $\Lambda_e:Q\otimes R\to P$, then clearly $e\fr g\in\beta(P,Q)$ for any $g\in R$.
Thus, $\beta(P,Q)\leq^1_1R$. 
\end{proof}

In contrast, we will show in Part II that {\em neither $\mathcal{D}^1_{<\omega}$, nor $\mathcal{D}^1_{\omega|<\omega}$, nor $\mathcal{D}^{<\omega}_1$, nor $\mathcal{D}^{<\omega}_{\omega}$} form Brouwer algebras.
In the meantime, the following modifications of $\mathcal{D}^1_{<\omega}$, $\mathcal{D}^1_{\omega|<\omega}$, $\mathcal{D}^{<\omega}_1$, and $\mathcal{D}^{<\omega}_{\omega}$ look more natural than our original definitions, from the viewpoint of constructive mathematics.
Indeed, in Proposition \ref{prop:1-2:comgen}, we will see that these modifications form Brouwer algebras.

\begin{definition}\label{def:1-2:comgen}
\index{${\tt eff}$}%
Let $D$ be a subset of Baire space $\nn^\nn$, and $\alpha,\beta,\gamma\leq\omega$ be ordinals, or ${\tt eff}$.
We generalize the $(\alpha,\beta|\gamma)$-computability as follows.
If $\alpha={\tt eff}$, then we revise the word ``for any $g\in D$, there is $e\in I$'' to the word ``there is a partial computable function $B_0:\subseteq\nn^\nn\to\nn$ such that, for any $g\in D$, there is $e<B_0(g)$''.
If $\beta={\tt eff}$, then we revise the mind change condition as $\#{\tt mcl}_{\Psi_e}(g)<B_1(g)$, where $B_1$ is a partial computable function from $\nn^\nn$ to $\nn$.
If $\gamma={\tt eff}$, then we revise the error condition as $\#{\tt indx}_{\Psi_e}(g)<B_2(g)$, where $B_2$ is a partial computable function from $\nn^\nn$ to $\nn$.
For new notions, $\leq^\alpha_{\beta|\gamma}$, $\mathcal{D}^\alpha_{\beta|\gamma}$, and $\mathcal{P}^\alpha_{\beta|\gamma}$ are also defined as the usual way.
\end{definition}

\begin{prop}\label{prop:1-2a:eff_compact}
Suppose that, if $\tau={\tt eff}$, then let $\tau^*$ mean the symbol $<\omega$, and otherwise, set $\tau^*=\tau$.
Then, every $(\alpha,\beta|\gamma)$-computable function with a compact domain is $(\alpha^*,\beta^*|\gamma^*)$-computable.
\end{prop}

\begin{proof}
By continuity of $B_0$, $B_1$, and $B_2$ in Definition \ref{def:1-2:comgen}, $\{B_i^{-1}(\{e\})\}_{e\in\nn}$ for each $i<3$ is an open cover of $D$.
Hence, by compactness of $D$, we have the desired condition.
\end{proof}

\begin{cor}
$\mathcal{P}^1_{\tt eff}=\mathcal{P}^1_{<\omega}$; $\mathcal{P}^1_{\omega|\tt eff}=\mathcal{P}^1_{\omega|<\omega}$; $\mathcal{P}^{\tt eff}_1=\mathcal{P}^{<\omega}_1$; and $\mathcal{P}^{\tt eff}_\omega=\mathcal{P}^{<\omega}_\omega$.
\qed
\end{cor}

That is to say, for $\Pi^0_1$ subsets of Cantor space $2^\nn$, no new reducibility notion is constructed from Definition \ref{def:1-2:comgen}.
However, from the perspective of intuitionistic caluculus, our new notions in Definition \ref{def:1-2:comgen} have nice features.

\begin{prop}\label{prop:1-2:comgen}
$\mathcal{D}^1_{\tt eff}$, $\mathcal{D}^1_{\omega|{\tt eff}}$, $\mathcal{D}^{\tt eff}_1$, and $\mathcal{D}^{\tt eff}_{\omega}$ are Brouwerian.
\end{prop}

\begin{proof}
Fix $\alpha,\beta,\gamma\in\{1,<\omega,{\tt eff},\omega\}$, and set $B(P,Q)=\{R\subseteq\nn^\nn:P\leq^\alpha_{\beta|\gamma}Q\otimes R\}$.
We need to construct a function $\beta:\mathcal{P}(\nn^\nn)\times\mathcal{P}(\nn^\nn)\to\mathcal{P}(\nn^\nn)$ such that $\beta(P,Q)=\min B(P,Q)$ for any $P,Q\subseteq\nn^\nn$.
Let $\Lambda_e$ denote the $e$-th $(1,\omega)$-computable function, and $\Theta_e$ be the $e$-th partial computable function from $\nn^\nn$ to $\nn$.
Put ${\rm change}_e(g)=\#\{n\in\nn:\Lambda_e(g\res n+1)\not=\Lambda_e(g\res n)\}$, and ${\rm error}_e(g)=\#\{\Lambda_e(g\res n):n\in\nn\}$.
Then,
\[
\beta(P,Q)=
\begin{cases}
\{(e,d)\fr g:(\forall f\in Q)\;\Lambda_e(f\oplus g)\in P\;\&\;\#{\tt mcl}_{\Lambda_e}(f\oplus g)<\Theta_d(f\oplus g)\},&\\
\hfill\mbox{if }(\alpha,\beta,\gamma)=(1,{\tt eff},\omega),&\\
\{(e,d)\fr g:(\forall f\in Q)\;\Lambda_e(f\oplus g)\in P\;\&\;\#{\tt indx}_{\Lambda_e}(f\oplus g)<\Theta_d(f\oplus g)\},&\\
\hfill\mbox{if }(\alpha,\beta,\gamma)=(1,\omega,{\tt eff}),&\\
\{d\fr g:(\forall f\in Q)(\exists e<\Theta(f\oplus g))\;\Phi_e(f\oplus g)\in P\},&\\
\hfill\mbox{if }(\alpha,\beta,\gamma)=({\tt eff},1,\omega),&\\
\{d\fr g:(\forall f\in Q)(\exists e<\Theta(f\oplus g))\;\Lambda_e(f\oplus g)\in P\},&\\
\hfill\mbox{if }(\alpha,\beta,\gamma)=({\tt eff},\omega,\omega),&
\end{cases}
\]
It is easy to see that $\beta(P,Q)\in B(P,Q)$ for any $P,Q\subseteq\nn^\nn$.
For the minimality, if $R\in B(P,Q)$, we have suitable $d$ and $e$ such that $(d,e)\fr g\in\beta(P,Q)$ for any $g\in R$.
Thus, $\beta(P,Q)\leq^1_1R$.
\end{proof}

\begin{remark}
Unfortunately, neither $\mathcal{P}^1_{\tt eff}$ nor $\mathcal{P}^1_{\omega|{\tt eff}}$ nor $\mathcal{P}^{\tt eff}_1$ nor $\mathcal{P}^{\tt eff}_{\omega}$ form Brouwer algebra (see Part II).
\end{remark}

\subsection{Falsifiable Problems and Total Functions}


In Part II, we will mainly pay attention to the behavior of nonuniform computability on {\em $\Pi^0_1$ subsets of Cantor space $2^\nn$}.
Such a restriction has an interesting feature by thinking of $\Pi^0_1$ sets as {\em falsifiable mass problems}.
\index{falsifiable mass problem}%
Consider a learner $\Psi$ identifies a $(1,\omega)$-computable function $\Gamma:Q\to P$.
On an observation $\sigma\in\nn^{<\nn}$ with $[\sigma]\cap Q\not=\emptyset$, a learner $\Psi$ conjectures that $e$ is a correct algorithm computing a solution of $P$ from $\sigma$, that is, $\Phi_{\Psi(\sigma)}(f)=\Phi_e(f)\in P$ for any future observation $f\in Q\cap[\sigma]$.
If $Q$ is $\Pi^0_1$, Proposition \ref{prop:1-2:collapse-tt} (3) suggests that we may assume that $e$ is an index of a total computable function.
Then, the learner $\Psi$ can find mistakes of his hypothesis on $P$ whenever $P$ is also a $\Pi^0_1$ subset of the Baire space $\nn^\nn$.
Therefore, restricting to $\Pi^0_1$ subsets is expected to be an analogy of {\em Popperian learning}.
In this context, the usual Popperian learning on total computable functions could be regarded as a learning process on $\Pi^0_1$ singletons.
We first see that, if we restrict our attention to $\Pi^0_1$ sets, then some reducibility notions collapse.

\begin{prop}\label{prop:1-2:collapse-tt}
Let $P$ be a $\Pi^0_1$ subset of $\nn^\nn$, and $X$ be any subset of $\nn^\nn$.
\begin{enumerate}
\item $X\leq^1_{tt,1}P$ if and only if $X\leq^1_1P$.
\item $X\leq^1_{tt,<\omega}P$ if and only if $X\leq^1_{<\omega}P$.
\item $X\leq^1_{tt,\omega}P$ if and only if $X\leq^1_{\omega}P$.
\item $P\leq^{1}_{tt,<\omega}X$ if and only if $P\leq^{<\omega}_{tt,\omega|<\omega}X$.
\item $P\leq^{1}_{tt,\omega}X$ if and only if $P\leq^{<\omega}_{tt,\omega}X$.
\end{enumerate}
\end{prop}

\begin{proof}\upshape
(1) See Simpson \cite{Sim}.

(2,3) Assume that $X\leq^1_\omega P$ via a learner $\Psi$.
From $\Psi$, we construct a Popperian learner $\Psi^*:\nn^{<\nn}\to\nn$, i.e., $\Psi(\sigma)$ is an index of truth-table functional for each $\sigma\in\nn^{<\nn}$.
We may assume that $\Psi(\sigma)$ is defined, by Proposition \ref{prop:1-2:learn-trick}.
Let $T_P$ be the corresponding computable tree for $P$.
If $\sigma\not\in T_P$, then $\Psi^*(\sigma)$ returns an index of the constant function $f\mapsto 0^\nn$.
If $\sigma\in T_P$, then let $\Psi^*(\sigma)$ be an index of the following computation procedure.
Given $f\in\nn^\nn$, at stage $s\in\nn$, if $\sigma\not\subset f$, then returns $0^\nn$.
If $f\res s\in T_P$ extends $\sigma$, and $\Psi(f\res t)=\sigma$ for any $|\sigma|\leq t\leq s$, then simulate the computation of $\Phi_{\Psi(\sigma)}(f\res s)$.
Otherwise, for the least such stage $s$, returns $\Phi_{\Psi(\sigma)}(f\res s-1)\fr 0^\nn$.
Clearly, $\Phi_{\Psi^*(\sigma)}(f)$ defines an element of $\nn^\nn$, for any $f\in\nn^\nn$.
Moreover, $\Psi^*$ agrees with $\Psi$ on $P$, i.e., $\Phi_{\lim_n\Psi^*(f\res n)}(f)=\Phi_{\lim_n\Psi(f\res n)}(f)$ for any $f\in P$.

(4,5)
Assume that $P\leq^{<\omega}_{tt,\omega|<\omega}X$ via $n$ Popperian learners, $\{\Psi_i\}_{i<n}$.
Given $g\in X$, on the first challenge, our leaner $\Delta$ guesses that $\Psi_0(g\res 0)$ is a correct algorithm.
As each $\Psi_i$ is Popperian, and $P$ is $\Pi^0_1$, the predicate $\Phi_{\Psi_0(g\res 0)}(g)\in P$ is $\Pi^0_1$.
Therefore, whenever $\Phi_{\Psi_0(g\res 0)}(g)\in P$ is incorrect, the learner $\Delta$ is able to understand that his guess is refuted.
If it happens, the learner goes to the next challenge.
On the $(ns+i)$-th challenge, $\Delta$ guesses that $\Psi_i(g\res s)$ is correct.
By continuing this procedure, eventually $\Delta$ learns a collect algorithm to solve the problem $P$.
Note that, if an $(n,b,c)$-computable function exists from $X$ to $P$, then the learning procedure of $\Delta$ is stabilized before the $(nc)$-th challenge starts, i.e., $\Delta$ determines a $(1,nc)$-truth-table computable function.
\end{proof}

\begin{cor}
$\mathcal{P}^1_{tt,1}=\mathcal{P}^1_1$; $\mathcal{P}^1_{tt,<\omega}=\mathcal{P}^1_{tt,\omega|<\omega}=\mathcal{P}^{<\omega}_{tt,\omega|<\omega}=\mathcal{P}^1_{<\omega}$;
and $\mathcal{P}^1_{tt,\omega}=\mathcal{P}^{<\omega}_{tt,\omega}=\mathcal{P}^1_\omega$.
Hence, $\{\mathcal{P}^\alpha_{\beta|\gamma}, \mathcal{P}^\alpha_{tt, \beta|\gamma}:\alpha,\beta,\gamma\in\{1,<\omega,\omega\}\}$ consists of at most nine lattices: $\mathcal{P}^1_1$, $\mathcal{P}^{<\omega}_{tt,1}$, $\mathcal{P}^1_{<\omega}$, $\mathcal{P}^1_{\omega|<\omega}$, $\mathcal{P}^{<\omega}_1$, $\mathcal{P}^1_{\omega}$, $\mathcal{P}^{<\omega}_\omega$, $\mathcal{P}^\omega_{tt,1}$, and $\mathcal{P}^\omega_1$.\qed
\end{cor}

One can interpreted $\leq^1_1$ ($\leq^1_\omega$, resp.) as computable reducibility with no (finitely many, resp.) mind-changes.
We see how $\leq^1_\omega$ behaves like a dynamical-approximation procedure.

\begin{prop}\label{lem:weak}
For any $\Pi^0_1$ set $P\subseteq\nn^\nn$ and any set $Q\subseteq \nn^\nn$, $P\leq^\omega_1 Q$ if and only if
\[(\exists\Psi)(\forall f\in Q)\;\Phi_{\liminf_n\Psi(f\res n)}(f)\in P.\]
Here $\Psi$ ranges over all learners (i.e., computable functions from $\nn^{<\nn}$ to $\nn$).
\end{prop}

\begin{proof}\upshape
The ``only if'' part is obvious.
For the ``if'' part, we will inductively define $\Psi(\sigma)$ and $l(\sigma,e)$ for each $\sigma\in\nn^{<\nn}$ and $e\in\nn$.
Let $T_P$ denote the corresponding tree for $P$.
First, put $\Psi(\lrangle{})=0$ and $l(\lrangle{},e)=0$ for each $e$.
Now assume that, for any $\tau\in\nn^{<\nn}$ with $|\tau|<|\sigma|$, we have already defined $\Psi(\tau)$, and $l(\tau,e)$ for each $e\in\nn$.
Then, we define $\Psi(\sigma)$ and $l(\sigma,e)$ for each $e$ as follows:
\begin{align*}
\Psi(\sigma)&=
\begin{cases}
\mu e<|\sigma|\;[\Phi_e(\sigma)\res (l(\sigma^-,e)+1)\in T_P]&\mbox{ if such }e\mbox{ exists,}\\
|\sigma|&\mbox{ otherwise}
\end{cases}
\\
l(\sigma,e)&=
\begin{cases}
l(\sigma^-,e)+1&\mbox{ if }e=\Psi(\sigma),\\
l(\sigma^-,e)&\mbox{ otherwise.}
\end{cases}
\end{align*}
By our assumption $P\leq^\omega_1 Q$, $\liminf_n\Psi(f\res n)$ exists for all $f\in P$.
Thus, the desired condition $\Phi_{\liminf_n\Psi(f\res n)}(f)\in Q$ holds.
\end{proof}

\begin{remark}
Recall that a subset of $2^\nn$ is $\Pi^0_1$ if and only if it is the set of all infinite paths through a computable subtree of $2^{<\nn}$.
Thus, in our model of inductive inference, each learner tries to learn a program for an infinite branch of $T$ from a given infinite branch of another tree $T^*$.
Another model of {\em branch learning} has been studied by Kummer-Ott \cite{KuOt}, and Ott-Stephan \cite{OtSt} in which each learner tries to learn a program for an infinite {\em computable} branch of $T$ from the global information about $T$.
They pointed out that the concept of branch learning is equivalent to learning winning strategies for closed computable Gale-Stewart games, since the class of $\Pi^0_1$ subsets of $2^\nn$ correspond exactly to the class of winning strategies for such games (see also Cenzer-Remmel \cite{CR}).
Case-Ott-Sharma-Stephan \cite{CaOtShSt} explains the concept of branch learning by using a temperature controller.
In their model, each learner tries to learn a program for an infinite computable branch of $T$ from the global information about $T$ {\em with an additional information about one infinite branch of $T$}, i.e., the leaner may watch a human {\em master}.
A $k$-wise variation for branch learning called {\em weak $k$-search problem} has been studied by Kaufmann-Kummer \cite{KaKu}.
\end{remark}

\subsection{Learnability versus Piecewise Computability}

Now we characterize our classes of nonuniformly computable functions using the concept of piecewise computability. 

\begin{definition}
For a class $\Lambda$ of subsets of Baire space $\nn^\nn$, we say that a collection $\{Q_i\}_{i\in I}$ is uniformly $\Lambda$ if the set $\{(i,f)\in I\times\nn^\nn:f\in Q_i\}$ belongs to $\Lambda$.
\index{uniform!$\Lambda$ collection}%
A partition or a cover $\{Q_i\}_{i\in I}$ of $Q$ is (uniformly) $\Lambda$ if there is a (uniform) $\Lambda$ collection $\{Q_i^*\}_{i\in I}$ such that $Q_i=Q\cap Q_i^*$ for any $i\in I$.
\index{uniform!$\Lambda$ partition}%
\index{uniform!$\Lambda$ cover}%
We say that $\{Q_i\}_{i\in I}$ is {\em a (uniform) $\Lambda$ layer of $Q$} if there is a uniform $\Lambda$ collection $\{Q^*_i\}_{i\in I}$ such that $Q^*_i\subseteq Q^*_{i+1}$ for each $i\in I$, $\{Q^*_i\}_{i\in I}$ covers $Q$, and $Q_i=Q\cap Q^*_i$.
\index{layer}\index{uniform!$\Lambda$ layer}%
We also say that $\{Q_i\}_{i\in I}$ is {\em a (uniform) $\Lambda$ $d$-layer of $Q$} if there is a (uniform) $\Lambda$ layer $\{Q^*_i\}_{i\in I}$ of $Q$ such that $Q_i=Q^*_i\setminus Q^*_{i-1}$ for any $i\in I$, where $Q^*_{-1}=\emptyset$.
\index{layer!$d$-layer}\index{uniform!$\Lambda$ $d$-layer}%
\end{definition}

\begin{remark}
The terminology ``{\em layer}'' comes from the concept of {\em layerwise computability} in algorithmic randomness theory (see Hoyrup-Rojas \cite{HoyrupR09}).
\index{layer!layerwise computability}%
\end{remark}

\begin{definition}\label{def:1-2a:piecewise-computability}
Let $\mathcal{F}$ be a class of partial functions on $\nn^\nn$.
\index{${\rm dec}^{X}_{x}[\Lambda]\mathcal{F}$}%
For $X\in\omega\cup\{<\omega,\omega\}$ and $x\in\{{\rm p},{\rm c},{\rm d}\}$, a partial functions $\Gamma:\subseteq\nn^\nn\to\nn^\nn$ is of class ${\rm dec}^{X}_{x}[\Lambda]\mathcal{F}$ if there is a uniform $\Lambda$ partition (if $x={\rm p}$), uniform cover (if $x={\rm c}$) or uniform $d$-layer (if $x={\rm d}$), $\{Q_i\}_{i\in I}$, of ${\rm dom}(\Gamma)$ such that $\Gamma\res Q_i$ is contained in $\mathcal{F}$ uniformly in $i\in I$, where $I=X$ if $X\in\omega\cup\{\omega\}$ and $I\in\omega$ if $X=<\omega$.
If $\mathcal{F}$ is the class of all partial computable functions, we simply write ${\rm dec}^{X}_{x}[\Lambda]$ instead of ${\rm dec}^{X}_{x}[\Lambda]\mathcal{F}$.
\index{${\rm dec}^{X}_{x}[\Lambda]$}%
Moreover, if $\Lambda$ is the class of all subsets of Baire space, then we write ${\rm dec}^{X}_{x}[-]$ and ${\rm dec}^{X}_{x}\mathcal{F}$ instead of ${\rm dec}^{X}_{x}[\Lambda]$ and ${\rm dec}^{X}_{x}[\Lambda]\mathcal{F}$, respectively.
\index{${\rm dec}^{X}_{x}[-]$}\index{${\rm dec}^{X}_{x}\mathcal{F}$}%
If we does not assume uniformity in the definition, we say that $\Gamma$ is of $\underline{\rm dec}^{X}_{x}[\Lambda]\mathcal{F}$.
\index{$\underline{\rm dec}^{X}_{x}[\Lambda]\mathcal{F}$}%
\end{definition}

If $\Lambda\in\{\Sigma^0_n,\Pi^0_n,\Delta^0_n\}_{n\in\nn}$, for every $X\in\{<\omega,\omega\}$, we have ${\rm dec}^{X}_{\rm p}[\Lambda]\subseteq{\rm dec}^{X}_{\rm c}[\Lambda]\subseteq{\rm dec}^{X}_{\rm d}[\Lambda]\subseteq{\rm dec}^{X}_{\rm p}[(\Lambda)_2]$.
Here a set is $(\Lambda)_2$ if it is the difference of two $\Lambda$ sets.
Note that ${\rm dec}^\omega_p[\Pi^0_n]={\rm dec}^\omega_c[\Sigma^0_{n+1}]$ holds for every $n\in\nn$.
Our seven concepts of nonuniform computability listed in Table \ref{rtable} can be characterized as classes of piecewise computable functions.

\begin{theorem}\label{thm:5:red-eq-dis}
Let $k$ be any finite number.
\begin{enumerate}
\item $[\mathfrak{C}_T]^1_k={\rm dec}^{k}_{\rm d}[\Pi^0_1]$.
\item $[\mathfrak{C}_T]^1_{\omega|k}={\rm dec}^{k}_{x}[\Delta^0_2]={\rm dec}^{k}_{\rm c}[\Sigma^0_2]$ for any $x\in\{{\rm p},{\rm c},{\rm d}\}$.
\item $[\mathfrak{C}_T]^1_{\omega}={\rm dec}^{\omega}_{x}[\Pi^0_1]={\rm dec}^{\omega}_{x}[\Delta^0_2]={\rm dec}^{\omega}_{\rm c}[\Sigma^0_2]$ for any $x\in\{{\rm p},{\rm c},{\rm d}\}$.
\item $[\mathfrak{C}_T]^{k}_{1}={\rm dec}^{k}_{\rm x}[-]$ for any $x\in\{{\rm p},{\rm c},{\rm d}\}$.
\item $[\mathfrak{C}_T]^{k}_{\omega}={\rm dec}^{k}_y{\rm dec}^{\omega}_{x}[\Pi^0_1]={\rm dec}^{k}_y{\rm dec}^{\omega}_{x}[\Delta^0_2]={\rm dec}^{k}_y{\rm dec}^{\omega}_{\rm c}[\Sigma^0_2]$ for any $x,y\in\{{\rm p},{\rm c},{\rm d}\}$.
\item $[\mathfrak{C}_T]^{\omega}_{1}={\rm dec}^{\omega}_{\rm x}[-]$ for any $x\in\{{\rm p},{\rm c},{\rm d}\}$.
\end{enumerate}
\end{theorem}

\begin{proof}\upshape
(1) Let $\Psi:\nn^{<\nn}\to\nn$ be a learner witnessing $\Gamma\in[\mathfrak{C}_T]^1_k$.
Then for each $m<k$, let ${\tt mc}_\Psi(\leq m)$ denote the set of all $g\in\nn^\nn$ such that $\#{\tt mcl}_\Psi(g)\leq m$.
The sets ${\tt mc}_\Psi(<m)$ and ${\tt mc}_\Psi(=m)$ are also defined by the same manner.
Then, it is easy to check that ${\tt mc}_\Psi(\leq m)$ and ${\tt mc}_\Psi(<m)$ are $\Pi^0_1$.
For each $m<k$, consider the following computable procedure $\Phi_{e(m)}$: given $g\in{\tt mc}_\Psi(=m)$, look for the least $n\in\nn$ such that $[g\res n]$ is included in the open set ${\tt mc}_\Psi(\geq m)$, and then return $\Phi_{\Psi(g\res n)}(g)$.
It is not hard to see that $\Gamma$ is decomposable into $k$ many computable functions $\{\Phi_{e(m)}\}_{m<k}$ with $\Pi^0_1$ $d$-layered domains $\{{\tt mc}_\Psi(=m)\}_{m<k}$.

Conversely, assume that $\Gamma\in{\rm dec}^{k}_{\rm d}[\Pi^0_1]$ is given.
Then, $\Gamma$ is decomposed into computable functions $\{\Phi_{e(m)}\}_{m<k}$ with $d$-layered domains $\{Q_m\setminus Q_{m-1}\}_{m<k}$, where $\{Q_m\}_{m< k}$ computable increasing sequence $\{Q_m\}_{m< k}$ of $\Pi^0_1$ sets with $Q_{-1}=\emptyset$.
For each $\sigma\in\nn^{<\nn}$, we compute the least $i(\sigma)$ such that $\sigma\in T_{Q_{i(\sigma)}}$, i.e., $\sigma\in T_{Q_{i(\sigma)}}\setminus T_{Q_{i(\sigma)-1}}$.
Then, on $\sigma\in\nn^{<\nn}$, the learner $\Psi$ guesses $\Psi(\sigma)=e(i(\sigma))$.
By our assumption, for any $g\in{\rm dom}(\Gamma)$, we have $g\in Q_i$ for some $i\in\nn$.
Then, $\lim_n\Psi(g\res n)$ converges to the least $e(i)$ such that $g\in Q_i$.
Again, by our assumption, we have $\Phi_{\lim_n\Psi(g\res n)}(g)=\Phi_{e(i)}(g)=\Gamma(g)$ for any $g\in {\rm dom}(\Gamma)\cap(Q_i\setminus Q_{i-1})$.
Therefore, we have $\Gamma\in[\mathfrak{C}_T]^1_k$.

\medskip

(2) Let $\Psi:\nn^{<\nn}\to\nn$ be a learner witnessing $\Gamma\in[\mathfrak{C}_T]^1_{\omega|k}$.
We define ${\tt reindex}_\Psi:\nn^{<\nn}\to\nn$ reindexing $\Psi(\sigma)$ in order of occurrence.
Put ${\tt reindex}_\Psi(\lrangle{})=0$.
Fix $\sigma\in\nn^{<\nn}$, and assume that ${\tt reindex}_\Psi(\tau)$ has been already defined for each $\tau\subsetneq\sigma$.
If $\Psi(\sigma)=\Psi(\tau)$ for some $\tau\subsetneq\sigma$, then we set ${\tt reindex}_\Psi(\sigma)={\tt reindex}_\Psi(\tau)$ for such $\tau$.
If there is no such $\tau$, then we set ${\tt reindex}_\Psi(\sigma)=\max\{{\tt reindex}_\Psi(\tau):\tau\subsetneq\sigma\}+1$.
Our assumption $\Gamma\in[\mathfrak{C}_T]^1_{\omega|k}$ implies that for every $g\in{\rm dom}(\Gamma)$, ${\tt reindex}_\Psi(g)=\lim_n{\tt reindex}_\Psi(g\res n)$ converges to a value less than $k$.
Hence, $R_m=\{g\in\nn^{<\nn}:\lim_n{\tt reindex}_\Psi(g\res n)=m\}$ is $\Delta^0_2$ in ${\rm dom}(\Gamma)$ uniformly in $m<k$.
For each $m<k$, consider the following computable procedure $\Phi_{e(m)}$: given $g\in R_m$, look for the least $n\in\nn$ such that ${\tt reindex}_\Psi(g\res n)=m$, and then return $\Phi_{\Psi(g\res n)}(g)$.
It is not hard to see that $\Gamma$ is decomposable into $k$ many computable functions $\{\Phi_{e(m)}\}_{m<k}$ with $\Delta^0_2$ domains $\{R_m\}_{m<k}$.

Conversely, assume that $\Gamma\in{\rm dec}^{k}_{\rm c}[\Sigma^0_2]$ is given.
Then, $\Gamma$ is decomposed into computable functions $\{\Phi_{e(m)}\}_{m<k}$ with $\Sigma^0_2$ domains $\{Q_m\}_{m<k}$.
Then, there is a computable relation $R\subseteq\nn\times\nn^{<\nn}$ such that $Q_m=\{g\in{\rm dom}(\Gamma):(\exists s)(\forall t>s)\;R(m,g\res t)\}$ for every $m\in\nn$.
We set $\Psi(\sigma)=e(\min(\{m:R(m,\sigma)\}\cup\{k-1\}))$.
Since ${\rm dom}(\Gamma)$ is covered by $\{Q_m\}_{m<k}$, for any $g\in{\rm dom}(\Gamma)$,  $\lim_n\Psi(g\res n)$ converges to some value $e(m)$, where $g\in Q_m$.
Moreover, the definition of $\Psi$ ensures that $\#\{\Psi(\sigma):\sigma\in\nn^{<\nn}\}\leq k$.
Therefore, we have $\Gamma\in[\mathfrak{C}_T]^1_{\omega|k}$.

\medskip

(3)
It is straightforward to show the $[\mathfrak{C}_T]^1_{\omega}={\rm dec}^\omega_d[\Pi^0_1]$ by the similar argument used in proof of (1).
Here, we note that ${\rm dec}^\omega_p[\Pi^0_1]={\rm dec}^\omega_c[\Sigma^0_2]$ as mentioned above.

\medskip

(4) It is obvious from the definition.

\medskip

(5) Combine (3) and (4).

\medskip

(6) It is obvious from the definition.
\end{proof}

\begin{table}\caption{Seven Classes of Nonuniformly Computable Functions}\label{stable}%
\begin{center}
\begin{tabular}{ccc}\toprule
$[\mathfrak{C}_T]^1_{<\omega}$ & ${\rm dec}^{<\omega}_{\rm d}[\Pi^0_1]$ & finite $(\Pi^0_1)_2$-piecewise computable \\
$[\mathfrak{C}_T]^1_{\omega|<\omega}$ & ${\rm dec}^{<\omega}_{\rm p}[\Delta^0_2]$ & finite $\Delta^0_2$-piecewise computable \\
$[\mathfrak{C}_T]^1_\omega$ & ${\rm dec}^{\omega}_{\rm p}[\Pi^0_1]$ & $\Pi^0_1$-piecewise computable \\
$[\mathfrak{C}_T]^{<\omega}_1$ & ${\rm dec}^{<\omega}_{\rm p}[-]$ & finite piecewise computable \\
$[\mathfrak{C}_T]^{<\omega}_\omega$ & ${\rm dec}^{<\omega}_{\rm p}{\rm dec}^{\omega}_{\rm p}[\Pi^0_1]$ & finite piecewise $\Pi^0_1$-piecewise computable \\
$[\mathfrak{C}_T]^\omega_1$ & ${\rm dec}^{\omega}_{\rm p}[-]$ & countably computable \\
\bottomrule
\end{tabular}
\end{center}
\end{table}

\begin{prop}\label{prop:5:refref1}
Let $P$ and $Q$ be subsets of $\nn^\nn$, where $P$ is $\Pi^0_n$ for $n\geq 2$.
Let $k$ be any finite number.
\begin{enumerate}
\item There is $\Gamma:Q\to P$ with $\Gamma\in[\mathfrak{C}_{T}]^{k}_1$ if and only if there is $\Gamma:Q\to P$ with $\Gamma\in{\rm dec}^k_{\rm d}[\Pi^0_n]$.
\item There is $\Gamma:Q\to P$ with $\Gamma\in[\mathfrak{C}_{T}]^{<\omega}_{\omega}$ if and only if there is $\Gamma:Q\to P$ with $\Gamma\in{\rm dec}^{<\omega}_{\rm d}[\Pi^0_n]{\rm dec}_{\rm p}^{\omega}[\Pi^0_1]$.
\item There is $\Gamma:Q\to P$ with $\Gamma\in[\mathfrak{C}_{T}]^{\omega}_1$ if and only if there is $\Gamma:Q\to P$ with $\Gamma\in{\rm dec}^\omega_{\rm d}[\Pi^0_n]$.
\end{enumerate}
Hence, $\mathcal{P}^{<\omega}_{1}=\mathcal{P}/{\rm dec}^{<\omega}_{\rm d}[\Pi^0_2]$, $\mathcal{P}^{<\omega}_{\omega}=\mathcal{P}/{\rm dec}^{<\omega}_{\rm d}[\Pi^0_2]{\rm dec}_{\rm p}^{\omega}[\Pi^0_1]$, and $\mathcal{P}^{\omega}_{1}=\mathcal{P}/{\rm dec}^\omega_{\rm d}[\Pi^0_2]$.
Here, recall from Definition \ref{def:1-2a:degree} that $\mathcal{P}/\mathcal{F}$ denotes the $\mathcal{F}$-degree structure of nonempty $\Pi^0_1$ subsets of Cantor space.
\end{prop}

\begin{proof}\upshape
For (3), we assume that $P\leq^{\omega}_1Q$.
Every partial computable function $\Phi_e$ can be assumed to have a $\Pi^0_2$ domain $D_e$.
Then, $Q_e=\bigcup_{d\leq e}(D_d\cap\Phi_d^{-1}[P])$ is $\Pi^0_n$, and $\{Q_e\}_{e\in\nn}$ forms a $\Pi^0_n$ layer.
Moreover, it is not hard to see that $\Phi_e$ maps every element of $Q_e\setminus Q_{e-1}$ into $P$.

For (2), we assume that $P\leq^{<\omega}_\omega Q$ is witnessed by two functions $\Gamma\in{\rm dec}^2_{\rm p}{\rm dec}^\omega_{\rm p}[\Pi^0_1]$ by Theorem \ref{thm:5:red-eq-dis}.
Then there is a collection of partial computable functions $\{\Gamma^i_n\}_{i<2,n\in\nn}$ and a partition $\{E_i\}_{i<2}$ of $Q$ and collections $\{Q^i_n\}_{n\in\nn}$ of pairwise disjoint $\Pi^0_1$ sets that covers $E_i$ and $\Gamma$ agrees with $\Gamma^i_n$ on the domain $E_i\cap Q^i_n$ for every $i<2$ and $n\in\nn$.
Then, $E^*_1=\bigcup_{n\in\nn}(Q^0_n\cap(\Gamma^0_n)^{-1}[\nn^\nn\setminus P])$ is $\Sigma^0_n$ and included in $E_1$.
Thus, $\{E^*_0,E^*_1\}$ forms a $\Pi^0_n$ $d$-layer, where $E^*_0=\nn^\nn\setminus E^*_1$.
It is not hard to see that $\Gamma$ agrees with $\Gamma^i_n$ on the domain $Q\cap E^*_i\cap Q^i_n$ for every $i<2$ and $n\in\nn$.
\end{proof}

\begin{remark}
It is not hard to see that ${\rm dec}_{\rm p}^{<\omega}[\Pi^0_1]$ is exactly the class of all partial computable functions, because, given a finite $\Pi^0_1$ partition $\{Q_i\}_{i<k}$ and $g\in{\rm dom}(\Gamma)$, we can effectively find the unique piece containing $g$. 
\end{remark}

%% file: NRMP_fullproof1-2b.tex
\section{Strange Set Constructions}

\subsection{Medvedev's Semantics for Intuitionism}
To introduce useful set constructions, let us return back to Medvedev's original idea.
To formulate semantics for the intuitionistic propositional calculus ({\sf IPC}), Kolmogorov tried to interpret each proposition as a problem.
Medvedev \cite{Med} formalized his idea by interpreting each proposition $p$ as a mass problem $\bhk{p}\subseteq\nn^\nn$.
Under the interpretation:
\begin{enumerate}
\item A {\em proof} $\pi$ is a dynamical process represented by an infinite sequence of natural numbers, i.e., $\pi\in\nn^\nn$.
\item $\bhk{p}$ is the set of all proofs of a proposition $p$, i.e., $\bhk{p}\subseteq\nn^\nn$.
\item A proposition $p$ is {\em provable} if $p$ has a computable proof, i.e., $\bhk{p}\subseteq\nn^\nn$ contains a computable element.
\end{enumerate}
To prove the disjunction $p_0\vee p_1$, we need to algorithmically decide which part is valid, i.e., we first declare one part to be valid and then construct a witness for this part.
Consequently, $p_0\vee p_1$ is provable under that interpretation if and only if we can algorithmically construct an element of $\bhk{p_0\vee p_1}=\bhk{p_0}\linf\bhk{p_1}=\{\lrangle{i}\fr f:i<2\;\&\;f\in\bhk{p_i}\}$.
Generally, let {\sf Form} denote the all propositional formulas.
Medvedev's idea is defining a mass-problem-interpretation of {\sf IPC} by a function $\bhk{\cdot}:{\sf Form}\to\mathcal{P}(\nn^\nn)$ as in Definition \ref{def:1-2:Med-interpre}.

\begin{definition}\label{def:1-2:Med-interpre}
\index{Medvedev!interpretation}%
We say that a function $\bhk{\cdot}:{\sf Form}\to\mathcal{P}(\nn^\nn)$ is {\em a Medvedev interpretation} if it satisfies the following six conditions.
\begin{enumerate}
\item $\bhk{\top}$ contains a computable element.
\item $\bhk{\bot}=\emptyset$.
\item $\bhk{\varphi\wedge\psi}=\bhk{\varphi}\lsup\bhk{\psi}=\{f\oplus g:f\in\bhk{\varphi}\;\&\;g\in\bhk{\psi}\}$.
\item $\bhk{\varphi\vee\psi}=\bhk{\varphi}\linf\bhk{\psi}=\{\lrangle{0}\fr f:f\in\bhk{\varphi}\}\cup\{\lrangle{1}\fr g:g\in\bhk{\psi}\}$.
\item $\bhk{\varphi\rightarrow\psi}=\bhk{\varphi}\!\boldsymbol{\rightarrow}\!\bhk{\psi}=\{e\fr g\mid\Phi_e(g\oplus *):\bhk{\varphi}\to\bhk{\psi}\}$.
\item $\bhk{\neg\varphi}=\bhk{\varphi\rightarrow\bot}$.
\end{enumerate}
Here, $\Phi(g\oplus *)$ denotes the partial function $\lambda f.\Phi(g\oplus f):\subseteq\nn^\nn\to\nn^\nn$, and recall that $\Phi_e$ is the $e$-th partial computable function on $\nn^\nn$.
Arithmetical quantifications can also be interpreted as follows.
\begin{enumerate}
\item[7.] $\bhk{\exists n\varphi(n)}=\bigoplus_{n\in\nn}\bhk{\varphi(n)}$.
\item[8.] $\bhk{\forall n\varphi(n)}=\bigotimes_{n\in\nn}\bhk{\varphi(n)}$.
\end{enumerate}
\end{definition}

As mentioned in Section \ref{sec:1-2a:deg_brow}, Medvedev \cite{Med} showed that the quotient algebra $\mathcal{D}^1_1$ called the Medvedev lattice is Brouwerian under Medvedev's interpretation (Definition \ref{def:1-2:Med-interpre}).
Following him, Muchnik \cite{Muc} showed that $\mathcal{D}^\omega_1$ called the Muchnik lattice is Brouwerian.
Usually, the Medvedev reducibility is written as $\leq_M$ or $\leq_s$ rather than $\leq^1_1$, and the Muchnik reducibility is written by $\leq_w$ rather than $\leq^\omega_1$.
\index{Medvedev!reducibility}\index{$\leq_M$}\index{$\leq_s$}\index{$\leq_w$}%

\begin{remark}~
\begin{enumerate}
\item 
\index{logic!Jankov}\index{excluded middle!weak law of}\index{logic!de Morgan}%
Both of the Medvedev lattice $\mathcal{D}^1_1$ and the Muchnik lattice $\mathcal{D}^\omega_1$ provide sound and complete semantics for {\em Jankov's Logic} ${\sf KC}={\sf IPC}+\neg p\vee\neg\neg p$, the intuitionistic propositional logic with {\em the weak law of excluded middle}, which is also called {\em De Morgan logic}.
The Medvedev lattice and the Muchnik lattice are extensively studied from the aspect of Intermediate Logic.
See Sorbi-Terwijn \cite{SoTe} and Hinman \cite{Hin}.
\item 
\index{${\sf RCA}$}\index{${\sf WKL}$}%
Forty years after the pioneering work by Muchnik, the Muchnik reducibility become useful in the context of Reverse Mathematics (see Simpson \cite{SimRM}).
The reason is that the Muchnik reducibility $\leq^\omega_1$ is strongly associated with the provability relation in ${\sf RCA}$, {\em the recursive comprehension axiom}.
Then, the Muchnik degrees of $\Pi^0_1$ subsets of $2^\nn$ might be seen as instances of ${\sf WKL}$, {\em the weak K\"onig's lemma}.
For example, by using a result of Binns and Simpson \cite{BS} for the Muchnik degrees of $\Pi^0_1$ subsets of $2^\nn$, Mummert \cite{Mum} obtains an embedding theorem about the Lindenbaum algebra between ${\sf RCA}_0$ and ${\sf WKL}_0$.
\item For more basic results about the Medvedev and Muchnik degrees of $\Pi^0_1$ subsets of $2^\nn$, see Simpson \cite{Sim,Sim2,Sim5,Simta}.
There are lots of research on the algebraic structure of the Medvedev degrees of $\Pi^0_1$ subsets of $2^\nn$, such as density \cite{CH1}, embeddability of distributive lattices \cite{BS}, join-reducibility \cite{Bin}, meet-irreducibility \cite{Al1}, noncuppability \cite{CKWW}, decidability \cite{CK}, and undecidability \cite{Sha}.
The structure of Weihrauch degrees, an extension of the Medvedev degrees, has also been widely studied as a computable-analysistic approach to (Constructive) Reverse Mathematics (see \cite{BMP,BGa,BG}).
\end{enumerate}
\end{remark}

\subsection{Disjunction Operations Based on Learning Theory}

Hayashi \cite{Hay0,Hay} introduced {\em Limit Computable Mathematics} ({\sf LCM}), an extended constructive mathematics based on {\em Learning Theory}.
Like the BHK-interpretation for intuitionistic logic, there is a {\em limit-BHK interpratation} for Limit Computable Mathematics.
We introduce three mass-problem-interpretations $\bhk{\cdot}_{\sf LCM}^i:{\sf Form}\to\mathcal{P}(\nn^\nn)$ of {\sf LCM} based on the limit-BHK interpretation.
To formulate a mass-problem-style interpretation of {\sf LCM}, imagine the following {\em dynamic} proof models.

\medskip

\noindent
\index{model!one-tape}%
{\bf The one-tape model} is defined as follows:
When a verifier $\Psi$ tries to prove that ``$P_0$ or $P_1$'', a tape $\Lambda$ is given.
At each stage, $\Psi$ declares $0$ or $1$, and writes one letter on the tape $\Lambda$.
\begin{itemize}
\item {\bf Intuitionism}: $\Psi$ does not change his declaration, say $i\in\{0,1\}$, and the infinite word written on the tape $\Lambda$ witnesses the validity of $P_i$.
\item {\bf LCM}: the sequence of declarations of $\Psi$ converges, say $i\in\{0,1\}$, and the infinite word written on the tape $\Lambda$ witnesses the validity of $P_i$.
\item {\bf Classical}: any declaration of $\Psi$ is nonsense, and the infinite word written on the tape $\Lambda$ witnesses the validity of $P_0$ or $P_1$.
\end{itemize}

\noindent
\index{model!two-tape}%
{\bf The two-tape model} is follows:
When a verifier $\Psi$ tries to prove ``$P_0$ or $P_1$'', two tapes $\Lambda_0$ and $\Lambda_1$ are given.
At each stage, $\Psi$ declares $0$ or $1$, say $i$, and he writes one letter on the tape $\Lambda_i$.
\begin{itemize}
\item {\bf Intuitionism}: For either $i<2$, the word written on $\Lambda_{1-i}$ is empty, and the infinite word written on $\Lambda_i$ witnesses the validity of $P_i$.
\item {\bf LCM}: For either $i<2$, the word written on $\Lambda_{1-i}$ is finite, and the infinite word written on $\Lambda_i$ witnesses the validity of $P_i$.
\item {\bf Classical}: For either $i<2$, the infinite word written on $\Lambda_i$ witnesses the validity of $P_i$.
\end{itemize}

\noindent
\index{model!backtrack-tape}%
{\bf The backtrack-tape model} is follows:
When a verifier $\Psi$ tries to prove that ``$P_0$ or $P_1$'', a cell $\square$, and two infinite tapes $\Lambda,\Delta$ are given.
The cell $\square$ is called {\em the declaration}, $\Lambda$ is called {\em the working tape}, and $\Delta$ is called {\em the record tape}.
At each stage, the verifier $\Psi$ works as follows.
\begin{enumerate}
\item If no letter is written on the declaration $\square$, then $\Psi$ declares $0$ or $1$ and this is written on the declaration $\square$ and the record tape $\Delta$.
\item When some letter is written on the declaration $\square$, the verifier $\Psi$ chooses one letter $k$ from $\nn\cup\{\sharp\}$, and his choice $k$ is written on the record tape $\Delta$.
\begin{enumerate}
\item In the case $k\not=\sharp$, it expresses that $\Psi$ writes the letter $k$ on the working tape $\Lambda$.
\item In the case $k=\sharp$, it expresses that $\Psi$ erases all letters from the declaration $\square$ and the working tape $\Lambda$.
\end{enumerate}
\end{enumerate}
\begin{itemize}
\item {\bf Intuitionism}: $\Psi$ does not choose $\sharp$, hence he does not change his declaration, say $i$, and the infinite word written on the tape $\Lambda$ witnesses the validity of $P_i$.
\item {\bf LCM}: $\Psi$ chooses $\sharp$ at most finitely often, hence the sequence of declarations of $\Psi$ converges, say $i$, and the infinite word written on the tape $\Lambda$ witnesses the validity of $P_i$.
\item {\bf Classical}: No classical counterpart.
\end{itemize}

To give formal definitions of these dynamic proof models, we introduce some auxiliary definitions.

\begin{definition}[Notations for One/Two-Tape Models]
Let $I\subseteq\nn$ be a set of indices of working tapes.
A pair $\pair{x_0,x_1}\in I\times\nn$ indicates the instruction to write the letter $x_1\in\nn$ on the $x_0$-th tape.
Then every string $\sigma=\pair{i(t),n(t)}_{t<s}\in(I\times \nn)^{<\nn}$ can be think of as the {\em record} of the process that obeys the sequence of instructions $\pair{i(0),n(0)},\pair{i(1),n(1)},\dots,\pair{i(s-1),n(s-1)}$.
Fix $\sigma\in(I\times\nn)^{<\nn}$, and $i\in I$.
Then {\em the $i$-th projection of $\sigma$} is inductively defined as follows.
\index{${\tt pr}_i$}%
\begin{align*}
{\tt pr}_i(\lrangle{})=\lrangle{},& & {\tt pr}_i(\sigma)=
\begin{cases}
{\tt pr}_i(\sigma^-)\fr n, \mbox{ if } \sigma=\sigma^-\fr\lrangle{\pair{i,n}},\\
{\tt pr}_i(\sigma^-), \mbox{ otherwise.}
\end{cases}
\end{align*}
The string ${\tt pr}_i(\sigma)$ represents the word written on the $i$-th tape reconstructed from the record $\sigma$.
Moreover, {\em the number of times of mind-changes of (the process reconstructed from a record) $\sigma\in (I\times\nn)^{<\nn}$} is given by
\index{${\tt mc}$}%
\[{\tt mc}(\sigma)=\#\{n<|\sigma|-1:(\sigma(n))_0\not=(\sigma(n+1))_0\}.\]
Here, for $x=\pair{x_0,x_1}\in I\times\nn$, the first (second, resp.) coordinate $x_0$ ($x_1$, resp.) is denoted by $(x)_0$ ($(x)_1$, resp.).
Furthermore, for $f\in(I\times\nn)^{\nn}$, we define ${\tt pr}_i(f)=\bigcup_{n\in\nn}{\tt pr}_i(f\res n)$ for each $i\in I$, and ${\tt mc}(f)=\lim_n{\tt mc}(f\res n)$, where if the limit does not exist, we write ${\tt mc}(f)=\infty$.
\end{definition}

\begin{definition}[Notations for Backtrack-Tape Models]
For any set $X$ and string $\sigma\in X^{<\nn}$, {\em the $n$-th shift} $\sigma^{\shft n}$ is defined as $\sigma^{\shft n}(m)=\sigma(n+m)$ for each $m<|\sigma|-n$.
\index{N-th shift@$n$-th shift}\index{${\shft n}$}%
The {\em tail of $\sigma$} is defined by 
\index{tail}\index{${\tt tall}$}%
\[{\tt tail}(\sigma)=\sigma^{\shft n}\mbox{, for }n=\min\{m\in\nn:\sigma(k)\not=\sharp\mbox{ for all }k\geq m\}.\]
Intuitively, the symbol $\sharp$ indicates the instruction to erase all letters written on the working tape.
Hence, the string ${\tt tail}(\sigma)$ extracts the remaining data from the record $\sigma$ after the latest erasing.
Furthermore, for $f\in X^{\nn}$, we define $f^{\shft n}=\bigcup_{m\geq n}(f\res m)^{\shft n}$, and ${\tt tail}(f)=\lim_m{\tt tail}(f\res m)$ if the limit exists.
Here, note that $\lim_m{\tt tail}(f\res m)$ exists if and only if $f$ contains only finitely many $\sharp$'s.
\end{definition}

\begin{example}
We consider two functions $\sigma\in(2\times\nn)^{<\nn}$ and $\tau\in(\nn\cup\{\sharp\})^{<\nn}$.
\begin{enumerate}
\item If $\sigma=\lrangle{\pair{1,3},\pair{1,1},\pair{0,4},\pair{0,15},\pair{1,9},\pair{0,26},\pair{0,5}}$, then the projections of $\sigma$ are ${\tt pr}_0(\sigma)=\lrangle{4,15,26,5}$, and ${\tt pr}_1(\sigma)=\lrangle{3,1,9}$.
Moreover, ${\tt mc}(\sigma)=3$.
\item If $\tau=\lrangle{0,2,7,18,28,\sharp,1,8,2,8,45,9,\sharp,0,4,52,35,3,6}$, then the tail of $\tau$ is ${\tt tail}(\tau)=\tau^{\shft 13}=\lrangle{0,4,52,35,3,6}$.
\end{enumerate}
\end{example}

\begin{definition}[One-Tape Disjunctions]\label{def:1-2b:onetape}
\index{disjunction!one-tape}\index{$\bhk{P_0\vee P_1}_{\sf Int}^1$}%
\index{$\bhk{P_0\vee P_1}_{\sf LCM}^1$}\index{$\bhk{P_0\vee P_1}_{\sf CL}^1$}%
Let $P_0$ and $P_1$ be subsets of Baire space $\nn^\nn$.
\begin{enumerate}
\item $\bhk{P_0\vee P_1}_{\sf Int}^1=\bigcup_{i<2}(\{i^\nn\}\lsup P_i)$.
\item $\bhk{P_0\vee P_1}_{\sf LCM}^1=\bigcup_{i<2}(\{f\in 2^\nn:(\forall^\infty n)\;f(n)=i\}\lsup P_i)$.
\item $\bhk{P_0\vee P_1}_{\sf CL}^1=\bigcup_{i<2}(2^\nn\lsup P_i)$.
\end{enumerate}
Here, $i^\nn$ denotes the infinite sequence consisting of $i$'s, i.e., $i^\nn=\lrangle{i,i,i,\dots,i,i,i,\dots}$.
\end{definition}

\begin{definition}[Two-Tape Disjunctions]\label{def:1-2b:twotape}
\index{disjunction!two-tape}\index{$\bhk{P_0\vee P_1}_{\sf Int}^2$}%
\index{$\bhk{P_0\vee P_1}^2_{\sf LCM}$}\index{$\bhk{P_0\vee P_1}_{\sf CL}^2$}%
Let $P_0$ and $P_1$ be subsets of Baire space $\nn^\nn$.
\begin{enumerate}
\item $\bhk{P_0\vee P_1}_{\sf Int}^2=\{f\in(2\times \nn)^\nn:((\exists i<2)\;{\tt pr}_i(f)\in P_i)\;\&\;{\tt mc}(f)=0\}$.
\item $\bhk{P_0\vee P_1}^2_{\sf LCM}=\{f\in(2\times \nn)^\nn:((\exists i<2)\;{\tt pr}_i(f)\in P_i)\;\&\;{\tt mc}(f)<\infty\}$.
\item $\bhk{P_0\vee P_1}_{\sf CL}^2=\{f\in(2\times \nn)^\nn:(\exists i<2)\;{\tt pr}_i(f)\in P_i\}$.
\end{enumerate}
\end{definition}

\begin{definition}[Backtrack Disjunctions]\label{def:12b:backtrack}
\index{disjunction!backtrack}\index{$\bhk{P_0\vee P_1}_{\sf Int}^3$}\index{$\bhk{P_0\vee P_1}_{\sf LCM}^3$}%
Let $P_0$ and $P_1$ be subsets of Baire space $\nn^\nn$.
\begin{enumerate}
\item $\bhk{P_0\vee P_1}_{\sf Int}^3=\{f\in(\nn\cup\{\sharp\})^\nn:{\tt tail}(f)^{\shft 1}\in P_{{\tt tail}(f;0)}\;\&\;(\forall n)\;f(n)\not=\sharp\}$.
\item $\bhk{P_0\vee P_1}_{\sf LCM}^3=\{f\in(\nn\cup\{\sharp\})^\nn:{\tt tail}(f)^{\shft 1}\in P_{{\tt tail}(f;0)}\;\&\;(\forall^\infty n)\;f(n)\not=\sharp\}$.
\end{enumerate}
\end{definition}

In Definition \ref{def:12b:backtrack}, for example, the string $\tau=\lrangle{\sharp}\fr\lrangle{i}\fr\sigma$ represents the record that a verifier $\Psi$ erased all letters from tapes (this action is indicated by $\sharp$), declared that $P_i$ is valid, and wrote the word $\sigma$ on the working tape.
That is to say, ${\tt tail}(\tau;0)=i$ is the current declaration of the verifier and  ${\tt tail}(\tau)^{\shft 1}=\sigma$ is the current word written on the working tape.

\begin{remark}
Note that we always have to choose a new symbol $\sharp$ which has not been already used, since we may need to distinguish the new $\sharp$ from other symbols and other $\sharp$'s used in other disjunctions.
Formally, we can assume that all objects in our paper are elements of $\nn^\nn$, subsets of $\nn^\nn$, or (partial) functions on $\nn^\nn$ by setting $\lrceil{0}=\sharp$, $\lrceil{(n+1)}=n$, and $\lrceil{f}(n)=f(\lrceil{n})$ for every $n\in\nn$.
For instance, $\bhk{P_0\vee P_1}_{\sf LCM}^3$ is always interpreted as the set $\bhk{P_0\vee P_1}_{{\sf LCM}}^{3\bullet}$ of all $f\in\nn^\nn$ such that $\lrceil{f}\in\bhk{P_0\vee P_1}_{\sf LCM}^3$, and then $\bhk{Q\vee \bhk{P_0\vee P_1}_{{\sf LCM}}^3}_{{\sf LCM}}^{3}$ is interpreted as $\bhk{Q\vee \bhk{P_0\vee P_1}_{{\sf LCM}}^{3\bullet}}_{{\sf LCM}}^{3\bullet}$ of all $f\in\nn^\nn$ such that $\lrceil{f}\in\bhk{P_0\vee P_1}_{\sf LCM}^3$.
Then, note that outer $\sharp$'s are automatically distinguished from inner $\sharp$'s contained in $f\in\bhk{Q\vee \bhk{P_0\vee P_1}_{{\sf LCM}}^{3\bullet}}_{{\sf LCM}}^{3\bullet}$.
Hereafter, $\bhk{P_0\vee P_1}_{\sf LCM}^3$ is identified with $\bhk{P_0\vee P_1}_{{\sf LCM}}^{3\bullet}$.
\end{remark}

\begin{notation}
Hereafter, we frequently use the notation ${\tt write}(i,\sigma)$ for any $i\in\nn$ and $\sigma\in\nn^{<\nn}$.
\index{${\tt write}$}%
\[{\tt write}(i,\sigma)=i^{|\sigma|}\oplus\sigma=\lrangle{\pair{i,\sigma(0)},\pair{i,\sigma(1)},\pair{i,\sigma(2)},\dots,\pair{i,\sigma(|\sigma|-1)}}.\]
This string indicates the {\em instruction to write the string $\sigma$ on the $i$-th tape} in the one/two-tape model.
We also use the notation ${\tt write}(i,f)=\bigcup_{n\in\nn}{\tt write}(i,f\res n)=i^\nn\oplus f$ for any $f\in\nn^{\nn}$.
\end{notation}

\begin{prop}\label{prop:1:basic}
Let $P$ and $Q$ be subsets of Baire space $\nn^\nn$.
\begin{enumerate}
\item $\bhk{P\vee P}^1_X\equiv^1_1P$ for each $X\in\{{\sf Int},{\sf LCM},{\sf CL}\}$.
\item $\bhk{P\vee Q}^i_{\sf CL}\leq^1_1\bhk{P\vee Q}^i_{\sf LCM}\leq^1_1\bhk{P\vee Q}^i_{\sf Int}$ for each $i\in\{1,2,3\}$ (except for ${\sf CL}$ if $i=3$).
\item $\bhk{P\vee Q}^i_{X}\leq^1_1\bhk{P\vee Q}^j_{X}$ for each $j\leq i$ and $X\in\{{\sf Int},{\sf LCM},{\sf CL}\}$.
\item $P\oplus Q\equiv^1_1\bhk{P\vee Q}^i_{\sf Int}$ for each $i\in\{1,2,3\}$.
\item $P\cup Q\equiv^1_1\bhk{P\vee Q}^1_{\sf CL}$.
\end{enumerate}
\end{prop}

\begin{proof}\upshape
(1) The reduction $f\oplus g\mapsto g$ witnesses $P\leq^1_1\bhk{P\vee P}^1_X$, and the reduction $f\mapsto {\tt write}(0,f)$ witnesses $\bhk{P\vee P}^1_X\leq^1_1P$, for each $X\in\{{\sf Int},{\sf LCM},{\sf CL}\}$.
Intuitively, ${\tt write}(0,f)$ indicates the instruction, in the one-tape model, to declare ``$P_0$ is correct'' at each stage and to write the infinite word $f$ on the tape $\Lambda$.

(2) Clearly, $\bhk{P\vee Q}^i_{\sf CL}\supseteq\bhk{P\vee Q}^i_{\sf LCM}\supseteq\bhk{P\vee Q}^i_{\sf Int}$ for each $i\in\{1,2,3\}$ (except for ${\sf CL}$ if $i=3$).

(3) Fix $X\in\{{\sf Int},{\sf LCM},{\sf CL}\}$.
We inductively construct a computable function $\Xi$ witnessing $\bhk{P\vee Q}^2_X\leq^1_1\bhk{P\vee Q}^1_X$.
First set $\Xi(\lrangle{})=\lrangle{}$, and assume that $\Xi(\sigma\oplus\tau)$ has been already defined for every strings $\sigma$ and $\tau$ of length $s$.
Then we now define $\Xi(\sigma\oplus\tau)$ for each strings $\sigma$ and $\tau$ of length $s+1$.
We inductively assume that ${\tt pr}_{i}(\Xi(\sigma^-\oplus\tau^-))\subseteq \tau^-$ for each $i<2$ (recall that $\sigma^-$ denotes the immediate predecessor of $\sigma$).
For $p=|{\tt pr}_{\sigma(s)}(\Xi(\sigma^-\oplus\tau^-))|$, we put $\Xi(\sigma\oplus\tau)=\Xi(\sigma^-\oplus\tau^-)\fr{\tt write}(\sigma(s),\tau^{\shft p})$.
Intuitively, this indicates the instruction to add some tail $\tau(p),\tau(p+1),\dots,\tau(s)$ to the word $\tau(0),\tau(1),\dots,\tau(p-1)$ written on the $\sigma(s)$-tape.
Then, we can inductively ensure the following condition.
\[{\tt pr}_{\sigma(s)}(\Xi(\sigma\oplus\tau))={\tt pr}_{\sigma(s)}(\Xi(\sigma^-\oplus\tau^-))\fr(\tau^{\shft p})=(\tau^-\res p)\fr \tau^{\shft p}=\tau.\]

Finally, we set $\Xi(f\oplus g)=\bigcup_{n\in\nn}\Xi((f\res n)\oplus(g\res n))$, for any $f,g\in\nn^\nn$.
Therefore, for any $f\oplus g\in\bhk{P\vee Q}^1_X$ and each $i<2$, if $f(n)=i$ for infinitely many $n\in\nn$, then ${\tt pr}_i(\Xi(f\oplus g))$ is total, and ${\tt pr}_i(\Xi(f\oplus g))=g$.
By definition, ${\tt pr}_i(\Xi(f\oplus g))=g\in P_i$ for some $i<2$.
Hence, $\Xi(f\oplus g)\in\bhk{P\vee Q}^2_X$.

Fix $X\in\{{\sf Int},{\sf LCM}\}$.
We inductively construct a computable function $\Xi$ witnessing $\bhk{P\vee Q}^3_X\leq^1_1\bhk{P\vee Q}^2_X$.
First set $\Xi(\lrangle{\pair{i,n}})=\lrangle{i,n}$ for each $\pair{i,n}\in 2\times\nn$.
Fix $\sigma=\sigma^{--}\fr\lrangle{\pair{i,m},\pair{j,n}}\in (2\times\nn)^{<\nn}$, and assume that $\Xi(\sigma^-)$ has been already defined.
Then, let us define $\Xi(\sigma)$ as follows:
\[
\Xi(\sigma^{--}\fr\lrangle{\pair{i,m},\pair{j,n}})=
\begin{cases}
\Xi(\sigma^{-})\fr\lrangle{n} & \text{ if }j=i;\\
\Xi(\sigma^{-})\fr\lrangle{\sharp,j}\fr{\tt pr}_{j}(\sigma) & \text{ otherwise.}
\end{cases}
\]
Finally set $\Xi(f)=\bigcup_n\Xi(f\res n)$, for any $f\in (2\times\nn)^\nn$.
It is easy to see that ${\tt tail}(f)$ is defined for any $f\in\bhk{P\vee Q}^2_X$, since $\#\{k\in\nn:\Xi(f;k)=\sharp\}={\tt mc}(f)$.
Therefore, ${\tt tail}^{\shft 1}(\Xi(f))\in P_{{\tt tail}(\Xi(f);0)}$.
If $X={\sf Int}$, then no $\sharp$ occurs in $\Xi(f)$.

(4) By definition, $\bhk{P\vee Q}^3_{\sf Int}=P\oplus Q$.
(5) The reduction $f\oplus g\mapsto g$ witnesses $P\cup Q\leq^1_1\bhk{P\vee Q}^1_{\sf CL}$, and the reduction $f\mapsto{\tt write}(0,f)=0^\nn\oplus f$ witnesses $\bhk{P\vee Q}^1_{\sf CL}\leq^1_1P\cup Q$.
\end{proof}

\begin{definition}
\index{disjunction!with a mind-changes-bound}\index{$\bhk{P_0\vee P_1}^1_{{\sf LCM}[n]}$}%
\index{$\bhk{P_0\vee P_1}^2_{{\sf LCM}[n]}$}\index{$\bhk{P_0\vee P_1}_{{\sf LCM}[n]}^3$}%
For each proof model, there are variations of {\sf LCM} disjunctions, for any {\em bound of mind changes}.
Let $P_0,P_1$ be any subsets of Baire space $\nn^\nn$, and $n$ be any natural number.
\begin{enumerate}
\item {\em The one-tape {\sf LCM} disjunction of $P_0$ and $P_1$ with mind-changes-bound $n$} is defined as follows.
\[\bhk{P_0\vee P_1}^1_{{\sf LCM}[n]}=\bhk{P_0\vee P_1}^1_{{\sf LCM}}\cap\{f\in 2^\nn:\#\{n\in\nn:f(n+1)\not=f(n)\}<n\}\otimes 2^\nn.\]
\item {\em The two-tape {\sf LCM} disjunction of $P_0$ and $P_1$ with mind-changes-bound $n$} is defined as follows.
\[\bhk{P_0\vee P_1}^2_{{\sf LCM}[n]}=\bhk{P_0\vee P_1}^2_{{\sf LCM}}\cap\{f\in(2\times \nn)^\nn:{\tt mc}(f)<n\}.\]
\item {\em The backtrack-tape {\sf LCM} disjunction of $P_0$ and $P_1$ with mind-changes-bound $n$} is defined as follows.
\[\bhk{P_0\vee P_1}_{{\sf LCM}[n]}^3=\bhk{P_0\vee P_1}^3_{{\sf LCM}}\cap\{f\in(\nn\cup\{\sharp\})^\nn:\#\{k\in\nn:f(k)=\sharp\}<n\}.\]
\end{enumerate}
\end{definition}

\begin{prop}\label{prop:1-2b:cnsivee}
Let $P,Q$ be subsets of Baire space $\nn^\nn$.
\begin{enumerate}
\item $P\oplus Q\equiv^1_1\bhk{P\vee Q}^i_{{\sf LCM}[1]}$ for each $i\in\{1,2,3\}$.
\item $\bhk{P\vee P}^2_{{\sf LCM}[2]}\equiv^1_1\bhk{P\vee P}^3_{{\sf LCM}[2]}$.
Indeed, $\bhk{\bigvee_{i<n}P_i}^2_{{\sf LCM}[n]}\equiv^1_1\bhk{P\vee P}^3_{{\sf LCM}[n]}$, where $P_i=P$ for each $i<n$.
Here, for each collection $\{P_i\}_{i<k}$ of subsets of Baire space, $\bhk{\bigvee_{i<k}P_i}^2_{{\sf LCM}[n]}$ is defined as follows.
\[\{f\in(k\times \nn)^\nn:((\exists i<k)\;{\tt pr}_i(f)\in P_i)\;\&\;{\tt mc}(f)<n\}.\]
\end{enumerate}
\end{prop}

\begin{proof}\upshape
(1) Clearly $\bhk{P\vee Q}^i_{\sf LCM[1]}=\bhk{P\vee Q}^i_{\sf Int}$ for each $i\in\{1,2,3\}$.
By Proposition \ref{prop:1:basic} (4), we have $P\oplus Q\equiv^1_1\bhk{P\vee Q}^i_{\sf Int}$.

(2) The reduction $\Xi:h\mapsto h^*$ in the proof of Proposition \ref{prop:1:basic} (3) also witnesses $\bhk{P\vee P}^3_{{\sf LCM}[n]}\leq^1_1\bhk{\bigvee_{i<n}P_i}^2_{{\sf LCM}[n]}$.
We inductively define a computable function $\Xi^*$ witnessing $\bhk{\bigvee_{i<n}P_i}^2_{{\sf LCM}[n]}\leq^1_1\bhk{P\vee P}^3_{{\sf LCM}[n]}$.
Put $\Xi^*(\lrangle{})=\lrangle{}$, and fix $\sigma=\sigma^-\fr\lrangle{k}\in(\nn\cup\{\sharp\})^{<\nn}$.
Assume that $\Xi^*(\sigma^-)$ has been already defined.
Then, $\Xi^*(\sigma)$ is defined as follows.
\begin{align*}
{\tt count}(\sigma)&=\#\{m<|\sigma|:\sigma(m)=\sharp\},\\
\Xi^*(\sigma^-\fr\lrangle{k})&=
\begin{cases}
\Xi^*(\sigma^-)\fr\lrangle{({\tt count}(\sigma),k)} & \mbox{ if } k\not=\sharp,\\
\Xi^*(\sigma^-) & \mbox{ otherwise.}
\end{cases}
\end{align*}
For any $g\in\bhk{P\vee P}^3_{{\sf LCM}[n]}$, we have ${\tt count}(g\res s)<n$ for any $s\in\nn$, and hence ${\tt mc}(\Xi^*(g))<n$, since $g$ contains at most $n$ many $\sharp$'s.
Moreover, ${\tt pr}_{\lim_s{\tt count}(g\res s)}(\Xi^*(g))={\tt tail}(g)^{\shft 1}\in P$.
\end{proof}

\begin{prop}\label{prop:1-2:wellbehaved}
Let $P_0$, $P_1$, $Q_0$, and $Q_1$ be subsets of Baire space $\nn^\nn$, and fix $i\in\{2,3\}$ and $X\in\{{\sf Int},{\sf LCM},{\sf CL}\}\cup\{{\sf LCM}[n]:n\in\nn\}$.
If $P_0\leq^1_1Q_0$ and $P_1\leq^1_1Q_1$, then $\bhk{P_0\vee P_1}^i_X\leq^1_1\bhk{Q_0\vee Q_1}^i_X$.
Hence, the operator $\mathbf{D}^i_X:\mathcal{D}^1_1\times\mathcal{D}^1_1\to\mathcal{D}^1_1$ introduced by $\mathbf{D}^i_X(\deg^1_1(P),\deg^1_1(Q))=\deg^1_1(\bhk{P\vee Q}^i_X)$ is well-defined.
Here, $\deg^1_1(P)$ denotes the equivalent class $\{R\subseteq\nn^\nn:R\equiv^1_1P\}$.
\end{prop}

\begin{proof}\upshape
We first consider the two-tape model.
Assume that $P_0\leq^1_1Q_0$ and $P_1\leq^1_1Q_1$ via computable functions $\Gamma_0$ and $\Gamma_1$, respectively.
We construct a computable function $\Delta$ witnessing $\bhk{P_0\vee P_1}^2_X\leq^1_1\bhk{Q_0\vee Q_1}^2_X$.
Set $\Delta(\lrangle{})=\lrangle{}$.
Fix $\sigma\in(2\times\nn)^{<\nn}$ and assume that $\Delta(\sigma^-)$ has been already defined.
For each $i<2$, we define ${\tt new}\Gamma_i({\tt pr}_i(\sigma))\in\nn^{<\nn}$ by the unique string such that $\Gamma_i({\tt pr}_i(\sigma))=\Gamma_i({\tt pr}_i(\sigma^-))\fr{\tt new}\Gamma_i({\tt pr}_i(\sigma))$.
Then we define $\Delta(\sigma)$ as follows.
\[\Delta(\sigma)=\Delta(\sigma^-)\fr{\tt write}(0,{\tt new}\Gamma_0({\tt pr}_0(\sigma)))\fr{\tt write}(1,{\tt new}\Gamma_1({\tt pr}_1(\sigma))).\]

Note that ${\tt new}\Gamma_i({\tt pr}_i(\sigma))=\lrangle{}$ for some $i<2$, since ${\tt pr}_i(\sigma)={\tt pr}_i(\sigma^-)$ for either $i<2$.
Therefore, ${\tt mc}(\Delta(g))={\tt mc}(g)$ for any $g\in\nn^\nn$.
Furthermore, for any $g\in\nn^\nn$, we have ${\tt pr}_i(\Delta(g))=\Gamma_i({\tt pr}_i(g))$ for each $i<2$.
Thus, $\Delta(g)\in\bhk{P_0\vee P_1}^2_X$ for any $g\in\bhk{Q_0\vee Q_1}^2_X$.

Next we consider the backtrack-tape model.
Assume that $P_0\leq^1_1Q_0$ and $P_1\leq^1_1Q_1$ via computable functions $\Gamma_0$ and $\Gamma_1$, respectively.
We construct a computable function $\Theta$ witnessing $\bhk{P_0\vee P_1}^3_X\leq^1_1\bhk{Q_0\vee Q_1}^3_X$.
Set $\Theta(\lrangle{})=\lrangle{}$.
Fix $\sigma\in(\nn\cup\{\sharp\})^{<\nn}$ and assume that $\Theta(\tau)$ has been already defined for each $\tau\subsetneq\sigma$.
If $\sigma=\sigma^{--}\fr\lrangle{m,n}$ for some $m,n\in\nn$, then we have $\Gamma_{{\tt tail}(\sigma;0)}({\tt tail}(\sigma)^{\shft 1})=\Gamma_{{\tt tail}(\sigma;0)}({\tt tail}(\sigma)^{\shft 1})\fr\eta$ for some $\eta\in\nn^{<\nn}$, and we define $\Theta(\sigma)=\Theta(\sigma^-)\fr\eta$.
If $\sigma=\sigma^{--}\fr\lrangle{\sharp,i}$ for some $i<2$, i.e., ${\tt tail}(\sigma;0)=i$, then define $\Theta(\sigma)=\Theta(\sigma^-)\fr\lrangle{\sharp,i}$.
Otherwise, we set $\Theta(\sigma)=\Theta(\sigma^-)$.
Note that $\#\{n\in\nn:\Theta(g;n)=\sharp\}=\#\{n\in\nn:g(n)=\sharp\}$ for any $g\in\nn^\nn$.
Furthermore, ${\tt tail}(\Theta(g);0)={\tt tail}(g;0)$, and ${\tt tail}(\Theta(g))^{\shft 1}=\Gamma_{{\tt tail}(g;0)}({\tt tail}(g)^{\shft 1})$ for any $g\in\bhk{Q_0\vee Q_1}^3_X$.
Hence, $\Theta(g)\in\bhk{P_0\vee P_1}^3_X$ for any $g\in\bhk{Q_0\vee Q_1}^3_X$.
\end{proof}

\begin{table}\label{intuitable}\small
\begin{center}
\begin{tabular}{ccccccc}
$P\cup Q$ & & & & & & $P\oplus Q$ \\
\rotatebox[origin=c]{90}{$\equiv$} & & & & & & \rotatebox[origin=c]{90}{$\equiv$} \\
$\bhk{P\vee Q}^1_{\sf CL}$ & $\leq$ ($\equiv$) & $\bhk{P\vee Q}^1_{\sf LCM}$ & $\leq$ ($\equiv$) & $\bhk{P\vee Q}^1_{\sf LCM[2]}$ & $\leq$ ($\equiv$) & $\bhk{P\vee Q}^1_{\sf Int}$ \\
\rotatebox[origin=c]{90}{$\leq$} & & \rotatebox[origin=c]{90}{$\leq$} & & \rotatebox[origin=c]{90}{$\leq$} & & \rotatebox[origin=c]{90}{$\equiv$} \\
$\bhk{P\vee Q}^2_{\sf CL}$ & $\leq$ & $\bhk{P\vee Q}^2_{\sf LCM}$ & $\leq$ & $\bhk{P\vee Q}^2_{\sf LCM[2]}$ & $\leq$ & $\bhk{P\vee Q}^2_{\sf Int}$ \\
& & \rotatebox[origin=c]{90}{$\leq$} & & \rotatebox[origin=c]{90}{$\leq$} (\rotatebox[origin=c]{90}{$\equiv$}) & & \rotatebox[origin=c]{90}{$\equiv$} \\
& & $\bhk{P\vee Q}^3_{\sf LCM}$ & $\leq$ & $\bhk{P\vee Q}^3_{\sf LCM[2]}$ & $\leq$ & $\bhk{P\vee Q}^3_{\sf Int}$ 
\end{tabular}
\end{center}
\caption{Degrees of difficulty of disjunctions, where $\leq$ and $\equiv$ denote the Medvedev reducibility and equivalence, and ($\equiv$) denotes the Medvedev equivalence when $P=Q$}%
\end{table}

\begin{remark}
Though the original limit-BHK interpretation of the disjunctive notion seems to be a one-tape notion, we will observe that the two-tape notions and the backtrack notions exhibit amazing and fascinating behaviors as operations on the subsets of Baire space.
While the one-tape models are almost static, the two-tape models can be understood as learning proof models with {\em bounded-errors}, and the backtrack tape models can be understood as learning proof models with no predetermined bound for errors.
In Part II, we adopt the two-tape notions except for the classical one-tape disjunction $\cup$, since the two-tape notions (the bounded-errors learning models) are useful to clarify differences among the classes $[\mathfrak{C}_T]^1_1,[\mathfrak{C}_T]^1_{<\omega},[\mathfrak{C}_{T}]^1_{\omega|<\omega},[\mathfrak{C}_T]^{<\omega}_1$ which are defined (as certain classes of bounded-errors functions) later.
In Part II, we also adopt dynamic generalizations of the backtrack tape models since such models turn out to be a strong tool to establish many theorems.
\end{remark}

%% file: NRMP_fullproof1-2c.tex
\section{Galois Connection}

\subsection{Decomposing Disjunction by Piecewise Computable Functions}

The main theorem in this section (Theorem \ref{theorem:12c:charact}) states that our degree structures $\mathcal{D}^\alpha_{\beta|\gamma}$ (Definition \ref{def:1-2a:degree}) are completely characterized by the disjunction operations (Definitions \ref{def:1-2b:onetape}, \ref{def:1-2b:twotape}, and \ref{def:12b:backtrack}).

\begin{prop}[Untangling]\label{prop:1-5:inf}
Let $P,Q$ be subsets of Baire space $\nn^\nn$.
\begin{enumerate}
\item There is a $(1,n|2)$-truth-table function $\Gamma:\bhk{P\vee Q}_{{\sf LCM}[n]}^1\to P\oplus Q$.
\item There is a $(1,n|2)$-computable function $\Gamma:\bhk{P\vee Q}_{{\sf LCM}[n]}^2\to P\oplus Q$.
\item There is a $(1,n)$-computable function $\Gamma:\bhk{P\vee Q}_{{\sf LCM}[n]}^3\to P\oplus Q$.
\item There is a $(1,\omega|2)$-truth-table function $\Gamma:\bhk{P\vee Q}_{{\sf LCM}}^1\to P\oplus Q$.
\item There is a $(1,\omega|2)$-computable function $\Gamma:\bhk{P\vee Q}_{{\sf LCM}}^2\to P\oplus Q$.
\item There is a $(1,\omega)$-computable function $\Gamma:\bhk{P\vee Q}_{{\sf LCM}}^3\to P\oplus Q$.
\item There is a $(2,1)$-truth-table function $\Gamma:\bhk{P\vee Q}_{{\sf CL}}^1\to P\oplus Q$.
\item There is a $(2,1)$-computable function $\Gamma:\bhk{P\vee Q}_{{\sf CL}}^2\to P\oplus Q$.
\end{enumerate}
\end{prop}

\begin{proof}\upshape
For the items (1), (4), and (7), we consider the truth-table functionals $\Delta_0:f\oplus g\mapsto 0\fr g$ and $\Delta_1:f\oplus g\mapsto 1\fr g$.
By the definition of $\bhk{P\vee Q}^1_{{\sf CL}}$, obviously $\Delta_0(f\oplus g)\in P\oplus Q$ or $\Delta_1(f\oplus g)\in P\oplus Q$ for any $f\oplus g\in\bhk{P\vee Q}^1_{\sf CL}$.
Let $e_0$ and $e_1$ be indices of $\Delta_0$ and $\Delta_1$, respectively.
On $\sigma\oplus\tau\in(2\times\nn)^{<\nn}$, we set $\Psi(\sigma\oplus\tau)=e_{\sigma(|\sigma|-1)}$.
Note that the partial function $\Gamma$ identified by the learner $\Psi$ is $(1,n|2)$-truth-table on $\bhk{P\vee Q}^1_{{\sf LCM[n]}}$, and $(1,\omega|2)$-truth-table on $\bhk{P\vee Q}^1_{{\sf LCM}}$.
Moreover, clearly $\Gamma(f\oplus g)=(\lim_sf(s))\fr g\in P\oplus Q$ for every $f\oplus g\in\bhk{P\vee Q}^1_{{\sf LCM}}$.

For the items (2), (5), and (8), we consider the partial computable functions $\Delta_0:f\mapsto 0\fr{\tt pr}_0(f)$ and $\Delta_1:f\mapsto 1\fr{\tt pr}_1(f)$.
By the definition of $\bhk{P\vee Q}^2_{{\sf CL}}$, obviously $\Delta_0(f)\in P\oplus Q$ or $\Delta_1(f)\in P\oplus Q$ for any $f\in\bhk{P\vee Q}^2_{\sf CL}$.
Let $e_0$ and $e_1$ be indices of $\Delta_0$ and $\Delta_1$, respectively.
On $\sigma\in(2\times\nn)^{<\nn}$, we set $\Psi(\sigma)=e_{(\sigma(|\sigma|-1))_0}$.
Note that the partial function $\Gamma$ identified by the learner $\Psi$ is $(1,n|2)$-computable on $\bhk{P\vee Q}^2_{{\sf LCM[n]}}$, and $(1,\omega|2)$-computable on $\bhk{P\vee Q}^2_{{\sf LCM}}$.
Moreover, clearly $\Gamma(f)\in P\oplus Q$ for every $f\in\bhk{P\vee Q}^2_{{\sf LCM}}$.

For the items (3) and (6), on $\sigma\in(\nn\cup\{\sharp\})^{<\nn}$, $\Psi(\sigma)$ guesses an index of the partial computable function $g\mapsto g^{\shft t(\sigma)}$, where $t(\sigma)=\max\{n:\sigma(n)=\sharp\}+1$ if such $n$ exists; otherwise, $t(\sigma)=0$.
Note that the partial function $\Gamma$ identified by the learner $\Psi$ is $(1,n)$-computable on $\bhk{P\vee Q}^3_{{\sf LCM[n]}}$, and $(1,\omega)$-computable on $\bhk{P\vee Q}^3_{{\sf LCM}}$.
Moreover, clearly $\Gamma(f)\in P\oplus Q$ for every $f\in\bhk{P\vee Q}^3_{{\sf LCM}}$.
\end{proof}

\begin{notation}
One can iterate two-tape disjunction operations as $\bhk{\bigvee^{(1)}P}^2_X=P$, and $\bhk{\bigvee^{(n+1)}P}^2_X=\bhk{P\vee\bhk{\bigvee^{(n)}P}^2_X}^2_X$.
Then, for instance, $\bhk{\bigvee^{(n)}P}^2_{\sf LCM}$ can be identified with the following subset of Baire space. 
\index{$\bhk{\bigvee^{(n)}P}^2_{\sf LCM}$}%
\[\{f\in (n\times\nn)^\nn:((\exists i<n)\;{\tt pr}_i(f)\in P)\;\&\;{\tt mc}(f)<\infty\}.\]

As in the proof of Proposition \ref{prop:1-2:wellbehaved}, we use the notation ${\tt new}\Gamma(\sigma)$ for any function $\Gamma:\nn^{<\nn}\to\nn^{<\nn}$ and $\sigma\in\nn^{<\nn}$ in the proof of the next theorem.
Here, ${\tt new}\Gamma(\sigma)$ is the unique string that satisfies the following condition.
\[\Gamma(\sigma)=\Gamma(\sigma^-)\fr{\tt new}\Gamma(\sigma).\]
\end{notation}

\begin{theorem}\label{theorem:12c:charact}
Let $P$ and $Q$ be any subsets of Baire space $\nn^\nn$.
\begin{enumerate}
\item $P\leq^1_{<\omega}Q$ if and only if $\bhk{P\vee P}^{3}_{{\sf LCM}[m]}\leq^1_1 Q$ for some $m\in\nn$.
\item $P\leq^1_{\omega|<\omega}Q$ if and only if $\bhk{\bigvee^{(m)}P}^{2}_{{\sf LCM}}\leq^1_1 Q$ for some $m\in\nn$.
\item $P\leq^1_\omega Q$ if and only if $\bhk{P\vee P}^{3}_{{\sf LCM}}\leq^1_1 Q$.
\item $P\leq^{<\omega}_1Q$ if and only if $\bhk{\bigvee^{(m)}P}^{2}_{{\sf CL}}\leq^1_1 Q$ for some $m\in\nn$.
\item $P\leq^{<\omega}_\omega Q$ if and only if $\bhk{\bigvee^{(m)}\bhk{P\vee P}^{3}_{{\sf LCM}}}^{2}_{{\sf CL}}\leq^1_1 Q$.
\item $P\leq^{\omega}_1 Q$ if and only if $\bigcup_{m\in\nn}\bhk{\bigvee^{(m)}P}^{2}_{{\sf CL}}\leq^1_1 Q$.
\end{enumerate}
\end{theorem}

\begin{proof}
The ``if'' parts of all items follow from Proposition \ref{prop:1-5:inf}.
We show the ``only if'' part for every item.

(1)
Assume that $P\leq^1_{<\omega}Q$ via a learner $\Psi$ with mind-change-bound $n$.
We need to construct a computable function $\Delta$ witnessing $\bhk{P\vee P}^{3}_{{\sf LCM}[n]}\leq^1_1Q$.
For any $g\in Q$, by uniformly computable procedure, we can enumerate all elements of ${\tt mcl}_\Psi(g)$ as $m^g_0,m^g_1,\dots,m^g_{k-1}$, where $k<n$.
Then, we define $\Delta(g)$ as follows.
\[\Delta(g)=0\fr\Phi_{\Psi(\lrangle{})}(g\res m^g_0)\fr\sharp
\fr 0\fr\left(\concat_{j<k-1}\Phi_{\Psi(g\res m^g_i+1)}(g\res m^g_i)\fr\sharp\fr 0\right)\fr\Phi_{\Psi(g\res m^g_{k-1}+1)}(g).\]
It is easy to see that $\Delta$ is computable.
Note that ${\tt tail}(\Delta(g))=\Phi_{\Psi(g\res m^g_{k-1}+1)}(g)\in P$, since $P\leq^1_{<\omega}Q$ via $\Psi$, and $\lim_s\Psi(g\res s)$ converges to $\Psi(g\res m^g_{k-1}+1)$.
Furthermore, $\sharp$ occurs $k$ times in $\Delta(g)$, and $k<n$ because of mind-change-bound $n$.
Thus, $\Delta(g)\in\bhk{P\vee P}^{3}_{{\sf LCM}[n]}$ for any $g\in Q$, as desired.

\medskip

(2)
Assume that $P\leq^1_{\omega|<\omega}Q$ via a leaner $\Psi$, where $\#{\tt indx}_\Psi(g)<n$ for any $g\in Q$.
We need to construct a computable function $\Delta$ witnessing $\bhk{\bigvee^{(n)}P}^{2}_{{\sf LCM}}\leq^1_1Q$.
We again use the function ${\tt reindex}_\Psi:\nn^{<\nn}\to\nn$ defined in the proof of Theorem \ref{thm:5:red-eq-dis} (2).
Fix $\sigma\in\nn^{<\nn}$.
Pick the greatest substring $\tau\subsetneq\sigma$ such that $\Psi(\tau)=\Psi(\sigma)$.
Then, define ${\tt new}^*\Phi_{\Psi(\sigma)}(\sigma)$ by the unique $\eta$ such that $\Phi_{\Psi(\sigma)}(\sigma)=\Phi_{\Psi(\sigma)}(\tau)\fr\eta$.
Here, if there is no such $\tau$, then we define ${\tt new}^*\Phi_{\Psi(\sigma)}(\sigma)=\Phi_{\Psi(\sigma)}(\sigma)$.
Assume that $\Delta(\sigma^-)$ has been already defined.
Then, we define $\Delta(\sigma)$ as follows.
\[\Delta(\sigma)=\Delta(\sigma^-)\fr{\tt write}({\tt reindex}_\Psi(\sigma),{\tt new}^*\Phi_{\Psi(\sigma)}(\sigma)).\]

Fix $g\in Q$.
Note that ${\tt reindex}_\Psi(g\res s)<n$ for each $s\in\nn$, since $\#{\tt indx}_\Psi(g)<n$.
Thus, we have $\Delta(g)\in(n\times\nn)^\nn$.
Moreover, ${\tt mc}(\Delta(g))<\infty$, since $\Psi$ is a learner converging on $Q$.
Thus, $\lim_s\Psi(g\res s)$ and hence $\lim_s{\tt reindex}_\Psi(g\res s)$ converge.
Therefore, ${\tt pr}_{\lim_s{\tt reindex}_\Psi(g\res s)}(\Delta(g))=\Phi_{\lim_s\Psi(g\res s)}(g)\in P$.
Hence, $\bhk{\bigvee^{(n)}P}^{2}_{{\sf LCM}}\leq^1_1Q$.

\medskip

(3)
By similar argument used in proof of (1).

\medskip

(4)
Assume that $P\leq^{<\omega}_1Q$ via a finite collection $\{\Phi_e\}_{e<n}$ of partial computable functions.
We need to construct a computable function $\Delta$ witnessing $\bhk{\bigvee^{(n)}P}^{2}_{{\sf CL}}\leq^1_1Q$.
Assume that $\Delta(\sigma^-)$ is already defined.
Define $\Delta(\sigma)$ as follows.
\[\Delta(\sigma)=\Delta(\sigma^-)\fr\concat_{e<n}{\tt write}(e,{\tt new}\Phi_e(\sigma)).\]

Note that ${\tt pr}_e(\Delta(g))\in P$ if $\Phi_e(g)\in P$.
Thus, for any $g\in Q$, we have ${\tt pr}_e(\Delta(g))\in P$ for some $e<n$.
In other words, $\bhk{\bigvee^{(m)}P}^{2}_{{\sf CL}}\leq^1_1Q$ via $\Delta$.

\medskip

(5)
Assume that $P\leq^{<\omega}_{\omega}$ via a team $\{\Psi_i\}_{i<n}$ of learners.
We construct a computable function $\Delta$.
We first set $\Delta(\lrangle{})=\lrangle{}$.
Fix $\sigma\in\nn^{<\nn}$, and assume that $\Delta(\sigma^-)$ has been already defined.
We define $\eta^\sigma_i\in\nn^{<\nn}$ for each $i<n$ as follows.
Fix $i<n$.
If $\Psi_i(\sigma)=\Psi_i(\sigma^-)$, put ${\tt new}^{**}\Phi_{\Psi_i(\sigma)}(\sigma)={\tt new}\Phi_{\Psi_i(\sigma)}(\sigma)$.
If $\Psi_i(\sigma)\not=\Psi_i(\sigma^-)$, put ${\tt new}^{**}\Phi_{\Psi_i(\sigma)}(\sigma)=\sharp\fr\Phi_{\Psi_i(\sigma)}(\sigma)$.
Then, we define $\Delta(\sigma)$ as follows.
\[\Delta(\sigma)=\Delta(\sigma^-)\fr\concat_{i<n}{\tt write}(i,{\tt new}^{**}\Phi_{\Psi_i(\sigma)}(\sigma)).\]

Pick $g\in Q$.
Then, by our assumption, $\Phi_{\lim_n\Psi_i(g\res n)}(g)\in P$ for some $i<b$.
Then ${\tt tail}({\tt pr}_i(\Delta(g)))$ converges, and ${\tt tail}({\tt pr}_i(\Delta(g)))^{\shft 1}=\Phi_{\lim_n\Psi_i(g\res n)}(g)\in P$.
Thus, $\Delta(g)\in\bhk{\bigvee^{(m)}\bhk{P\vee P}^{3}_{{\sf LCM}}}^{2}_{{\sf CL}}$.

\medskip

(6)
Assume that $P\leq^{\omega}_1Q$.
We need to construct a computable function $\Delta$ witnessing $\bigcup_{m\in\nn}\bhk{\bigvee^{(m)}P}^{2}_{{\sf CL}}\leq^1_1Q$.
Assume that $\Delta(\sigma^-)$ has been already defined.
Define $\Delta(\sigma)$ as follows. 
\[\Delta(\sigma)=\Delta(\sigma^-)\fr\left(\concat_{e<|\sigma|}{\tt write}(e,{\tt new}\Phi_e(\sigma))\right)\fr({\tt write}(|\sigma|,\Phi_{|\sigma|}(\sigma))).\]

Note that ${\tt pr}_e(\Delta(g))=\Phi_e(g)$.
Thus, for any $g\in Q$, we have ${\tt pr}_e(\Delta(g))\in P$ for some $e\in\nn$.
In other words, $\bigcup_{m\in\nn}\bhk{\bigvee^{(m)}P}^{2}_{{\sf CL}}\leq^1_1Q$ via $\Delta$.
\end{proof}

\begin{remark}
Given an operation $O:\mathcal{P}(\nn^\nn)\times\mathcal{P}(\nn^\nn)\to\mathcal{P}(\nn^\nn)$, one can introduce the reducibility notion $\leq_O$ by defining $P\leq_OQ$ as $O^{(n)}(P)\leq^1_1Q$ for some $n\in\nn$, where $O^{(1)}(P)=P$ and $O^{(n+1)}(P)=O(P, O^{(n)}(P))$.
\index{$P\leq_OQ$}%
Then, Theorem \ref{theorem:12c:charact} indicates that our reducibility notions induced by seven monoids in Theorem \ref{thm:main:first1-2} are also induced from corresponding disjunction operations.
\end{remark}

\subsection{Galois Connection between Degree Structures}

\begin{remark}
For degree structures $\mathcal{D}_u$ and $\mathcal{D}_r$ on $\mathcal{P}(\nn^\nn)$, each operator $O:\mathcal{P}(\nn^\nn)\to\mathcal{P}(\nn^\nn)$ induces the new operator $O_{ur}:\mathcal{D}_u\to\mathcal{D}_r$ defined by $O_{ur}(\deg_u(P))=\deg_r(O(P))$ for any $P\subseteq\nn^\nn$.
We identify $O$ with $O_{ur}$ whenever $O_{ur}$ is well-defined.
\index{$O_{ur}$}%
Recall that every partially ordered set can be viewed as a category.
Sorbi \cite{Sor} showed that $\widehat{\rm Deg}:\mathcal{D}^\omega_1\to\mathcal{D}^1_1$ is left-adjoint to ${\rm id}:\mathcal{D}^1_1\to\mathcal{D}^\omega_1$, and ${\rm id}\circ\widehat{\rm Deg}:\mathcal{D}^\omega_1\to\mathcal{D}^\omega_1$ is identity, where $\widehat{\rm Deg}(P)$ denotes the Turing upward closure of $P\subseteq\nn^\nn$.
\index{$\widehat{\rm Deg}(P)$}%
\end{remark}

\begin{definition}~
\index{$\mathbb{V}^1_{\tt eff}(P)$}\index{$\mathbb{V}^1_{\omega\mid{\tt eff}}(P)$}%
\index{$\mathbb{V}^1_{\omega}(P)$}\index{$\mathbb{V}^{\tt eff}_{1}(P)$}%
\index{$\mathbb{V}^{\tt eff}_\omega(P)$}\index{$\mathbb{V}^{\omega}_1(P)$}%
\begin{enumerate}
\item $\mathbb{V}^1_{\tt eff}(P)=\bigoplus_{m\in\nn}\bhk{P\vee P}^{3}_{{\sf LCM}[m]}$.
\item $\mathbb{V}^1_{\omega|\tt eff}(P)=\bigoplus_{m\in\nn}\bhk{\bigvee^{(m)}P}^{2}_{{\sf LCM}}$.
\item $\mathbb{V}^1_{\omega}(P)=\bhk{P\vee P}^{3}_{{\sf LCM}}$.
\item $\mathbb{V}^{\tt eff}_{1}(P)=\bigoplus_{m\in\nn}\bhk{\bigvee^{(m)}P}^{2}_{{\sf CL}}$.
\item $\mathbb{V}^{\tt eff}_\omega(P)=\bigoplus_{m\in\nn}\bhk{\bigvee^{(m)}\bhk{P\vee P}^{3}_{{\sf LCM}}}^{2}_{{\sf CL}}$.
\item $\mathbb{V}^{\omega}_1(P)=\bigcup_{m\in\nn}\bhk{\bigvee^{(m)}P}^{2}_{{\sf CL}}$.
\end{enumerate}
\end{definition}

\begin{cor}~
\begin{enumerate}
\item $\mathbb{V}^1_{\tt eff}:\mathcal{D}^1_{\tt eff}\to\mathcal{D}^1_1$ is left-adjoint to ${\rm id}_{\mathcal{P}(\nn^\nn)}:\mathcal{D}^1_1\to\mathcal{D}^1_{\tt eff}$, and ${\rm id}_{\mathcal{P}(\nn^\nn)}\circ\mathbb{V}^1_{\tt eff}$ is the identity on $\mathcal{D}^1_{\tt eff}$.
\item $\mathbb{V}^1_{\omega|{\tt eff}}:\mathcal{D}^1_{\omega|{\tt eff}}\to\mathcal{D}^1_1$ is left-adjoint to ${\rm id}_{\mathcal{P}(\nn^\nn)}:\mathcal{D}^1_1\to\mathcal{D}^1_{\omega|{\tt eff}}$, and ${\rm id}_{\mathcal{P}(\nn^\nn)}\circ\mathbb{V}^1_{\omega|{\tt eff}}$ is the identity on $\mathcal{D}^1_{\omega|{\tt eff}}$.
\item $\mathbb{V}^1_{\omega}:\mathcal{D}^1_{\omega}\to\mathcal{D}^1_1$ is left-adjoint to ${\rm id}_{\mathcal{P}(\nn^\nn)}:\mathcal{D}^1_1\to\mathcal{D}^1_{\omega}$, and ${\rm id}_{\mathcal{P}(\nn^\nn)}\circ\mathbb{V}^1_{\omega}$ is the identity on $\mathcal{D}^1_{\omega}$.
\item $\mathbb{V}^{\tt eff}_1:\mathcal{D}^{\tt eff}_1\to\mathcal{D}^1_1$ is left-adjoint to ${\rm id}_{\mathcal{P}(\nn^\nn)}:\mathcal{D}^1_1\to\mathcal{D}^{\tt eff}_1$, and ${\rm id}_{\mathcal{P}(\nn^\nn)}\circ\mathbb{V}^{\tt eff}_1$ is the identity on $\mathcal{D}^{\tt eff}_1$.
\item $\mathbb{V}^{\tt eff}_\omega:\mathcal{D}^{\tt eff}_\omega\to\mathcal{D}^1_1$ is left-adjoint to ${\rm id}_{\mathcal{P}(\nn^\nn)}:\mathcal{D}^1_1\to\mathcal{D}^{\tt eff}_\omega$, and ${\rm id}_{\mathcal{P}(\nn^\nn)}\circ\mathbb{V}^{\tt eff}_\omega$ is the identity on $\mathcal{D}^{\tt eff}_\omega$.
\item $\mathbb{V}^\omega_1:\mathcal{D}^\omega_1\to\mathcal{D}^1_1$ is left-adjoint to ${\rm id}_{\mathcal{P}(\nn^\nn)}:\mathcal{D}^1_1\to\mathcal{D}^\omega_1$, and ${\rm id}_{\mathcal{P}(\nn^\nn)}\circ\mathbb{V}^\omega_1$ is the identity on $\mathcal{D}^\omega_1$.
\end{enumerate}
\end{cor}

\begin{proof}\upshape
By Theorem \ref{thm:5:red-eq-dis}.
\end{proof}

\subsection{$\Sigma^0_2$ Decompositions}

In computability theory, we sometimes encounter conditional branching given by a $\Sigma^0_2$ formula $S\equiv \exists n\tilde{S}(n)$.
That is, if $S$ is true, one chooses a procedure $p_1$, and if $S$ is false, one chooses another procedure $p_2$.
Thus, one may define {\em the computability with a $\Sigma^0_2$ conditional branching} as the class ${\rm dec}^2_{\rm d}[\Pi^0_2]$.
However, even if we know that $S$ is true, we have no algorithm to find a witness of $S$ since $\tilde{S}(n)$ is $\Pi^0_1$, while we sometimes require a witness of $S$.
This observation motivates us to study a missing interesting subclass of the nonuniformly computable functions.

\begin{prop}\label{prop:1-2c:dec-sigma2-lem}
${\rm dec}^{<\omega}_{\rm d}[\Pi^0_2]{\rm dec}^\omega_{\rm p}[\Pi^0_1]$ is the smallest monoid including ${\rm dec}^2_{\rm d}[\Pi^0_2]$ and ${\rm dec}^\omega_{\rm p}[\Pi^0_1]$.
\end{prop}

\begin{proof}
It suffices to show that every $\Gamma\in{\rm dec}^{2}_{\rm d}[\Pi^0_2]{\rm dec}^\omega_{\rm p}[\Pi^0_1]$ is the composition of some $\Gamma_0\in{\rm dec}^2_{\rm d}[\Pi^0_2]$ and $\Gamma_1\in{\rm dec}^\omega_{\rm p}[\Pi^0_1]$.
For every $\Gamma\in{\rm dec}^2_{\rm d}[\Pi^0_2]{\rm dec}^\omega_{\rm p}[\Pi^0_1]$, there exist a $\Pi^0_2$ $d$-layer $\{D_0,D_1\}$ and $\Pi^0_1$ partitions $\{\{P^0_{n}\}_{n\in\nn},\{P^1_{n}\}_{n\in\nn}\}$ such that $\Gamma^i_{n}=\Gamma\res D_i\cap P^i_{n}$ is computable uniformly in $i<2$ and $n\in\nn$, where $\{P^i_{n}\}_{n\in\nn}$ is a partition of $D_i$ for every $i\in\{0,1\}$.
Let $\Gamma_0:D_0\cup D_1\to D_0\oplus D_1$ be the union of two computable homeomorphisms $D_0\simeq 0\fr D_0$ and $D_1\simeq 1\fr D_1$.
For instance, put $\Gamma_0(g)=i\fr g$ for $g\in D_i$.
Then $\Gamma_0\in {\rm dec}^2_{\rm d}[\Pi^0_2]$ since $\{D_0,D_1\}$ is a $\Pi^0_2$ $d$-layer.
Define $\Gamma_1(i\fr g)=\Gamma^i_n(g)$ for any $i<2$ and $g\in i\fr P^i_n$.
Then, $\Gamma_1\in{\rm dec}^\omega_{\rm p}[\Pi^0_1]$, since $\{\Gamma^i_n\}_{i<2,n\in\nn}$ is uniformly computable, and $\{P^i_n\}_{i<2,n\in\nn}$ is uniformly $\Pi^0_1$.
Clearly we have $\Gamma^i_n\res D_i\cap P^i_n=\Gamma_1\circ\Gamma_0\res D_i\cap P^i_n$ for any $i<2$ and $n\in\nn$.
Hence, $\Gamma=\Gamma_1\circ\Gamma_0$.
\end{proof}

The following concept of {\em hyperconcatenation} plays a key role in many proofs in Part II.
In the next section, we will see that the hyperconcatenation can be defined as {\em infinitary disjunction along an ill-founded tree}.

\begin{definition}[Hyperconcatenation]\label{def:1-2c:hyperconcat_f}
\index{${\tt content}$}\index{${\tt walk}$}%
For any strings $\sigma\in(\nn\cup\{{\tt pass}\})^{<\nn}$ and $\tau\in(\nn\cup\{\sharp,{\tt pass}\})^{<\nn}$, the {\em content of $\sigma$}, ${\tt content}(\sigma)$, and the {\em walk of $\tau$}, ${\tt walk}(\tau)$, is inductively defined as follows.
\begin{align*}
&{\tt content}(\lrangle{})=\lrangle{}, & & {\tt content}(\sigma)=
\begin{cases}
{\tt content}(\sigma^-)\fr \sigma(|\sigma|-1) & \mbox{if }\sigma(|\sigma|-1)\not={\tt pass},\\
{\tt content}(\sigma^-) & \mbox{otherwise.}
\end{cases}
\\
&{\tt walk}(\tau\res 1)=\lrangle{}, & & {\tt walk}(\tau)=
\begin{cases}
{\tt walk}(\tau^-)\fr v & \mbox{ if }\tau(|\tau|-2)=\sharp\;\&\;\tau(|\tau|-1)=v,\\
{\tt walk}(\tau^-) & \mbox{ otherwise.}
\end{cases}
\end{align*}
Then, the {\em content} of $f\in(\nn\cup\{{\tt pass}\})^{\nn}$ and the {\em walk} of $g\in(\nn\cup\{\sharp,{\tt pass}\})^{\nn}$ are defined by ${\tt content}(f)=\bigcup_{n\in\nn}{\tt content}(f\res n)$ and ${\tt walk}(g)=\bigcup_{n\in\nn}{\tt walk}(g\res n)$, respectively.
Let $P$ and $Q$ be {\em any} subsets of Baire space $\nn^\nn$.
\index{concatenation!hyper-}\index{hyperconcatenation}\index{$\bhk{Q\vee P}_{\Sigma^0_2}^\htie$}%
\index{concatenation!hyper-!non-Lipschitz}\index{hyperconcatenation!non-Lipschitz}\index{$\bhk{Q\vee P}_{\Sigma^0_2}$}%
The {\em hyperconcatenation} $\bhk{Q\vee P}_{\Sigma^0_2}^\htie$ and the {\em non-Lipschitz hyperconcatenation} $\bhk{Q\vee P}_{\Sigma^0_2}$ of $Q$ and $P$ are defined as follows.
\begin{align*}
\bhk{Q\vee P}_{\Sigma^0_2}^\htie&=\{g\in(\nn\cup\{\sharp\})^\nn:{\tt walk}(g)\in Q\;\text{or}\;{\tt tail}(g)^{\shft 1}\in P\},\\
\bhk{Q\vee P}_{\Sigma^0_2}&=\{g\in(\nn\cup\{\sharp,{\tt pass}\})^\nn:{\tt content}\circ{\tt walk}(g)\in Q\;\text{or}\;{\tt tail}(g)^{\shft 1}\in P\}.
\end{align*}

Note that these notions are non-commutative.
\end{definition}

\begin{theorem}[As the Law of Excluded Middle]\label{thm:5:lemiddle}
The implications (b$^+$) $\rightarrow$ (a) $\rightarrow$ (a$^-$) $\leftrightarrow$ (b$^-$) hold for any $P,Q,R\subseteq\nn^\nn$:
\begin{enumerate}
\item[(a)] $\bhk{Q\vee P}^\htie_{\Sigma^0_2}\leq^1_1R$.
\item[(a$^-$)] $\bhk{Q\vee P}_{\Sigma^0_2}\leq^1_1R$.
\item[(b$^+$)] There is a $\Sigma^0_2$ sentence $\varphi\equiv\exists v\theta(v)$ with a uniform sequence $\{\Gamma_i\}_{i\in\nn},\Delta$ of computable functions such that
\begin{itemize}
\item if $g\in R$ satisfies $\theta(v)$, then $\Gamma_v(g;u)\downarrow$ for any $u\in\nn$, and $\Gamma_v(g)\in P$.
\item if $g\in R$ satisfies $\neg\theta(v)$, then $\Delta(g;u)\downarrow$ for any $u\leq v$, and $[\Delta(g)\res v+1]$ intersects with $Q$.
\item if $g\in R$ satisfies $\neg\exists v\theta(v)$, then $\Delta(g;u)\downarrow$ for any $u\in\nn$, and $\Delta(g)\in Q$.
\end{itemize}
\item[(b$^-$)] There is a $\Sigma^0_2$ sentence $\varphi\equiv\exists v\theta(v)$ with a uniform sequence $\{\Gamma_i\}_{i\in\nn},\Delta$ of computable functions such that
\begin{itemize}
\item if $g\in R$ satisfies $\theta(v)$, then $\Gamma_v(g;u)\downarrow$ for any $u\in\nn$, and $\Gamma_v(g)\in P$.
\item if $g\in R$ satisfies $\neg\exists v\theta(v)$, then $\Delta(g;u)\downarrow$ for any $u\in\nn$, and $\Delta(g)\in Q$.
\end{itemize}
\end{enumerate}
\end{theorem}

\begin{proof}\upshape
(b$^+$)$\rightarrow$(a):
Assume that $S_i=\{g\in\nn^\nn:\Theta(g;i)\uparrow\}$ for some computable function $\Theta$, and that $P\leq^1_1R\cap S_i$ via $\Gamma_i$ and $Q\leq^1_1R\setminus\bigcup_{i\in\nn}S_i$ via $\Delta$.
For a string $\sigma\in\nn^{<\nn}$, define $d(\sigma)$ and $t(\sigma;i)$ as follows:
\begin{align*}
d(\sigma)&=\max\{d\in\nn:(\forall i<d)\;\Theta(\sigma;i)\downarrow\};\\
t(\sigma;i)&=\min\{t\in\nn:\Theta(\sigma\res t;i)\downarrow\},\text{ for any }i<d(\sigma).
\end{align*}
Then let us define $\Lambda(\sigma)=\concat_{i<d(\sigma)}\left(\Gamma_i(\sigma\res t(\sigma;i))\fr\sharp\fr\Delta(\sigma;i)\right)\fr\Gamma_{d(\sigma)}(\sigma)$.

(a$^-$)$\rightarrow$(b$^-$):
Assume that $\bhk{Q\vee P}_{\Sigma^0_2}\leq^1_1R$ via a computable function $\Phi$.
Set $S_v=\{g\in\nn^\nn:(\forall n\geq v)\;\Phi(g;n)\not=\sharp\}$.
For a string $\sigma\in\nn^{<\nn}$, we first computes the following ${\tt count}(\sigma)$ and ${\tt mcl}_\sharp(\sigma,n)$ for each $n\in\nn$:
\begin{align*}
{\tt count}(\sigma)&=\#\{m<|\sigma|:\Phi(\sigma;m)=\sharp\},\\
{\tt mcl}_\sharp(\sigma,n)&=\min\{m\leq |\sigma|:{\tt count}(\sigma\res m)>n\}, \text{ if such $m$ exists.}
\end{align*}
Then set $\Gamma_v(\sigma)=\Phi(\sigma)^{\shft {\tt mcl}_\sharp(\sigma,{\tt count}(\sigma\res v))+1}$; and set $\Delta(\sigma)=\lambda n.\Phi(\sigma,{\tt mcl}_\sharp(\sigma,n))$.
Note that if $g\in R\cap S_k$ for some $k\in\nn$, then $\Gamma_k(g)\in P$; otherwise, $\Delta(g)\in Q$.
Therefore, $P\leq^1_1R\cap S_v$ via $\Gamma_v$ and $Q\leq^1_1R\setminus S$ via $\Delta$.

(b$^-$)$\rightarrow$(a$^-$):
For each $\sigma\in\nn^{<\nn}$, let $v(\sigma)$ be the least $v$ such that $R(u,v,\sigma)$ holds for all $u<|\sigma|$, where $\varphi(g)\equiv(\exists v)(\forall u)R(u,v,g\res u)$.
We inductively define a computable function $\Phi$ as follows.
We first set $\Phi(\lrangle{})=\lrangle{}$.
Assume that $\Phi(\sigma^-)$ has been already defined.
\[
\Phi(\sigma)=
\begin{cases}
\Phi(\sigma^-)\fr\gamma, & \mbox{ if }v(\sigma)=v(\sigma^-)\;\&\;\Gamma_{v(\sigma)}(\sigma)={\tt tail}^+(\Phi(\sigma^-))\fr\gamma,\\
\Phi(\sigma^-)\fr\lrangle{\sharp,\delta(0)}, & \mbox{ if }v(\sigma)\not=v(\sigma^-)\;\&\;\Delta(\sigma)={\tt content}\circ{\tt walk}(\Phi(\sigma^-))\fr\delta,\\
\Phi(\sigma^-)\fr\lrangle{\sharp,{\tt pass}}, & \mbox{ if }v(\sigma)\not=v(\sigma^-)\;\&\;\Delta(\sigma)={\tt content}\circ{\tt walk}(\Phi(\sigma^-)).
\end{cases}
\]
For any $g\in\nn^\nn$, if $\varphi(g)\equiv(\exists v)(\forall u)R(u,v,g\res u)$, then for the least such $v\in\nn$, we have ${\tt tail}^+(\Phi(g))=\Gamma_v(g)$.
Otherwise, we have ${\tt content}\circ{\tt walk}(\Phi(g))=\Delta(g)$.
Hence, $\Phi(g)\in \bhk{Q\vee P}_{\Sigma^0_2}$, for any $g\in R$.
\end{proof}

\begin{definition}\label{def:1-2c:computable-along}
\index{computable!along a sequence}\index{computable!strictly along a sequence}%
Let $\{S_n\}_{n\in\nn}$ be an increasing sequence of subsets of $\nn^\nn$.
We say that a partial function $\Gamma:\subseteq\nn\to\nn$ is {\em computable along $\{S_n\}_{n\in\nn}$} if $\Gamma\res{\rm dom}(\Gamma)\setminus\bigcup_nS_n$ and $\Gamma\res {\rm dom}(\Gamma)\cap S_n\setminus S_{n-1}$ is computable uniformly in $n\in\nn$, where $S_{-1}=\emptyset$.
Moreover, we also say that a partial function $\Gamma:\subseteq\nn\to\nn$ is {\em computable strictly along $\{S_n\}_{n\in\nn}$} if there is a uniform sequence of computable functions $\{\Gamma_n\}_{n\in\nn}$ and $\Delta$ such that $\Gamma\res{\rm dom}(\Gamma)\setminus\bigcup_nS_n=\Delta\res{\rm dom}(\Gamma)\setminus\bigcup_nS_n$ and $\Gamma\res{\rm dom}(\Gamma)\cap S_n\setminus S_{n-1}=\Gamma_n\res {\rm dom}(\Gamma)\cap S_n\setminus S_{n-1}$ and $\Delta(g)\res n$ is defined for any $g\in{\rm dom}(\Gamma)\setminus S_n$.
\end{definition}

\begin{remark}
Theorem \ref{thm:5:lemiddle} implies that there is a function $\Gamma:\bhk{Q\vee P}_{\Sigma^0_2}\to P\oplus Q$ ($\Gamma:\bhk{Q\vee P}^\htie_{\Sigma^0_2}\to P\oplus Q$) such that $\Gamma$ is computable (strictly) along sequences of $\Pi^0_1$ sets.
\end{remark}

\begin{cor}\label{cor:1-2c:computable-along}
${\rm dec}^{<\omega}_{\rm d}[\Pi^0_2]{\rm dec}^\omega_{\rm p}[\Pi^0_1]$ is the smallest monoid containing all functions computable (strictly) along sequences of $\Pi^0_1$ sets.
\end{cor}

\begin{proof}
Let $\mathcal{S}$ be the class of all functions computable (strictly) along sequences of $\Pi^0_1$ sets.
Then, clearly, we have ${\rm dec}^2_{\rm d}[\Pi^0_2]\cup\Gamma_1\in{\rm dec}^\omega_{\rm p}[\Pi^0_1]\subseteq\mathcal{S}\subseteq{\rm dec}^{<\omega}_{\rm d}[\Pi^0_2]{\rm dec}^\omega_{\rm p}[\Pi^0_1]$.
Thus, the desired condition follows from Proposition \ref{prop:1-2c:dec-sigma2-lem}.
\end{proof}

%% file: NRMP_fullproof1-2d.tex
\section{Going Deeper and Deeper}

\subsection{Falsifiable Mass Problems}

We are mostly interested in local degree structures such as Turing degrees of c.e.~subsets of $\nn$ and Medvedev degrees of $\Pi^0_1$ subsets of $2^\nn$.
In such cases, the straightforward two-tape (backtrack) notions in Definitions \ref{def:1-2b:twotape} and \ref{def:12b:backtrack} are hard to use, since, for instance, $\bhk{P\vee Q}^2_{{\sf LCM}[2]}$ may not belong to $\Pi^0_1$ even if $P$ and $Q$ are $\Pi^0_1$.
This observation prompts us to define {\em consistent} two-tape disjunctions.
{\em The consistency set ${\rm Con}(T_i)_{i\in I}$ for $\{T_i\}_{i\in I}$} is defined as follows.
\index{consistency set}%
\[{\rm Con}(T_i)_{i\in I}=\{f\in(I\times\nn)^\nn:(\forall i\in I)(\forall n\in\nn)\;{\tt pr}_i(f\res n)\in T_i\}.\]

The notion of consistency sets has a relationship with consistent learning (see also Remark below Proposition \ref{prop:1:concat}).
The consistency sets are useful to reduce the complexity of our disjunctions to be $\Pi^0_1$.
We now introduce the following consistent modifications of our disjunctive notions.

\begin{definition}
Let $P_0$ and $P_1$ denote $\Pi^0_1$ subsets of $\nn^\nn$.
\index{$P_0\lcm P_1$}\index{$P_0\tie_n P_1$}\index{$P_0\cls P_1$}%
\begin{align*}
P_0\lcm P_1&=\bhk{P_0\vee P_1}_{\sf LCM}^2\cap{\rm Con}(T_{P_0},T_{P_1}).\\
P_0\tie_{n}P_1&=\bhk{P_0\vee P_1}_{{\sf LCM}[n]}^2\cap{\rm Con}(T_{P_0},T_{P_1}).\\
P_0\cls P_1&=\bhk{P_0\vee P_1}_{\sf CL}^2\cap{\rm Con}(T_{P_0},T_{P_1}).
\end{align*}
Here $T_{P_0}$ and $T_{P_1}$ are corresponding trees for $P_0$ and $P_1$, respectively.
\end{definition}

\begin{prop}\label{prop:1-2:consistency}
Let $P$ and $Q$ be $\Pi^0_1$ subsets of $\nn^\nn$.
\begin{enumerate}
\item $P\tie_{n}Q\equiv^1_1\bhk{P\vee Q}_{{\sf LCM}[n]}^2$ for each $n\in\nn$.
\item $P\lcm Q\equiv^1_1\bhk{P\vee Q}_{\sf LCM}^2$.
\item $P\cls Q\equiv^1_1\bhk{P\vee Q}_{\sf CL}^2$.
\end{enumerate}
\end{prop}

\begin{proof}\upshape
For each item, clearly $P\tie_*Q\geq^1_1\bhk{P\vee Q}_*^2$.
Thus, it suffices to construct a computable functional $\Phi$ witnessing $P\tie_*Q\leq^1_1\bhk{P\vee Q}_*^2$.
Let $T_0$ and $T_1$ denote the corresponding computable trees for $P$ and $Q$ respectively.
Set $\Phi(\lrangle{})=\lrangle{}$.
Fix $\sigma\in(2\times\nn)^{<\nn}$.
Assume that $\Phi(\sigma^-)$ has already been defined, and $\sigma=\sigma^-\fr\lrangle{\pair{i,k}}$ for some $i<2$ and $k\in\nn$.
Then, 
\[
\Phi(\sigma)=
\begin{cases}
\Phi(\sigma^-)\fr\lrangle{\pair{i,k}}&\mbox{ if }{\tt pr}_i(\sigma)\in T_i,\\
\Phi(\sigma^-)&\mbox{ if }{\tt pr}_i(\sigma)\not\in T_i,
\end{cases}
\]

Clearly, $\Phi$ is a computable function, since $T_i$ is computable for each $i<2$.
For any $g\in(2\times\nn)^\nn$, clearly ${\tt mc}(\Phi(g))\leq{\tt mc}(g)$.
Fix $g\in\bhk{P\vee Q}_*^2$, where $*\in\{{\sf LCM},{\sf LCM}[n],{\sf CL}\}$.
Then ${\tt pr}_i(g)\in P_i$ for some $i<2$, where $P_0=P$ and $P_1=Q$.
Therefore, $\Phi(g)$ is total, and ${\tt pr}_i(\Phi(g))\in P_i$ for such $i<2$.
\end{proof}

\begin{prop}
Let $P$ and $Q$ be $\Pi^0_1$ subsets of $\nn^\nn$.
\begin{enumerate}
\item $P\tie_{n}Q$ is $\Pi^0_1$, for any $n\in\nn$.
\item $P\lcm Q$ is $\Sigma^0_2$.
\item $P\cls Q$ is $\Pi^0_1$.
\end{enumerate}
\end{prop}

\begin{proof}\upshape
Let $T_0$ and $T_1$ denote the corresponding computable trees for $P$ and $Q$ respectively.
We consider the following computable tree:
\[T_{P,Q,n}=\{\sigma\in(2\times\nn)^{<\nn}:(\forall i<2)\;{\tt pr}_i(\sigma)\in T_i\;\&\;{\tt mc}(\sigma)<n\}.\]
Note that $T_{P,Q,n}$ is uniformly computable in $n$, since ${\tt pr}_i(\sigma)$ and ${\tt mc}(\sigma)$ are computable uniformly in $\sigma\in\nn^{<\nn}$.
Clearly, $P\tie_nQ\subseteq[T_{P,Q,n}]$.
Moreover, for any $g\in[T_{P,Q,n}]$, ${\tt pr}_i(g)$ is total for some $i<2$.
Then, ${\tt pr}_i(g)\in[T_i]$ for such $i$, and ${\tt mc}(g)\leq n$, since the relation ${\tt mc}(f)\leq n$ is equivalent to $(\forall k)\;{\tt mc}(f\res k)\leq n$.
Thus, $g\in P\tie_nQ$.
Consequently, $P\tie_nQ=[T_{P,Q,n}]$ is $\Pi^0_1$.
Hence, $P\lcm Q=\bigcup_n[T_{P,Q,n}]$ is $\Sigma^0_2$.
The items (3) also follows from the similar argument.
\end{proof}

\begin{definition}
Let $L_P$ denote the set of all leaves of the corresponding tree for a nonempty $\Pi^0_1$ set $P$ (where recall from Section \ref{subsec:1:notation} that such a tree is assumed to be uniquely determined when an index of $P$ is given).
Then {\em the (non-commutative) concatenation of $P$ and $Q$} is defined as follows.
\index{concatenation!non-commutative}\index{concatenation}\index{$P\ntie Q$}%
\[P\ntie Q=P\cup\bigcup_{\rho\in L_P}\rho\fr Q.\]
{\em The commutative concatenation of $P$ and $Q$} is defined by $P\tie Q=(P\ntie Q)\linf(Q\ntie P)$.
\index{concatenation!commutative}\index{$P\tie Q$}%
\end{definition}

\begin{remark}
On the study of Wadge degrees of finite level of Borel hierarchy, Duparc \cite{Dup} introduced various operators such as $P^{\longrightarrow}Q=P\cup\bigcup_{\rho\in\nn^{<\nn}}\rho\fr\lrangle{\sharp}\fr Q$.
\index{Duparc's operation}\index{$P^{\longrightarrow}Q$}%
The following proposition indicates that our non-commutative concatenation is essentially same as Duparc's operation $P^{\longrightarrow}Q$.
\end{remark}

\begin{prop}\label{prop:1-2:cut-leaf}
Let $P,Q$ be $\Pi^0_1$ subsets of Baire space $\nn^\nn$.
Then, the concatenation $P\fr Q$ is $(1,1)$-equivalent to the set $P^{\rightarrow}Q$ of all infinite paths of the tree $\{\sigma\fr\lrangle{\sharp}\fr\tau:\sigma\in T_P\;\&\;\tau\in T_Q\}$.
\end{prop}

\begin{proof}\upshape
To see $P^\rightarrow Q\leq^1_1P\fr Q$, we inductively define a total computable function ${\tt cut}:\nn^\nn\to\nn^\nn$.
First set ${\tt cut}(\lrangle{})=\lrangle{}$, and fix $\sigma=\sigma^-\fr\lrangle{n}\in\nn^{<\nn}$.
We assume that ${\tt cut}(\sigma^-)$ has been already defined.
If $\sigma=\sigma^-\fr\lrangle{n}\in L_P$, then we set ${\tt cut}(\sigma)={\tt cut}(\sigma^-)\fr\lrangle{n,\sharp}$.
Otherwise, we set ${\tt cut}(\sigma)={\tt cut}(\sigma^-)\fr\lrangle{n}$.
Then, ${\tt cut}$ is computable, since $P$ is $\Pi^0_1$ and then $T_P$ is computable.
Moreover, we can see the following.
\[
{\tt cut}(f)=
\begin{cases}
f & \mbox{ if } f\in P,\\
(f\res k)\fr\lrangle{\sharp}\fr f^{\shft k} & \mbox{ if }(\exists k\in\nn)\;f\res k\in L_P.
\end{cases}
\]
Clearly, $P^\rightarrow Q\leq^1_1P\fr Q$ via the computable function ${\tt cut}$.

Conversely, we consider the computable function ${\tt leaf}:\nn^{<\nn}\to\nn^{<\nn}$ which maps $\sigma$ to the least leaf of $L_P$ extending $\sigma$.
Then, we inductively define a computable function $\Gamma$ witnessing $P\fr Q\leq^1_1P^\rightarrow Q$ as follows.
First set $\Gamma(\lrangle{})=\lrangle{}$, and fix $\sigma=\sigma^-\fr\lrangle{n}\in(\nn\cup\{\sharp\})^{<\nn}$.
We assume that $\Gamma(\sigma^-)$ has been already defined.
If $n\not=\sharp$, then we set $\Gamma(\sigma)=\Gamma(\sigma^-)\fr\lrangle{n}$.
If $n=\sharp$, then we set $\Gamma(\sigma)={\tt leaf}(\Gamma(\sigma^-))$.
It is easy to see that $P\fr Q\leq^1_1P^\rightarrow Q$ via $\Gamma$.
\end{proof}

\begin{remark}
Inspired by our method used in Part II, Cenzer-Kihara-Weber-Wu \cite{CKWW} explicitly employed the concept of the (non-commutative) concatenation to show that ${\sf CPA}\ntie{\sf CPA}$ has a greatest Medvedev degree of $\Pi^0_1$ subsets of $2^\nn$ with no tree-immune.
Here, a $\Pi^0_1$ set $P\subseteq 2^\nn$ is {\em tree-immune} if the $\Pi^0_1$ tree $\{\sigma\in 2^{<\nn}:P\cap[\sigma]\not=\emptyset\}$ includes no infinite computable subtree, and ${\sf CPA}$ is the set of all {\em complete consistent extensions of Peano Arithmetic}.
\index{tree-immune}\index{${\sf CPA}$}%
Note that ${\sf CPA}$ is a {\em Medvedev complete} $\Pi^0_1$ subset of $2^\nn$.
\end{remark}

\begin{prop}\label{prop:1:concat}
Let $P,Q$ be $\Pi^0_1$ subsets of $\nn^\nn$.
\begin{enumerate}
\item $P\tie P\equiv^1_1P\ntie P$.
\item $P\tie Q\equiv^1_1\bhk{P\vee Q}^2_{{\sf LCM}[2]}$.
\end{enumerate}
\end{prop}

\begin{proof}\upshape
(1) $P\tie P=(P\ntie P)\oplus(P\ntie P)\equiv^1_1P\ntie P$.
(2) By Proposition \ref{prop:1-2:consistency} (1), we have $P\tie_2Q\equiv^1_1\bhk{P\vee Q}^2_{{\sf LCM}[2]}$.
Then, $P\tie_2Q\leq^1_1P\tie Q$ is witnessed by the following reduction $\Delta$.
\[
\Delta(f)=
\begin{cases}
{\tt write}(f(0),f^{\shft 1}), & \text{ if }f^{\shft 1}\in [T_{\sigma(0)}],\\
{\tt write}(f(0),f^{\shft 1}\res k)\fr{\tt write}(1-f(0),f^{\shft k+1}), & \text{ if }(\exists k\in\nn)\;f^{\shft 1}\res k\in L_{\sigma(0)}.
\end{cases}
\]
Here, $T_0$ and $T_1$ are the corresponding computable trees for $P$ and $Q$ respectively, and $L_i$ is the set of all leaves of $T_i$ for each $i<2$.
Clearly, $\Delta$ is computable.
Fix $\lrangle{i}\fr g\in P\tie Q$.
Obviously, ${\tt mc}(\lrangle{i}\fr g)<2$.
If $g\in[T_i]$ then ${\tt pr}_i(\Delta(\lrangle{i}\fr g))=g\in[T_i]$, and if $g=\sigma\fr h$ for some $\sigma\in L_i$ and $h\in[T_{1-i}]$ then ${\tt pr}_i(\Delta(\lrangle{i}\fr\sigma\fr h))=h\in[T_{1-i}]$.
Hence, $\Delta(\lrangle{i}\fr g)\in P\tie_2Q$.

To see $P\tie Q\leq^1_1P\tie_2Q$, it suffices to construct a computable functional $\Gamma$ witnessing $(P^\to Q)\oplus(Q^\to P)\leq^1_1P\tie_2Q$ by Proposition \ref{prop:1-2:cut-leaf}.
Set $\Gamma(\lrangle{})=\lrangle{}$, and $\Gamma(\lrangle{\pair{i,n}})=\lrangle{i,n}$ for any $i<2$ and $n\in\nn$.
Fix $\sigma=\sigma^{--}\fr\lrangle{\pair{i,m},\pair{j,n}}\in(2\times\nn)^{<\nn}$, and assume that $\Gamma(\sigma^-)$ is already defined.
If $i\not=j$, then set $\Gamma(\sigma)=\Gamma(\sigma^-)\fr\lrangle{\sharp,n}$.
Otherwise, set $\Gamma(\sigma)=\Gamma(\sigma^-)\fr\lrangle{n}$.
Clearly $\Gamma$ is computable.
Fix $g\in P\tie_2Q$.
If ${\tt mc}(g)=0$, then $\Gamma(g)=\lrangle{i}\fr{\tt pr}_i(g)\in P\oplus Q\subseteq(P^\to Q)\oplus(Q^\to P)$, where $i=(g(0))_0$.
If ${\tt mc}(g)=1$, then ${\tt pr}_i(g)$ is a finite string, where $i=(g(0))_0$.
In this case, we can easily see $\Gamma(g)=\lrangle{i}\fr{\tt pr}_i(g)\fr\lrangle{\sharp}\fr{\tt pr}_{1-i}(g)\in(P^\to Q)\oplus(Q^\to P)$.
\end{proof}

In the case of $P\tie P$, we use the non-commutative concatenation $P\fr P$ to simplify our proof without mentioning.

\begin{remark}
These disjunctions have some connection with {\em consistent conservative Popperian learning} (see \cite{JORS}).
\begin{itemize}
\item The term ``{\em consistent}'' means:
the scientist should modify his hypothesis whenever it was found to be refuted.
\item The term ``{\em conservative}'' means:
the scientist changes his hypothesis only when it was found to be refuted.
\item The term ``{\em Popperian}'' means:
the scientist can test whether his hypothesis is currently consistent or refuted.
\end{itemize}
The notion of {\em Popperian learning} is introduced by Case and Ngo-Manguelle \cite{CN79} based on Gold's theory of ``identification in the limit'' \cite{Gol}.
A learner (a scientist) is a computable function $\Psi:\nn^{<\nn}\to\nn$, and a natural phenomenon is a computable function $f:\nn\to\nn$.
Then the formula $\Psi(f\res n)=e$ means the following situation: the scientist $\Psi$ predicts that a rule generating the phenomenon $f$ can be explained by a word (a formula, or an algorithm) $e$ (i.e., $f=\Phi_e$) when he observes $f(0),\dots f(n-1)$.
We say that $\Psi$ {\em learns} $f$ if $\Phi_{\lim_n\Psi(f\res n)}=f$.
The learner $\Psi$ is {\em Popperian} if $\Phi_{\Psi(\sigma)}$ is total for each $\sigma\in\nn^{<\nn}$.
The learner $\Psi$ is {\em consistent} at $\sigma\in\nn^{<\nn}$ if $\Phi_{\Psi(\sigma)}\res|\sigma|=\sigma$.
The learner $\Psi$ is {\em conservative} if, for any $\sigma\in\nn^{<\nn}$, $\Psi(\sigma)=\Psi(\sigma^-)$ whenever $\Phi_{\Psi(\sigma^-)}\res|\sigma|=\sigma$.
Note that, for every Popperian learner $\Psi$, he can algorithmically determine whether $\Psi$ is consistent at $\sigma$ or not, for a given $\sigma\in\nn^{<\nn}$.
The terminology ``{\em Popperian}'' derives from Popper's falsifiabillity principle in philosophy of science.
\index{falsifiability}\index{learner}\index{learner!Popperian}\index{learner!consistent}\index{learner!conservative}%

The complexity $\Pi^0_1$ reflects the concept of Popperian learning.
The consistency set ${\rm Con}(T_i)_{i\in I}$ restricts our learning process to be consistent.
Additionally, the non-commutative concatenation $P\ntie Q$ of $P$ and $Q$ restricts our learning process to be conservative, since it represents the following situation:
a choice on the first hypothesis $P$ is refuted if, and only if, the scientist proposes the second (refutable) hypothesis $Q$ and start verifying it.
\end{remark}

\begin{table}\label{tablehiercon}\caption{Hierarchy of Consistent Disjunctions}%
\begin{center}
\begin{tabular}{ccc}\toprule
$P\linf Q$ & $\bhk{P\vee Q}^1_{\sf Int}$ & Intuitionistic disujunction ($=P\tie_1Q$) \\
$P\cup Q$ & $\bhk{P\vee Q}^1_{\sf CL}$ & Classical one-tape disjunction \\
$P\tie Q$ & $\bhk{P\vee Q}^2_{{\sf LCM}[2]}$ & Commutative concatenation ($\equiv P\fr Q$ if $P=Q$) \\
$P\tie_{n}Q$ & $\bhk{P\vee Q}^2_{{\sf LCM}[n]}$ & {\sf LCM} disjunction with mind-changes-bound $n$ \\
$P\lcm Q$ & $\bhk{P\vee Q}^2_{{\sf LCM}}$ & {\sf LCM} disjunction \\
$P\cls Q$ & $\bhk{P\vee Q}^2_{\sf CL}$ & Classical disjunction \\
\bottomrule
\end{tabular}
\end{center}
\end{table}

\begin{prop}
For $\Pi^0_1$ sets $P,Q\subseteq \nn^\nn$ and $n\in\nn$,
\[\bhk{P\vee Q}_{\sf LCM}^2\leq^1_1\bhk{P\vee Q}_{{\sf LCM}[n+2]}^2\leq^1_1\bhk{P\vee Q}^1_{\sf CL}\leq^1_1\bhk{P\vee Q}^1_{\sf Int}.\]
\end{prop}

\begin{proof}\upshape
It suffices to show $P\tie Q\leq^1_1P\cup Q$, since $\bhk{P\vee Q}^1_{\sf CL}\equiv^1_1P\cup Q$ by Proposition \ref{prop:1:basic} (5) and $\bhk{P\vee Q}^2_{{\sf LCM[2]}}\equiv^1_1P\tie Q$ by Proposition \ref{prop:1:concat} (2).
Indeed, we can show that $(P\ntie Q)\otimes(Q\ntie P)\leq^1_1P\cup Q$.
We construct a computable functional $\Phi$ witnessing $P\ntie Q\leq^1_1P\cup Q$.
If $\sigma\in T_P$, then set $\Phi(\sigma)=\sigma$.
If $\sigma\not\in T_P$, then pick a unique $\rho\subseteq\sigma$ such that $\rho\in L_P$, and set $\Phi(\sigma)=\rho\fr\sigma$ for such $\rho$, where $L_P$ is the set of all leaves of $T_P$.
Clearly $\Phi$ is computable, and note that $\Phi(\sigma)\subseteq\Phi(\tau)$ whenever $\sigma\subseteq\tau$.
If $g\in P$, then $\Phi(g)=g\in P$.
If $g\in Q\setminus P$, then there is a unique $\rho\subset g$ such that $\rho\in L_P$, and $\Phi(g)=\rho\fr g\in P\ntie Q$.
Thus, $P\ntie Q\leq^1_1P\cup Q$ via $\Phi$.
\end{proof}

\subsection{Compactified Infinitaly Disjunctions}

This subsection is concerned with a trick to represent {\em infinitary} disjunctive notions as effective compact sets.

\begin{definition}
Fix a collection $\{P_i\}_{i\in I}$ of subsets of Baire space $\nn^\nn$.
\index{$\bhk{\bigvee_{i\in I}P_i}_{\sf Int}$}\index{$\bhk{\bigvee_{i\in I}P_i}_{\sf LCM}$}%
\index{$\bhk{\bigvee_{i\in I}P_i}_{\sf CL}$}%
\begin{enumerate}
\item $\bhk{\bigvee_{i\in I}P_i}_{\sf Int}=\{f\in(I\times\nn)^\nn:((\exists i\in I)\;{\tt pr}_i(f)\in P_i)\;\&\;{\tt mc}(f)=0\}$.
\item $\bhk{\bigvee_{i\in I}P_i}_{\sf LCM}=\{f\in(I\times\nn)^\nn:((\exists i\in I)\;{\tt pr}_i(f)\in P_i)\;\&\;{\tt mc}(f)<\infty\}$.
\item $\bhk{\bigvee_{i\in I}P_i}_{\sf CL}=\{f\in(I\times\nn)^\nn:(\exists i\in I)\;{\tt pr}_i(f)\in P_i\}$.
\end{enumerate}
\end{definition}

\begin{prop}\label{prop:2:inf-disjunc-equiv}
Let $\{P_n\}_{n\in\nn}$ be an infinite collection of subsets of Baire space $\nn^\nn$.
\begin{enumerate}
\item $\bhk{\bigvee_{n\in\nn}P_n}_{\sf Int}\equiv^1_1\bigoplus_{n\in\nn}P_n$, where $\bigoplus_{n\in\nn}P_n=\{\lrangle{n}\fr f:f\in P_n\}$.
\item $\bhk{\bigvee_{i,n}P_{i,n}}_{\sf LCM}\equiv^1_1\bhk{P_0\vee P_1}^3_{\sf LCM}$, where $P_{i,n}=P_i$ for each $i<2$ and $n\in\nn$.
\end{enumerate}
\end{prop}

\begin{proof}\upshape
(1) $\bhk{\bigvee_{n\in\nn}P_n}_{\sf Int}\geq^1_1\bigoplus_{n\in\nn}P_n$ is witnessed by $f\mapsto(f(0))_0\fr{\tt pr}_{(f(0))_0}(f)$, and $\bhk{\bigvee_{n\in\nn}P_n}_{\sf Int}\leq^1_1\bigoplus_{n\in\nn}P_n$ is witnessed by $f\mapsto{\tt write}(f(0),f^{\shft 1})$, where recall that ${\tt write}(f(0),f^{\shft 1})=(f(0))^\nn\oplus(\lambda n.f(n+1))$ indicates the instruction to writing the infinite word $f^{\shft 1}$ on the $f(0)$-th tape.

(2) We first construct a computable function $\Xi$ witnessing $\bhk{\bigvee_{i,n}P_{i,n}}_{\sf LCM}\geq^1_1\bhk{P_0\vee P_1}^3_{\sf LCM}$.
For $\pair{\pair{i,n},v}\in(2\times\nn)\times\nn$, we first set $\Xi(\lrangle{\pair{\pair{i,n},v}})=\lrangle{\pair{\pair{i,n},v}}$.
For each string $\sigma=\sigma^{--}\fr\lrangle{\pair{\pair{i,n},v},\pair{\pair{j,m},w}}\in((2\times\nn)\times\nn)^{<\nn}$, inductively assume that $\Xi(\sigma^-)$ has been already defined.
If $\pair{i,n}=\pair{j,m}$, then we set $\Xi(\sigma)=\Xi(\sigma^-)\fr\lrangle{w}$.
Otherwise, we set $\Xi(\sigma)=\Xi(\sigma^-)\fr\lrangle{\sharp,j,w}$.
For any $f\in\bhk{\bigvee_{i,n}P_{i,n}}_{\sf LCM}$, the backtrack symbol $\sharp$ occurs in $\Xi(f)$ finitely often, since ${\tt mc}(f)<\infty$.
Therefore, ${\tt tail}(\Xi(f))$ converges, and ${\tt tail}(\Xi(f))^{\shft 1}={\tt pr}_{i,m}(f)\in P_i$ for some $i<2$ and $m\in\nn$.
Thus, $\Xi(f)\in\bhk{P_0\vee P_1}^3_{\sf LCM}$.

We next construct a computable function $\Xi^*$ witnessing $\bhk{\bigvee_{i,n}P_{i,n}}_{\sf LCM}\leq^1_1\bhk{P_0\vee P_1}^3_{\sf LCM}$.
Set $\Xi^*(\lrangle{})=\lrangle{}$.
For $\sigma=\sigma^{--}\fr\lrangle{v,w}\in(\nn\cup\{\sharp\})^{<\nn}$, inductively assume that $\Xi^*(\sigma^-)$ has been already defined.
To define $\Xi^*(\sigma)$, recall the definition ${\tt count}(\sigma)=\#\{n<|\sigma|:\sigma(n)=\sharp\}$.
Then $\Xi^*(\sigma)$ is defined as follows.
\[
\Xi^*(\sigma)=
\begin{cases}
\Xi^*(\sigma^-)\fr\lrangle{\pair{\pair{{\tt tail}(\sigma;0),{\tt count}(\sigma)},w}},&\mbox{ if $v\not=\sharp$ and $w\not=\sharp$},\\
\Xi^*(\sigma^-),&\mbox{ otherwise}
\end{cases}
\]

For any $f\in\bhk{P_0\vee P_1}^3_{\sf LCM}$, we have ${\tt mc}(\Xi^*(f))<\infty$, since ${\tt count}(f)=\#\{k\in\nn:f(k)=\sharp\}$ is finite.
Therefore, we have ${\tt pr}_{\pair{{\tt tail}(f;0),{\tt count}(f)}}(\Xi^*(f))={\tt tail}(f)^{\shft 1}\in P_{{\tt tail}(f;0)}$.
Thus, $\Xi^*(f)\in\bhk{\bigvee_{i,n}P_{i,n}}_{\sf LCM}$.
\end{proof}

We again use the consistent modifications of infinitary models, $\left[\blcm\right]_{n\in\nn}P_n=\bhk{\bigvee_{n\in\nn}P_n}_{\sf LCM}\cap{\rm Con}(T_{P_n})_{n\in\nn}$, and $\left[\bcls\right]_{n\in\nn}P_n=\bhk{\bigvee_{n\in\nn}P_n}_{\sf CL}\cap{\rm Con}(T_{P_n})_{n\in\nn}$.
\index{$\left[\blcm\right]_{n\in\nn}P_n$}\index{$\left[\bcls\right]_{n\in\nn}P_n$}%

\begin{prop}\label{prop:2:inf-consistency}
Let $\{P_n\}_{n\in\nn}$ be a computable collection of $\Pi^0_1$ subsets of Baire space $\nn^\nn$.
\begin{enumerate}
\item $\bhk{\bigvee_{n\in\nn}P_n}_{\sf LCM}\equiv^1_1\left[\blcm\right]_{n\in\nn}P_n$.
\item $\bhk{\bigvee_{n\in\nn}P_n}_{\sf CL}\equiv^1_1\left[\bcls\right]_{n\in\nn}P_n$.
\end{enumerate}
\end{prop}

\begin{proof}\upshape
As in the proof of Proposition \ref{prop:1-2:consistency}.
\end{proof}

However, the problem is that our models of infinitary disjunctions are not compact.
A modification of infinitary sum was introduced by Binns-Simpson \cite{BS} to embed a free Boolean algebra into the Muchnik lattice of $\Pi^0_1$ subsets of Cantor space, and such a variation was called {\em a recursive meet}.
An important feature of their modification is that it is a $\Pi^0_1$ subset of the compact space $2^\nn$.

\begin{definition}[Binns-Simpson \cite{BS}]
Let $P$ and $\{Q_n\}_{n\in\nn}$ be computable collection of $\Pi^0_1$ subsets of $2^\nn$, and let $\rho_n$ denote the length-lexicographically $n$-th leaf of the corresponding computable tree of $P$.
Then, we define the {\em infinitary concatenation} and {\em recursive meet} as follows:
\index{concatenation!infinitary}\index{$P\ntie\{Q_i\}_{i\in\nn}$}
\index{recursive meet}\index{$\cmeet_{i\in\nn}Q_i$}%
\begin{align*}
P\ntie\{Q_i\}_{i\in\nn}=P\cup\bigcup_n\rho_n\fr Q_n,& &\cmeet_{i\in\nn}Q_i={\sf CPA}\ntie\{Q_i\}_{i\in\nn}.
\end{align*}

Here, recall that ${\sf CPA}$ is a Medvedev complete set, which consists of all {\em complete consistent extensions of Peano Arithmetic}.
The Medvedev completeness of ${\sf CPA}$ ensures that for any nonempty $\Pi^0_1$ subset $P\subseteq 2^\nn$, a computable function $\Phi:{\sf CPA}\to P$ exists.
\end{definition}

\begin{prop}
For any computable sequence $\{P_n\}_{n\in\nn}$ of nonempty $\Pi^0_1$ subsets of $2^\nn$, $\cmeet_{n\in\nn}P_n\equiv^1_{<\omega}\bigoplus_{n\in\nn}P_n$.
\end{prop}

\begin{proof}\upshape
The condition $\cmeet_{n\in\nn}P_n\leq^1_1\bigoplus_{n\in\nn}P_n$ is witnessed by a computable function $n\fr g\mapsto\rho_n\fr g$.
We will construct a learner witnessing $\cmeet_{n\in\nn}P_n\geq^1_{<\omega}\bigoplus_{n\in\nn}P_n$.
Fix a computable function $\Phi_e:{\sf CPA}\to 0\fr P_0$.
Such $\Phi_e$ exists, since every nonempty $\Pi^0_1$ subset of $2^\nn$ is $(1,1)$-reducible to ${\sf CPA}$.
We also fix a partial computable function $\Phi_{i(n)}:\rho_n\fr g\mapsto n\fr g$, for each $n\in\nn$.
For $\sigma\in 2^{<\nn}$, if $\sigma\in T_{\sf CPA}$ then set $\Psi(\sigma)=e$.
If $\sigma\not\in T_{\sf CPA}$, then $\rho_n\subseteq\sigma$ for some $n$.
For such $n$, we set $\Psi(\sigma)=i(n)$.
The function $\Gamma$ identified by the learner $\Psi$ is clearly $(1,2)$-computable, and $\Gamma(g)\in\bigoplus_{n\in\nn}P_n$ for any $g\in\cmeet_{n\in\nn}P_n$.
\end{proof}

\subsection{Infinitary Disjunctions along well-Founded Trees}
\label{section:gdd_wellfounded}

\begin{definition}[Transfinite Mind-Changes]\label{def:div:trans-mc}
Let $(\mathcal{O},\leq_{\mathcal{O}})$ denote {\em Kleene's system of ordinal notations} (see Rogers \cite{Rog}).
Then for each $a\in\mathcal{O}$ we introduce the $a$-th derivative of $P\subseteq\nn^\nn$ as follows.
\index{$\mathcal{O}$}\index{$\leq_\mathcal{O}$}\index{Kleene's system of ordinal notations}\index{derivative}\index{$P^a$}\index{$P^{a+}$}%
\begin{align*}
P^a=
\begin{cases}
P\\
\bhk{P\vee P^{b}}^2_{\sf LCM[2]}\\
\bigoplus_{n\in\nn}P^{\Phi_e(n)}
\end{cases}
&
P^{a+}=
\begin{cases}
P&\mbox{ if }a=0,\\
\bhk{P\vee P^{b+}}^2_{\sf LCM[2]}&\mbox{ if }a=2^b,\\
\bhk{P\vee\bigoplus_{n\in\nn}P^{\Phi_e(n)+}}^2_{\sf LCM[2]}&\mbox{ if }a=3\cdot 5^e.
\end{cases}
\end{align*}
Here, we require $\Phi_e(n)<_\mathcal{O}\Phi_e(n+1)$ for every $3\cdot 5^e\in\mathcal{O}$ in the definition of $\mathcal{O}$.
In particular, this implies that $P^{(\Phi_e(m))}\leq^1_1P^{(\Phi_e(n))}$ whenever $n\leq m$.
Additionally, we may require that $\Phi_e(n)<\Phi_e(n+1)$ as a natural number by padding.
If $P$ is a nonempty $\Pi^0_1$ subset of $2^\nn$, we also define another derivative $P^{(a)}$ as follows.
\index{$P^{(a)}$}%
\[
P^{(a)}=
\begin{cases}
P&\mbox{ if }a=0,\\
P\ntie P^{(b)}&\mbox{ if }a=2^b,\\
P\fr\{P^{(\Phi_e(n))}\}_{n\in\nn}&\mbox{ if }a=3\cdot 5^e.
\end{cases}
\]
\end{definition}

\begin{prop}\label{prop:3:compactness-trans}
For any nonempty $\Pi^0_1$ set $P\subseteq 2^\nn$ and any notation $a\in\mathcal{O}$, the $a$-th derivative $P^{(a)}$ is a $\Pi^0_1$ subset of $2^\nn$.
\end{prop}

\begin{proof}\upshape
Fix $a\in\mathcal{O}$.
By our definition, obviously $P^{(a)}$ is a subset of $2^\nn$.
We inductively assume that $\{P^{(b)}:b<_\mathcal{O}a\}$ is uniformly $\Pi^0_1$.
For $a=2^b$, we can easily compute a $\Pi^0_1$ index of $P^{(a)}=P\ntie P^{(b)}$ is from a $\Pi^0_1$ index of $P^{(b)}$.
For $a=3\cdot 5^e$, we can also easily compute a $\Pi^0_1$ index of $P^{a}=P\fr\{P^{(\Phi_e(n))}\}_{n\in\nn}$ from a computable sequence of $\Pi^0_1$ indices of $\{P^{(\Phi_e(n))}\}_{n\in\nn}$.
Thus, $\{P^{(b)}:b\leq_{\mathcal{O}}a\}$ is uniformly $\Pi^0_1$.
\end{proof}

\begin{prop}\label{prop:div:ordinal-as-tree}
For any nonempty $\Pi^0_1$ set $P\subseteq 2^\nn$ and any notation $a\in\mathcal{O}$, the condition $P^{a+}\leq^1_1P^{(a)}\leq^1_1P^{a}$ holds.
\end{prop}

\begin{proof}
Clearly $P\ntie P^{(b)}$ is $(1,1)$-equivalent to $\bhk{P\vee P^{(b)}}^2_{\sf LCM[2]}$, since $P^{(b)}$ is $\Pi^0_1$ by Proposition \ref{prop:3:compactness-trans}, where the $(1,1)$-equivalence follows by Proposition \ref{prop:1-2b:cnsivee} and \ref{prop:1:concat}.
It is easy to see that $\bhk{P\vee\bigoplus_{n\in\nn}P^{(\Phi_e(n))}}^2_{\sf LCM[2]}\leq^1_1P\fr\{P^{(\Phi_e(n))}\}_{n\in\nn}\leq^1_1\bigoplus_{n\in\nn}P^{(\Phi_e(n))}$ holds.
For successor steps, it suffices to show that $P\fr P^{(b)}\leq_W(P^{(b)}\fr P)$.
If $|b|_\mathcal{O}$ is a finite ordinal, it is clear.
If $|b|_\mathcal{O}$ is an infinite ordinal, say $b=3\cdot 5^e$, then $P^{(b)}\leq^1_1P^{(b)}\fr P$ holds, since $\Phi_e(n)+1\leq_\mathcal{O}\Phi_e(n+1)$.
\end{proof}

\begin{notation}
Every $a\in\mathcal{O}$ is often identified with the corresponding well-founded tree $T_a$ consisting of all finite nonempty $<_\mathcal{O}$-decreasing sequences $\lrangle{a_0,a_1,a_2,\dots}$, where $a_0=a$ and for every $i\in\nn$, either $2^{a_{i+1}}=a_i$ or $a_{i+1}=\Phi_e(n)$ holds for some $n\in\nn$ and $e$ with $3\cdot 5^e=a_i$.
Our padding assumption $\Phi_e(n)<\Phi_e(n+1)$ implies that $T_a$ is computable.
\end{notation}

Definition \ref{def:div:trans-mc} immediately induces associated piecewise computability notions.
For a notation $a\in\mathcal{O}$, a collection $\{S_\kappa\}_{\kappa\in T_a}$ of $\Sigma^0_1$ subsets of $X\subseteq\nn^\nn$ is {\em $a$-indexed} if $S_{\lrangle{a}}=X$ and the mapping $\kappa\mapsto S_\kappa$ is an order preserving homomorphism from the tree $(T_a,\subseteq)$ onto the ordered set $(\{S_\kappa\}_{\kappa\in T_a},\supseteq)$, where $O(\leq a)=\{b:b\leq_\mathcal{O}a\}$.
It is {\em strictly $a$-indexed} if it is $a$-indexed and $S_\kappa=\bigcup_{n\in\nn}S_{\kappa\fr\Phi_e(n)}$ whenever $\kappa=\kappa^-\fr 3\cdot 5^e$.
A partial function $\Gamma:\subseteq\omega^\omega\to\omega^\omega$ is said to be {\em (strictly) $a$-indexed $\Pi^0_1$ $d$-layerwise computable} if there are a (strictly) $a$-indexed collection of $\Sigma^0_1$ subsets $\{S_\kappa\}_{\kappa\in T_a}$ of the domain of $\Gamma$ and a uniformly computable collection $\{\Gamma_\kappa\}_{\kappa\in T_a}$ of partial computable functions such that $\Gamma$ agrees with $\Gamma_\kappa$ on the domain $S_\kappa\setminus\bigcup_{\lambda\supsetneq\kappa}S_\lambda$.
\index{layer!$d$-layer!$a$-indexed}\index{layer!$d$-layer!strictly $a$-indexed}%
\index{layer!layerwise computability!$a$-indexed $\Pi^0_1$ $d$-}%
\index{layer!layerwise computability!strictly $a$-indexed $\Pi^0_1$ $d$-}%

It is easy to see that these notions are subclasses of ${\rm dec}^\omega_{\rm p}[\Pi^0_1]$.
If the order type $|a|_\mathcal{O}$ of $\{b:b<_\mathcal{O}a\}$ is $\omega$, the strict $a$-indexed $\Pi^0_1$ $d$-layerwise computability realizes the class $[\mathfrak{C}_T]^1_{\tt eff}$.
Obviously, a strict $a$-indexed $\Pi^0_1$ $d$-layerwise computable function $\Gamma:P^a\to P$ and an $a$-indexed $\Pi^0_1$ $d$-layerwise computable function $\Gamma^*:P^{a+}\to P$ exist.

\begin{remark}
Obviously, $a$-indexed $\Pi^0_1$ $d$-layerwise computability can be viewed as the effective version of discontinuity level $\leq_\mathcal{O}a$ in the sense of Hertling \cite{Hertling96} and Hemmerling \cite{Hemmerling08}.
Here, a partial function $\Gamma:\subseteq\nn^\nn\to\nn^\nn$ shall be said to be {\em of effective discontinuity level $\leq_\mathcal{O}a$} if there is a computable collection $\{\Gamma_b\}_{b\leq_\mathcal{O}}a$ of partial computable functions with uniform $\Sigma^0_1$ domains $\{S_b\}_{b\leq_\mathcal{O}a}$ such that for every $x\in{\rm dom}(\Gamma)$, $\Gamma(x)=\Gamma_b(x)$ for a unique $b\leq_\mathcal{O}a$ with $x\in S_b\setminus\bigcup_{c<_\mathcal{O}b}S_c$.

Note that Hemmerling \cite{Hemmerling08} studied its boldface version in the context of levels of subhierarchy (see Ma{\l}ek \cite{Malek06}) of the Baire one star functions $\mathcal{B}^*_1$ (see O'Malley \cite{Malley77}), whose original definition seems to be a boldface version of the Blum-Blum locking \cite{BlBl} in learning theory.
Then, the boldface version of the learnability with mind-change $1$ seems to be interpreted as the Baire one double star functions $\mathcal{B}^{**}_1$ (see Pawlak \cite{Pawlak00}).

Indeed, the notion of the discontinuity level is a useful tool to analyze the Baire hierarchy of the Borel measurable functions.
For instance, Solecki \cite[Theorem 3.1]{Sole98} used a transfinite derivation process in the proof of his dichotomy theorem for the Baire one functions, and Semmes \cite[Lemma 4.3.3]{Sem} introduced a high level analog of a transfinite derivation process in the proof of his decomposition theorem for the $\mathbf{\Lambda}_{2,3}$ functions (a subclass of the Baire two functions).

See also de Brecht \cite{deBre13} for a systematic study on the levels of discontinuity.
\end{remark}

\begin{definition}[see Freivalds-Smith \cite{FreSmi} and Luo-Schulte \cite{LuoS06}]
Let $\Psi:\nn^{<\nn}\to\nn$ be a learner.
We say that $c:\nn^{<\nn}\to\mathcal{O}$ is a {\em mind-change counter} for $\Psi$ if, for any $\sigma\in\nn^{<\nn}$, $c(\sigma)<_\mathcal{O}c(\sigma^-)$ whenever $\Psi(\sigma)\not=\Psi(\sigma^-)$.
\index{mind-change counter}%
A learner $\Psi$ is {\em $a$-bounded} if there is a computable mind-change counter $c:\nn^{<\nn}\to\mathcal{O}$ for $\Psi$ such that $c(\lrangle{})\leq_\mathcal{O}a$.
\index{learner!$a$-bounded}%
\end{definition}

\begin{remark}
The computational power of $a$-bounded learnability is very closely related to Ershov's mind-change hierarchy (Ershov hierarchy \cite{Ershov68}) of $\Delta^0_2$ subsets of $\nn$, or the effective version of the Hausdorff difference hierarchy of $\mathbf{\Delta}^0_2$ subsets of $\nn^\nn$ (for Ershov hierarchy, see also Stephan-Yang-Yu \cite{SteYanYu}).
\end{remark}

\begin{prop}
For a notation $a\in\mathcal{O}$, a partial function $\Gamma:\subseteq\nn^\nn\to\nn^\nn$ is of effective discontinuity level $\leq_\mathcal{O}a$ if and only if it is learnable via an $a$-bounded learner.
\end{prop}

\begin{proof}
The desired equivalence is obtained from an interpretation between $S_b$ and the $\Sigma^0_1$ set generated by the c.e.~set $\{\sigma\in\nn^{<\nn}:c(\sigma)\leq b\}$.
\end{proof}

\subsection{Infinitary Disjunctions along any Graphs}

In the classical proof process, a verifier $\Psi$ on ``$P_0$ or $P_1$'' may change his mind infinitely often.
In the backtrack-tape model, this situation means that $\Psi$ chooses the backtrack symbol $\sharp$ infinitely many often.
Then the word on $\Lambda$ is eventually finite, and it verifies neither $P_0$ nor $P_1$.
Therefore, in the model, if $\Psi$ succeeds to verify ``$P_0$ or $P_1$'' then the backtrack symbol $\sharp$ occurs on the record $\Delta$ at most finitely often.
Consequently, in the backtrack-tape model, classical verification coincides with {\sf LCM} verification.
However, we would like to cover the case that unbounded or infinitely many mind-changes occur.
This may be archived by regarding the backtrack-tape model as a kind of infinitary tape model.

\medskip

\noindent
{\bf The dynamic-tape model}: 
\index{model!dynamic-tape}%
Assume that a directed graph $(V,E)$ is given, where $V$ can be infinite, $E\subseteq V\times V$, and an {\em initial vertex} $\varepsilon\in V$ is chosen.
For any $v\in V$, let ${\rm adj}(v)=\{w\in V:(v,w)\in E\}$.
When a verifier $\Psi$ tries to prove that ``$\bigvee_{v\in V}P_v$'', infinite tapes $\square$, and $\Lambda_v$ for $v\in V$ are given.
The tape $\square$ is called {\em the declaration}, $\Lambda_v$ is called {\em the working tape} for each $v\in V$.
First the letter $\varepsilon$ is written on $\square$, and no word is written on $\Lambda_v$ for $v\in V$.
At each stage $s$, assume that $v[s]$ is written on $\square$.
Then the verifier $\Psi$ executes one or the other of two following actions.
\begin{enumerate}
\item $\Psi$ declares some $w\in{\rm adj}(v[s])$, erases all words on $\square$, and writes $w$ on $\square$; or
\item $\Psi$ writes a letter $k\in\nn$ on the working tape $\Lambda_{v[s]}$.
\end{enumerate}

Assume that a verifier $\Psi$ tries to prove that ``$P_0$ or $P_1$''.
\begin{itemize}
\item {\bf Intuitionism}: Consider $V=\{\varepsilon,0,1\}$, $E=\{(\varepsilon,0),(\varepsilon,1)\}$, and $P_\varepsilon=\emptyset$.
\item {\bf LCM with ordinal-bounded mind-changes}: For a computable well-founded tree $V=T\subseteq\nn^{<\nn}$, consider the following.
\[E=E(T)=\{(\sigma,\tau)\in T\times T:(\exists i\in\nn)\;\tau=\sigma\fr i\},\ P_\sigma=
\begin{cases}
P_0, & \mbox{ if $|\sigma|$ is even,}\\
P_1, & \mbox{ if $|\sigma|$ is odd,}
\end{cases}
\]
\item {\bf LCM}: Consider $V=\nn$; $E=\{(n,n+1):n\in\nn\}$; $P_{2n}=P_0$ for any $n\in\nn$; and $P_{2n+1}=P_1$ for any $n\in\nn$.
Moreover, the word written on the declaration $\square$ must converge.
\item {\bf $(V,E)$-relaxed Classical}: $(V,E)=(V_0,V_1,E)$ is a given directed bipartite graph, and $P_\tau=P_i$ for any $\tau\in V_i$ and $i<2$.
\end{itemize}

\begin{definition}[Dynamic Disjunctions]\label{def:dynamic}
\index{disjunction!dynamic}%
Let $G=(V,E)$ be a directed graph, and let $\{P_v\}_{v\in V}$ be a collection of subsets of Baire space.
For $E\subseteq V^2$, put $\overline{E}=E\cup\{\lrangle{v,v}:v\in V\}$.
We define the {\em dynamic disjunction of $\{P_v\}_{v\in V}$ along the graph $(V,E)$} as follows.
\index{disjunction!dynamic!along a graph}\index{$\bigbhk{\bigvee_{v\in (V,E)}P_v}$}%
\[\bigbhk{\bigvee_{v\in (V,E)}P_v}=\left\{f\in(V\times\nn)^\nn:(\forall n\in\nn)\;(\lrangle{(f(n))_0,(f(n+1))_0}\in \overline{E})\;\&\;(\exists v\in V)\;{\tt pr}_v(f)\in P_v\right\}.\]

Moreover, if $\{P_v\}_{v\in V}$ is a computable sequence of $\Pi^0_1$ subsets of $\nn^\nn$, and $T_{P_v}$ be the corresponding tree for $P_v$, we also define its consistent versions.
\index{$\btie_{v\in(V,E)}P_v$}\index{$\bhtie_{v\in(V,E)}P_v$}%
\begin{enumerate}
\item $\btie_{v\in(V,E)}P_v=\bhk{\bigvee_{v\in (V,E)}P_v}\cap{\rm Con}(T_{P_v})_{v\in V}$.
\item $\bhtie_{v\in(V,E)}P_v=\{f\in(V\times\nn)^\nn:(\forall n\in\nn)\;(\lrangle{(f(n))_0,(f(n+1))_0}\in \overline{E})\}\cap{\rm Con}(T_{P_v})_{v\in V}$.
\end{enumerate}
Here, recall that, for $x=(x_0,x_1)$, the first coordinate $x_0$ is denoted by $(x)_0$.
If $P_v=P$ for any $v\in V$, then we simply write $\btie_{v\in V}P$ and $\bhtie_{v\in V}P$ for $\btie_{v\in(V,E)}P_v$ and $\bhtie_{v\in(V,E)}P_v$ respectively.
\index{$\btie_{v\in V}P$}\index{$\bhtie_{v\in V}P$}%
\end{definition}

As our dynamic-tape model is an infinitary-tape model, this model may be natural to be regarded as expressing a proof process of an infinitary disjunction $\bigvee_{v\in V}P_v$.
Therefore, we refer the model with $(V,E)$ as {\em an infinitary disjunction along $(V,E)$}.
\index{disjunction!infinitary!along a graph}%
Later we will introduce a more complicated model.
It will be called {\em the nested-tape model}.
We first see an upper and lower bound of the degrees of difficulty of these disjunctive notions, and a relationship among various models we have introduced.
Let $\widehat{\rm Deg}(P)$ denote the {\em Turing upward closure} of $P$, i.e., $\widehat{\rm Deg}(P)=\{g:(\exists f\leq_Tg)\;f\in P\}$, and $[(V,E)]$ denote the set of all infinite paths through a graph $(V,E)$, i.e., $[(V,E)]=\{p\in V^\nn:(p(n),p(n+1))\in E\}$.
\index{Turing upward closure}\index{$\widehat{\rm Deg}(P)$}\index{$[(V,E)]$}%

\begin{prop}\label{prop:3:uplow}
Let $(V,E)$ be a computable directed graph, and $\{P_v\}_{v\in V}$ be a computable sequence of $\Pi^0_1$ subsets of $\nn^\nn$.
\begin{enumerate}
\item $\widehat{\rm Deg}\left(\bigoplus_{v\in V}P_v\right)\leq^1_1\btie_{v\in(V,E)}P_v\leq^1_1\bigoplus_{v\in V}P_v$.
\item $\widehat{\rm Deg}\left([(V,E)]\oplus\bigoplus_{v\in V}P_v\right)\leq^1_1\bhtie_{v\in(V,E)}P_v\leq^1_1[(V,E)]\oplus\bigoplus_{v\in V}P_v$.
\end{enumerate}
\end{prop}

\begin{proof}\upshape
(1) $\btie_{v\in(V,E)}P_v\leq^1_1\bigoplus_{v\in V}P_v$ is witnessed by $v\fr f\mapsto{\tt write}(v,f)=v^\nn\oplus f$.
For any $f\in\btie_{v\in(V,E)}P_v$, we have ${\tt pr}_v(f)\in P_v$ for some $v\in V$.
Thus, we have ${\tt pr}_v(f)\leq_Tf$, since ${\tt pr}_v$ is partially computable, and $f\in{\rm dom}({\tt pr}_v)$.
Hence, $f\in\widehat{\rm Deg}(P_v)$.

(2) Fix $f\in [(V,E)]\oplus\bigoplus_{v\in V}P_v$.
If $f(0)=1$, then we can show the desired condition as in (1).
If $f$ is of the form $f=0\fr g$, we have $\lambda n.\lrangle{g(n),0}\in\bhtie_{v\in(V,E)}P_v$ since $g\in[(V,E)]$.
Hence, $\bhtie_{v\in(V,E)}P_v\leq^1_1[(V,E)]\oplus\bigoplus_{v\in V}P_v$.
To see $\widehat{\rm Deg}\left([(V,E)]\oplus\bigoplus_{v\in V}P_v\right)\leq^1_1\bhtie_{v\in(V,E)}P_v$, we inductively define a partial computable function ${\tt walk}:\subseteq(V\times\nn)^\nn\to V^\nn$ as follows.
Set ${\tt walk}(\lrangle{})=\lrangle{}$, and fix $\sigma=\sigma^{--}\fr\lrangle{\pair{u,m},\pair{v,n}}\in(V\times\nn)^{<\nn}$.
Assume that ${\tt walk}(\sigma^-)$ has been already defined.
Then, ${\tt walk}(\sigma)$ is defined as follows.
\[
{\tt walk}(\sigma^{--}\fr\lrangle{\pair{u,m},\pair{v,n}})=
\begin{cases}
{\tt walk}(\sigma^-)\fr\lrangle{v} & \mbox{ if }v\not=u,\\
{\tt walk}(\sigma^-) & \mbox{ otherwise.}
\end{cases}
\]

The notation ${\tt walk}$ has already been introduced in Definition \ref{def:1-2c:hyperconcat_f} with a slightly different definition, but these two notions are essentially equivalent.
Therefore, we may use the same notation.

For any $f\in\bhtie_{v\in(V,E)}P_v$, if ${\tt pr}_v(f)$ is total for some $v\in V$, then the desired condition follows as in (1).
Otherwise, ${\tt mc}(f)=\infty$, i.e., there are infinitely many $n\in\nn$ such that $(f(n+1))_0\not=(f(n))_0$.
In this case, ${\tt walk}(f)=\bigcup_{s\in\nn}{\tt walk}(f\res s)$ is an infinite path through the graph $(V,E)$.
In other words, the condition $f\in\bhtie_{v\in(V,E)}P_v$ ensures that ${\tt pr}_v(f)$ is total and belongs to $P_v$ for some $v\in V$, or otherwise ${\tt walk}(f)$ is total and belongs to $[(V,E)]$.
Consequently, $f\in\widehat{\rm Deg}\left([(V,E)]\oplus\bigoplus_{v\in V}P_v\right)$, since ${\tt pr}_v$ and ${\tt walk}$ are partial computable.
\end{proof}

\begin{prop}
Let $P,P_0,P_1,P_v$, for $v\in V$, be $\Pi^0_1$ subsets of $\nn^\nn$, uniformly.
\begin{enumerate}
\item $\btie_{v\in(T,E(T))}P_v=\bhtie_{v\in(T,E(T))}P_v$ for any well-founded tree $T\subseteq\nn^{<\nn}$.
\item $P_0\oplus P_1\equiv^{1}_{1}\btie_{v\in(V_1,E_1)}P_v\equiv^{1}_{1}\bhtie_{v\in(V_1,E_1)}P_v$, where $V_1=\{\varepsilon,0,1\}$, $E_1=\{(\varepsilon,0),(\varepsilon,1)\}$, and $P_\varepsilon=\emptyset$.
\item $P_0\tie P_1\equiv^{1}_{1}\btie_{v\in(V_2,E_2)}P_v\equiv^{1}_{1}\bhtie_{v\in(V_1,E_1)}P_v$, where $V_2=\{\varepsilon,0,1,01,10\}$, $E_2=\{(\varepsilon,0),(\varepsilon,1),(0,01),(1,10)\}$, $P_\varepsilon=\emptyset$, $P_{01}=P_1$, and $P_{10}=P_0$.
\item $P^{a+}\equiv^1_1\btie_{v\in(T_a,E(T_a))}P$ for every $a\in\mathcal{O}$, where recall the definition of $P^{a+}$ and $T_a$ in Definition \ref{def:div:trans-mc} and the notation below Proposition \ref{prop:div:ordinal-as-tree}.
\item $\bhk{P_0\vee P_1}^2_{\sf CL}\equiv^{1}_{1}\btie_{v\in(\{0,1\},\{0,1\}^2)}P_v$.
\item $\bhk{P_0\vee P_1}^3_{\sf LCM}\equiv^{1}_{1}\btie_{v\in(\nn,S)}P_v$, where $S=\{(n,n+1):n\in\nn\}$; $P_{2n}=P_0$ and $P_{2n+1}=P_1$ for any $n\in\nn$.
\item $\bhk{P\vee P}^3_{\sf LCM}\equiv^{1}_{1}\bhk{\bigvee_{n\in\nn}P}_{\sf LCM}\equiv^{1}_{1}\btie_{v\in(\nn,S)}P$.
\item $\widehat{\rm Deg}\left(\bigoplus_{v\in\nn}P_v\right)\equiv^1_1\bhk{\bigvee_{v\in\nn}P_v}_{\sf CL}\equiv^1_1\btie_{v\in(\nn,\nn^2)}P_v\equiv^1_1\left[\bcls\right]_{v\in\nn}P_v$.
\end{enumerate}
\end{prop}

\begin{proof}\upshape
(1) By Definition, $\btie_vP_v\subseteq\bhtie_vP_v$.
On the other hand, any $f\in\bhtie_{v\in(T,E(T))}P_v$ can pass at most finitely many vertices since $(T,E(T))$ has no infinite path.
In other words, the set $\{(f(n))_0:n\in\nn\}$ is finite.
By Pigeon Hole Principle, there is a vertex $v\in T$ such that $(f(n))_0=v$ occurs for infinitely many $n\in\nn$.
Then, ${\tt pr}_v(f)$ must be infinite.
Therefore, ${\tt pr}_v(f)\in[T_{P_v}]=P_v$ since $f\in{\rm Con}(T_{P_v})_{v\in V}$.
Hence, $f\in\btie_{v\in(T,E(T))}P_v$.

(2) The condition $\btie_{v\in(V_1,E_1)}P_v\leq^1_1P_0\oplus P_1$ follows from Proposition \ref{prop:3:uplow} (1).
For any $f\in \btie_{v\in(V_1,E_1)}P_v$, there is $i<2$ such that $(f(n))_0=i$ for any $n\in\nn$.
Thus, $i\fr f\in P_0\oplus P_1$.
The $(1,1)$-equivalence of $\btie_{v\in(V_1,E_1)}P_v$ and $\bhtie_{v\in(V_1,E_1)}P_v$ follows from the item (1) since $(V_1,E_1)$ is finite.

(3) Clearly, $P_0\tie_2P_1\subseteq\btie_{v\in(V_2,E_2)}P_v$.
Thus, by Proposition \ref{prop:1:concat} (2), $P_0\tie P_1\geq^1_1\btie_{v\in(V_2,E_2)}P_v$.
For $f\in\btie_{v\in(V_2,E_2)}P_v$, if $|(f(0))_0|=1$ then $\Phi(f)=f\in P_0\tie_2 P_1$.
If $|(f(0))_0|=2$, say $(f(0))_0=\lrangle{i,j}$, then $\Phi(f)={\tt write}(j,{\tt pr}_j(f))\in P_0\tie_1P_1$.
Hence, $P_0\tie P_1\leq^1_1\btie_{v\in(V_2,E_2)}P_v$ via the computable function $\Phi$.
The $(1,1)$-equivalence of $\btie_{v\in(V_1,E_1)}P_v$ and $\bhtie_{v\in(V_1,E_1)}P_v$ follows from the item (1) since $(V_1,E_1)$ is finite.

(4) If $\sigma$ is extendible to an element of $\btie_{v\in(T_a,E(T_a))}P$, there is a unique $\kappa\in T_a$ such that $\sigma$ can be represented as $\concat_{i\leq|\kappa|}{\tt write}(\kappa\res i,{\tt cut}(\sigma;i))$ for some sequence ${\tt cut}(\sigma)\in (T_P)^{|\kappa|}$.
Conversely, if $\sigma$ is extendible to an element of $P^{a+}$, there is a unique $\kappa\in T_a$ such that $\sigma$ can be represented as $(\concat_{i<|\kappa|-1}\kappa^*(i)\fr {\tt cut}(\sigma;i)\fr\sharp)\fr\kappa^*(|\kappa|-1){\tt }$ for some sequence ${\tt cut}(\sigma)\in (T_P)^{|\kappa|}$, where $\kappa^*(i)$ indicates the location of $\kappa(i)$ in the tree $T_P$.
The procedures to interchange these cuts are the desired $(1,1)$-reductions.

(5) It is easy to see that $\btie_{v\in(\{0,1\},\{0,1\}^2)}P_v=P_0\tie_\infty P_1$.
Moreover, $\bhk{P_0\vee P_1}^2_{\sf CL}\equiv^1_1P_0\tie_\infty P_1$ by Proposition \ref{prop:1-2:consistency}.

(6) For each $\sigma=\tau\fr\lrangle{\pair{i,m},\pair{j,n}}\in(\nn\times\nn)^{<\nn}$, we inductively define a computable function $\Xi(\sigma)$ as follows.
If $i=j$, then we set $\Xi(\sigma)=\Xi(\tau\fr\lrangle{\pair{i,m}})\fr\lrangle{n}$.
Otherwise, we set $\Xi(\sigma)=\Xi(\tau\fr\lrangle{\pair{i,m}})\fr\lrangle{\sharp,j,n}$.
Then, $\bhk{P_0\vee P_1}^3_{\sf LCM}\leq^{1}_{1}\btie_{v\in(\nn,S)}P_v$ is witnessed by $\Xi$.
Conversely, to see $\btie_{v\in(\nn,S)}P_v\leq^1_1\bhk{P_0\vee P_1}^3_{\sf LCM}$, we again inductively define another computable function $\Xi^*(\sigma)$, for each $\sigma\in(\nn\cup\{\sharp\})$.
Set $\Xi^*(\lrangle{})=\lrangle{}$, fix $\sigma=\sigma^{--}\fr\lrangle{j,k}\in(\nn\cup\{\sharp\})^{<\nn}$, and assume that $\Xi^*(\sigma^-)$ has been already defined.
For $w\geq v+2$, we consider the instruction ${\tt move}(v,w)=\lrangle{(v+1,0),(v+2,0),\dots,(w-2,0),(w-1,0)}\in (V\times\nn)^{w-v-1}$ to move from the tape $\Lambda_v$ to the tape $\Lambda_w$ in the dynamic tape model.
If $w<v+2$, then we assume that ${\tt move}(v,w)$ is the empty string.
Put $p(\sigma)=2\cdot{\tt count}(\sigma)+{\tt tail}(\sigma;0)$, where recall that ${\tt count}(\sigma)=\#\{n<|\sigma|:\sigma(n)=\sharp\}$.
If $j\not=\sharp$ and $k\not=\sharp$, then we define $\Xi^*(\sigma)=\Xi^*(\sigma^-)\fr{\tt move}(p(\sigma^*),p(\sigma))\fr\lrangle{\pair{p(\sigma),k}}$, where $\sigma^*$ is the last string $\Xi^*(\sigma^*)\supsetneq\Xi^*((\sigma^*)^-)$.
Otherwise, we set $\Xi^*(\sigma)=\Xi^*(\sigma^-)$.
Then, we have $\lrangle{(\Xi^*(f;n))_0,(\Xi^*(f;n+1))_0}\in\overline{S}$ for any $f\in\bhk{P_0\vee P_1}^3_{\sf LCM}$.
It is easy to verify that $\Xi^*(f)\in\btie_{v\in(\nn,S)}P_v$.

(7) The $(1,1)$-equivalence of $\bhk{P\vee P}^3_{\sf LCM}$ and $\bhk{\bigvee_{n\in\nn}P}_{\sf LCM}$ follows from Proposition \ref{prop:2:inf-disjunc-equiv} (2).
Thus, the desired condition follows from (5).

(8)
Clearly, $\bhk{\bigvee_{v\in\nn}P_v}_{\sf CL}\cap{\rm Con}(T_{P_v})_{v\in\nn}=\btie_{v\in(\nn,\nn^2)}P_v$.
Thus, the equivalence $\bhk{\bigvee_{v\in\nn}P_v}_{\sf CL}\equiv^1_1\btie_{v\in(\nn,\nn^2)}P_v\equiv^1_1\left[\bcls\right]_{v\in\nn}P_v$ follows from Proposition \ref{prop:1-2:consistency} and \ref{prop:2:inf-consistency}.
$\widehat{\rm Deg}\left(\bigcup_{v\in\nn}P_v\right)\leq^1_1\btie_{v\in(\nn,\nn^2)}P_v$ follows from Proposition \ref{prop:3:uplow} (1).
We may assume that $\Phi_e(\lrangle{})=\lrangle{}$ for each index $e\in\nn$.
We inductively define a computable function $\Gamma$ witnessing $\btie_{v\in(\nn,\nn^2)}P_v\leq^1_1\widehat{\rm Deg}\left(\bigcup_{v\in\nn}P_v\right)$.
For each $\sigma\in\nn^{<\nn}$ and $e\in\nn$, we also inductively define two parameters ${\tt act}_e(\sigma)\in\nn$ and ${\tt rq}_e(\sigma)\in\nn\cup\{-1\}$.
Here, ${\tt act}_e(\sigma)$ will represent the last stage at which the $e$-th strategy acts along $\sigma$, and ${\tt rq}_e(\sigma)\geq 0$ will indicate that the $e$-th strategy {\em requires attention}.
First we set ${\tt act}_e(\lrangle{})=0$ and ${\tt rq}_e(\lrangle{})=-1$ for each $e\in\nn$.
Inductively we assume that $\Gamma(\sigma^-)$, ${\tt act}_e(\sigma^-)$, and ${\tt rq}_e(\sigma^-)$ is already defined.
Calculate $r=\min\{{\tt rq}_e(\sigma^-):e<|\sigma|\;\&\;{\tt rq}_e(\sigma^-)>0\}$, and pick the least $e$ such that ${\tt rq}_e(\sigma^-)=r$ if such $r$ and $e$ exist.
In this case, we say that $e$ {\em acts}.
If there is no such $e$, we set $\Gamma(\sigma)=\Gamma(\sigma^-)$, ${\tt act}_e(\sigma)={\tt act}_e(\sigma^-)$, and ${\tt rq}_e(\sigma)={\tt rq}_e(\sigma^-)$.
If there is such $e$, put $\sigma^*=(\Phi_e(\sigma))^{\shft |\Phi_e(\sigma\res|{\tt act}_e(\sigma)|)|}$, i.e., $\Phi_e(\sigma)=(\Phi_e(\sigma\res|{\tt act}_e(\sigma)|)|)\fr\sigma^*$.
Then we set $\Gamma(\sigma)=\Gamma(\sigma^-)\fr {\tt write}(e,\sigma^*)$.
Then, put ${\tt rq}_e(\sigma)=-1$ and ${\tt act}_e(\sigma)=|\sigma|$.
For each $e^*\in\nn\setminus\{e\}$, set ${\tt act}_{e^*}(\sigma)={\tt act}_{e^*}(\sigma^-)$.
Moreover, if $e^*\leq|\sigma|$, ${\tt rq}_{e^*}(\sigma^-)=-1$, and $|\Phi_{e^*}(\sigma\res|{\tt act}_{e^*}(\sigma)|)|<|\Phi_{e^*}(\sigma)|$, then declare ${\tt rq}_{e^*}(\sigma)=|\sigma|$.
Otherwise, put ${\tt rq}_{e^*}(\sigma)={\tt rq}_{e^*}(\sigma^-)$.
Fix $g\in\nn^\nn$.
We claim that $\Phi_e(g)$ act infinitely often whenever $\Phi_e(g)$ is total.
Our construction ensures that only finitely many $e$'s require attentions along $g\res s$ for each $s\in\nn$.
Therefore, for $R=\{e\in\nn:{\tt rq}_e(g\res s)>0\}$, if $e\in R$, then the strategy $e$ acts by stage $s+\#R$, i.e., ${\tt act}_e(g\res s+\#R)\geq s$.
Assume that $e$ act at stage $t\in\nn$.
Then the algorithm $\Gamma(g\res t)$ writes the new information $(g\res t)^*$ of $\Phi_e(g)$ on the $e$-th tape, i.e., ${\tt pr}_e(\Gamma(g\res t))=\Phi_e(g\res t)$.
Thus, eventually, we have ${\tt pr}_e(\Gamma(g))=\Phi_e(g)$.
For any $g\in\widehat{\rm Deg}\left(\bigcup_{v\in\nn}P_v\right)$, there is an index $e\in\nn$ such that $\Phi_e(g)\in P_v$ for some $v\in\nn$.
Consequently, $\Gamma(g)\in\btie_{v\in(\nn,\nn^2)}P_v$.
\end{proof}

\begin{figure}[t]\centering
\begin{center}
\[
	\SelectTips{cm}{}
	\xymatrix @-1pc@C=0.5pc@R=0.2pc {
		& *+[o][F-]{P} & & *+[o][F-]{Q} & & 
		& *+[o][F-]{P} \ar[rr] & & *+[o][F-]{Q} & &
		& & & & &
		& \ar@/^1pc/ @{.>} [d] \\
		(1) & & *\txt{start} \ar[ul] \ar[ur] & & &
		(2) & & *\txt{start} \ar[ul] & & &
		& *+[o][F-]{P} \ar[rr] & & *+[o][F-]{Q} \ar@{.>}[ull] & &
		& *+[o][F-]{P} \ar@/^1pc/ @{.>} [u] \ar@/^1pc/[d] \\
		& & & & & 
		& & & & &
		(5) & *+[o][F-]{P} \ar[rr] & & *+[o][F-]{Q} \ar[ull] & &
		(6) & *+[o][F-]{P} \ar@/^1pc/[u] \ar@/^1pc/[d] \\
		& *+[o][F-]{P} \ar[rr] & & *+[o][F-]{Q} & & 
		& *+[o][F-]{P} \ar@<0.5ex>[rr] & & *+[o][F-]{Q} \ar@<0.5ex>[ll] & &
		& *+[o][F-]{P} \ar[rr] & & *+[o][F-]{Q} \ar[ull] & &
		& *+[o][F-]{P} \ar@/^1pc/[u] \ar@/^1pc/[d] \\
		(3) & & *\txt{start} \ar[ul] \ar[dr] & & &
		(4) & & *\txt{start} \ar[ul] \ar[ur] & & &
		& *+[o][F-]{P} \ar[rr] & & *+[o][F-]{Q} \ar[ull] & &
		& *+[o][F-]{P} \ar@/^1pc/[u] \\
		& *+[o][F-]{P} & & *+[o][F-]{Q} \ar[ll] & & 
		& & & & & 
		& & *\txt{start} \ar[ul] & & & 
		& *\txt{start} \ar[u]
	}
\]
\end{center}
 \vspace{-0.5em}
\caption{The dynamical representations of disjunction operations: (1) $\bhk{P\vee Q}_{\sf Int}$ ($P\oplus Q$); (2) $P\fr Q$; (3) $\bhk{P\vee Q}^2_{{\sf LCM}[2]}$ ($P\tie Q$); (4) $\bhk{P\vee Q}^2_{\sf CL}$ ($P\tie_\infty Q$); (5) $\bhk{P\vee Q}^3_{\sf LCM}$; (6) $\widehat{\rm Deg}(P)$, the Turing upward closure of $P$.}
  \label{fig:arrow2b}
\end{figure}
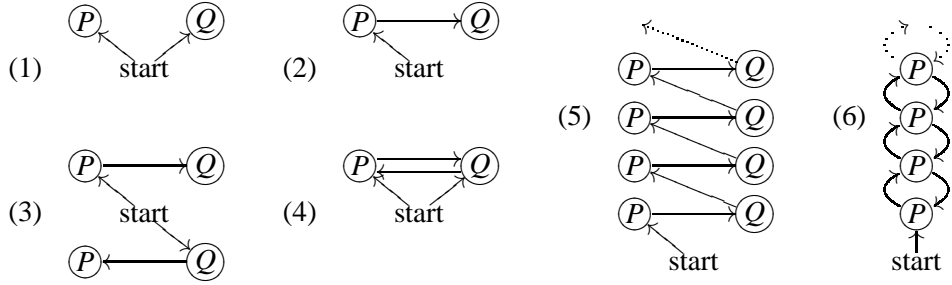

\begin{prop}
Let $(V,E)$ be a computable directed graph, and $\{P_v\}_{v\in V}$ be a computable sequence of $\Pi^0_1$ subsets of $2^\nn$.
Then we have the following.
\begin{enumerate}
\item $\btie_{v\in(V,E)}P_v$ is $\Sigma^0_3$.
\item $\bhtie_{v\in(V,E)}P_v$ is $\Pi^0_1$.
\end{enumerate}
\end{prop}

\begin{proof}\upshape
Clearly, ${\rm Con}(T_{P_v})_{v\in V}$ is $\Pi^0_1$.
Moreover, the relation $\lrangle{(f(n))_0,(f(n+1))_0}\in\overline{E}$ is computable, uniformly in $f\in(\nn\times\nn)^\nn$ and $n\in\nn$.
Thus, $\bhtie_{v\in(V,E)}P_v$ is $\Pi^0_1$.
The relation ${\tt pr}_v(f)\in P_v$ is $\Pi^0_2$ in $v\in V$ and $f\in\nn^\nn$, since it is equivalent to the following formula.
\[(\forall n\in\nn)(\exists m\in\nn)\;|{\tt pr}_v(f\res m)|>n\;\&\;{\tt pr}_v(f\res m)\in T_{P_v}.\]
Therefore, $\btie_{v\in(V,E)}P_v$ is $\Sigma^0_3$.
\end{proof}

\subsection{Infinitary Disjunctions along ill-Founded Trees}

To study $(<\omega,\omega)$-degrees, the team-learning proof model of $P$ is expected to be useful.
However, the model may be far from $\Pi^0_1$ whenever $P$ is $\Pi^0_1$.
To break out of the dilemma, the following minor modification of consistent dynamic disjunction is helpful.
For any tree $T_P\subseteq\nn^{<\nn}$ and $i\in\nn$, we let $T_P\fr\lrangle{i}$ denote the tree $T_P\cup\bigcup_{\rho\in L_P}\rho\fr\lrangle{i}$, and $T_P\fr T_Q$ denote the tree $T_P\cup\bigcup_{\rho\in L_P}\rho\fr T_Q$.
In other words, $T_P\fr T_Q$ is a corresponding tree of $P\fr Q$.

\begin{definition}\label{def:1-2d:concatetree}
\index{concatenation!along a tree}\index{$\bhtie_{\sigma\in V}P_\sigma$}%
Let $V$ be a subtree of $\nn^{<\nn}$, $\{P_\sigma\}_{\sigma\in V}$ be a computable sequence of $\Pi^0_1$ subsets of $2^\nn$, and $T_\sigma\subseteq 2^{<\nn}$ be the corresponding tree of $P_\sigma$ for each $\sigma\in V$.
Then {\em the concatenation of $\{P_\sigma\}_{\sigma\in V}$ along the tree $V$} is defined as follows.
\[\bhtie_{\sigma\in V}P_\sigma=\left[\bigcup_{\tau\in V}\left(\concat_{i<|\tau|}T_{\tau\res i}\fr\lrangle{\tau(i)}\right)\fr T_\tau\right].\]
We assume that $T_\sigma$ is the full binary tree $2^{<\nn}$ for each $\sigma\not\in V$.
Each $\alpha\in 2^{<\nn}$ is uniquely represented as 
\[\alpha=\rho_0\fr\lrangle{\tau(0)}\fr\rho_1\fr\lrangle{\tau(1)}\fr\dots\fr\lrangle{\tau(|\tau|-2)}\fr\rho_{|\tau|-1}\fr\lrangle{\tau(|\tau|-1)}\fr\beta,\]
where $\tau\in 2^{<\nn}$, $\rho(i)\in T_{\tau\res i}$ for each $i<|\tau|$, and $\beta\in T_\tau$.
For such $\tau$ and $\beta$, we set ${\tt walk}(\alpha)=\tau$, and ${\tt cut}(\alpha)=\lrangle{\rho_0,\rho_1,\dots,\rho_{|\tau|-1},\beta}$.
\index{${\tt walk}$}\index{${\tt cut}$}%
We also define ${\tt tail}^{\tt cut}(\alpha)={\tt cut}(\alpha;|{\tt walk}(\alpha)|)=\beta$.
\index{${\tt tail}^{\tt cut}$}%
Hence, each $\alpha\in 2^{<\nn}$ is represented as 
\[\alpha=\left(\concat_{i<|{\tt walk}(\alpha)|}{\tt cut}(\alpha;i)\fr\lrangle{{\tt walk}(\alpha;i)}\right)\fr{\tt cut}(\alpha;|{\tt walk}(\alpha)|).\]
Then the set $\bhtie_{\sigma\in V}P_\sigma$ is characterized as follows.
\[\bhtie_{\sigma\in V}P_\sigma=\left[\left\{\alpha\in 2^{<\nn}:{\tt walk}(\alpha)\in V\;\&\;(\forall i\leq|{\tt walk}(\alpha)|)\;{\tt cut}(\alpha;i)\in T_{{{\tt walk}(\alpha)\res i}}\right\}\right].\]
\end{definition}

\begin{remark}
The notation ${\tt walk}$ has already been introduced in Definition \ref{def:1-2c:hyperconcat_f} and the proof of Proposition \ref{prop:3:uplow}.
The meanings of the symbol ${\tt walk}$ in Definitions \ref{def:1-2c:hyperconcat_f} and \ref{def:1-2d:concatetree} are formally different, but the ideas behind these definitions are the same.
Thus, there is no confusion in using the same notation.
\end{remark}

\begin{prop}\label{prop:3b:dyn1}
Let $V$ be a computable subtree of $2^{<\nn}$, and $\{P_\sigma\}_{\sigma\in V}$ be a computable sequence of $\Pi^0_1$ subsets of $2^\nn$.
Then $\bhtie_{\sigma\in V}P_\sigma$ is $\Pi^0_1$ subset of $2^\nn$.
Moreover, $\bhtie_{\sigma\in V}P_\sigma$ is $(1,1)$-equivalent to $\bhtie_{\sigma\in (V,E(V))}P_\sigma$ in the sense of Definition \ref{def:dynamic}.
\end{prop}

\begin{proof}\upshape
Note that ${\tt walk}$, ${\tt cut}$, and ${\tt tail}^{\tt cut}$ are total computable on $\nn^{<\nn}$.
Therefore, it is $\Pi^0_1$.
Then, 
\[\Phi(\alpha)=\concat_{i\leq|{\tt walk}(\alpha)|}{\tt write}({\tt walk}(\alpha)\res i,{\tt cut}(\alpha;i))\]
 witnesses $\bhtie_{\sigma\in V}P_\sigma\geq^1_1\bhtie_{\sigma\in (V,E(V))}P_\sigma$.

Conversely, to see $\bhtie_{\sigma\in V}P_\sigma\leq^1_1\bhtie_{\sigma\in (V,E(V))}P_\sigma$, we inductively define a computable function $\Xi$.
Set $\Phi(\lrangle{})$.
Fix $\alpha=\alpha^{--}\fr\lrangle{\pair{\sigma,m},\pair{\tau,n}}\in(V\times 2)^{<\nn}$, and assume that $\Phi(\alpha^-)$ has been already defined.
If $\sigma=\tau$, then set $\Xi(\alpha)=\Xi(\alpha^-)\fr\lrangle{n}$.
If $\sigma\not=\tau$, say $\tau=\sigma\fr\lrangle{i}$, then we first calculate the least leaf ${\tt leaf}(\Xi(\alpha^-))$ of $T_{P_\sigma}$ extending $\Xi(\alpha^-)$.
Then we set $\Xi(\alpha)={\tt leaf}(\Xi(\alpha^-)\fr\lrangle{i,n}$.
Note that, for each $\alpha=\alpha^-\fr\lrangle{\pair{\tau,n}}\in(V\times 2)^{<\nn}$, we have ${\tt walk}(\Xi(\alpha))=(\alpha(|\alpha|-1))_0=\tau$, and ${\tt tail}^{\tt cut}(\Xi(\alpha))={\tt pr}_{{\tt walk}(\alpha)}(\alpha)$.
Thus, $\Xi$ witnesses $\bhtie_{\sigma\in V}P_\sigma\leq^1_1\bhtie_{\sigma\in (V,E(V))}P_\sigma$.
\end{proof}

\begin{definition}[Hyperconcatenation]\label{def:1-2d:hyperconcat}
\index{concatenation!hyper-}\index{hyperconcatenation}\index{$Q\htie P$}%
For $\Pi^0_1$ sets $P,Q\subseteq 2^\nn$, the {\em hyperconcatenation of $P$ and $Q$} is defined by
\[Q\htie P=\bhtie_{\sigma\in T_Q}P_\sigma=\{g\in 2^\nn:(\forall n)\;{\tt walk}(g\res n)\in T_Q\;\&\;(\forall n\leq |{\tt walk}(g)|)\;{\tt cut}(g;n)\in T_P\},\]
where $T_Q$ denotes the corresponding tree for $Q$, and $P_\sigma=P$ for any $\sigma\in T_Q$.
\end{definition}

\begin{remark}
For every $g\in Q\htie P$, if ${\tt walk}(g)$ is total, then ${\tt walk}(g)\in Q$, or otherwise ${\tt tail}^{\tt cut}(g)\in P$.
Therefore, the hyperconcatenation $Q\htie P$ in the sense of Definition \ref{def:1-2d:hyperconcat} can be seen as a consistent version of the hyperconcatenation $\bhk{Q\vee P}_{\Sigma^0_2}^\htie$ in the sense of Definition \ref{def:1-2c:hyperconcat_f}.
\end{remark}

To see the learnability feature of hyperconcatenation, we introduce new learnability notions.

\begin{definition}
\index{learner!confident}\index{learner!eventually-Popperian}\index{learner!eventually-Lipschitz}%
Let $\Psi$ be a learner.
\begin{enumerate}
\item $\Psi$ is {\em confident} (see also \cite{JORS}) if $\lim_s\Psi(f\res s)$ converges for every $f\in\nn^\nn$.
\item $\Psi$ is {\em eventually-Popperian} if, for every $f\in\nn^\nn$, $\Phi_{\lim_s\Psi(f\res s)}(f)$ is total whenever $\lim_s\Psi(f\res s)$ converges.
\item $\Psi$ is {\em eventually-Lipschitz} if there is a constant $c\in\nn$ such that, for every $f\in\nn^\nn$, $|\Phi_{\lim_s\Psi(f\res s)}(f\res l+c)|\geq l$ for any $l\in\nn$, whenever $\lim_s\Psi(f\res s)$ converges.
\end{enumerate}
\end{definition}

\begin{prop}~
\begin{enumerate}
\item For any set $X,Y\subseteq\nn^\nn$, if $X\leq^{<\omega}_{tt,\omega}Y$, then $X\leq^{<\omega}_{\omega}Y$ via a team of eventually-Popperian learners.
\item For any $\Sigma^0_2$ set $S\subseteq 2^\nn$ and any set $R\subseteq\nn^\nn$, if $R\leq^1_\omega S$, then it can be witnessed by an eventually-Popperian learner.
Moreover, if $S$ is $\Pi^0_1$, then it can be witnessed by a confident eventually-Popperian learner.
\item For any $\Pi^0_1$ set $P\subseteq 2^\nn$ and any set $Q\subseteq\nn^\nn$, if $P\leq^{<\omega}_{1}Q$ then $P\leq^{<\omega}_{\omega}Q$ by a team of confident learners.
\end{enumerate}
\end{prop}

\begin{proof}\upshape
(1)
Straightforward from the definition.

(2)
Fix a computable increasing sequence $\{T_i\}_{i\in\omega}$ of infinite computable trees such that $S=\bigcup_i[T_i]$.
By padding, there is a computable function $p:\nn^2\to\nn$ such that $\Phi_{p(e,n)}$ corresponds exactly to $\Phi_e$, and $p(e,n+1)>p(e,n)$ for any index $e$ and $n$.
Assume that $R\leq^1_\omega S$ via a learner $\Psi$.
We need to construct a eventually-Popperian learner $\Delta$ witnessing $R\leq^1_\omega S$.
At each stage $s$, we define a value of $\Delta(\sigma)$ for each $\sigma\in 2^s$.
For a given $\sigma\in 2^s$, we compute $q(\sigma)=\min(\{i<s:(\forall \tau\in 2^s)\;\tau\supseteq\sigma\;\rightarrow\;\tau\in T_i\}\cup\{s\})$, and put $\Delta(\sigma)=p(\Psi(\sigma),q(\sigma))$.
If $f\not\in S$, then $\lim_nq(f\res n)$ diverges.
Therefore, $\lim_n\Delta(f\res n)$ diverges.
On the other hand, if $f\in S$, then $\lim_nq(f\res n)$ converges to some $q$.
Then $\Phi_{\lim_n\Delta(f\res n)}(f)=\Phi_{p(\lim_n\Psi(f\res n),q)}(f)=\Phi_{\lim_n\Psi(f\res n)}(f)\in R$.
Consequently, $\Delta$ is eventually-Popperian, and witnesses $R\leq^1_\omega S$.
If $S$ is $\Pi^0_1$, then we modify $\Delta$ by setting $\Delta(\sigma)$ to be a fixed index of a total computable function $g\mapsto 0^\omega$, whenever $\sigma$ extends a leaf of $T_S$.
Then, $\Delta$ is also confident.

(3)
If $P\leq^{<\omega}_1Q$ via $n$ many computable functions $\{\Phi_i\}_{i<n}$, then each learner $\Psi_i$ for each $i<n$ guesses an index of $\Phi_i$.
Note that $\Psi_i$ does not change his mind.
In particular, $\Psi_i$ is confident.
\end{proof}

\begin{prop}\label{prop:3:confevePop}
Let $V$ be a computable subtree of $\nn^{<\nn}$, and $\{P_\sigma\}_{\sigma\in V}$ be a computable collection of $\Pi^0_1$ subsets of $\nn^\nn$.
Then $[(V,E)]\oplus\bigoplus_{\sigma\in\nn^{<\nn}}P_\sigma\leq^{<\omega}_{\omega}\bhtie_{\sigma\in V}P_\sigma$ by a team of a confident learner and an eventually-Popperian learner.
\end{prop}

\begin{proof}\upshape
We consider two learners: a learner $\Psi_0$ who guesses an index of $\alpha\mapsto 0\fr{\tt walk}(\alpha)$, and a learner $\Psi_1$ who guesses an index of $\alpha\mapsto\lrangle{1,{\tt walk}(\alpha)}\fr{\tt tail}^{\tt cut}(\alpha)$.
As $f\mapsto 0\fr{\tt walk}(f)$ is partial computable, $\Psi_0$ does not change his mind.
In particular, $\Psi_0$ is confident.
On $f\in\nn^\nn$, the learner $\Psi_1$ changes his mind whenever ${\tt walk}(f\res n+1)$ properly extends ${\tt walk}(f\res n)$. 
If $\lim_{n\in\nn}\Psi_1(f\res n)$ converges, then ${\tt walk}(f)$ must be partial.
Thus, ${\tt tail}^{\tt cut}(f)$ must be total.
Then, $\lrangle{1,{\tt walk}(f)}\fr{\tt tail}^{\tt cut}(f)$ is total.
Therefore, $\Psi_1$ is eventually-Popperian.
\end{proof}

\begin{prop}
Let $P_0,P_1,Q_0,Q_1$ be $\Pi^0_1$ subsets of $2^\nn$ such that $Q_0\leq^1_\omega Q_1$ via an eventually Lipschitz learner and that $P_0\leq^1_1 P_1$.
Then, $Q_0\htie P_0\leq^1_\omega Q_1\htie P_1$.
\end{prop}

\begin{proof}\upshape
For any partial computable function $\Phi$, without loss of generality, we may assume $|\Phi(\sigma)|\leq|\Phi(\sigma^-)|+1$ for any string $\sigma\in\nn^{<\nn}$.
For given indices $i$ and $j$, we effectively construct a computable function $\Phi_{{\tt hyp}(i,j)}$ as follows.
Put $\Phi_{{\tt hyp}(i,j)}(\lrangle{})=\lrangle{}$, and assume that $\Phi_{{\tt hyp}(i,j)}(\sigma^-)$ has been already defined.
Note that, either $|{\tt walk}(\sigma)|=|{\tt walk}(\sigma^-)|+1$ or $|{\tt tail}^{\tt cut}(\sigma)|=|{\tt tail}^{\tt cut}(\sigma^-)|+1$ is satisfied.
Here, the notation ${\tt tail}^{\tt cut}$ is used in referring to decomposing $Q_1\htie P_1$.
If the former is the case (i.e., $|{\tt walk}(\sigma)|=|{\tt walk}(\sigma^-)|+1$), then we extend ${\tt tail}^{\tt cut}(\Phi_{{\tt hyp}(i,j)}(\sigma^-))$ to ${\tt leaf}\circ{\tt tail}^{\tt cut}(\Phi_{{\tt hyp}(i,j)}(\sigma^-))$, the least leaf of $T_{P_0}$ extending it, and then, concatenate the bit $\Phi_i({\tt walk}(\sigma);|{\tt walk}(\sigma)|-c)$ to it.
Formally, for a string $\tau\in\nn^{<\nn}$ with $\Phi_{{\tt hyp}(i,j)}(\sigma^-)=\tau\fr{\tt tail}^{\tt cut}(\Phi_{{\tt hyp}(i,j)}(\sigma^-))$, we define
\[\Phi_{{\tt hyp}(i,j)}(\sigma)=\tau\fr{\tt leaf}\circ{\tt tail}^{\tt cut}(\Phi_{{\tt hyp}(i,j)}(\sigma^-))\fr\lrangle{\Phi_i({\tt walk}(\sigma);|{\tt walk}(\sigma)|-c)}.\]
Here, we fix some string $\rho\in T_{Q_0}$ of length $c$, and we set $\Phi_i(\sigma;k-c)=\sigma(k)$ for each $k<c$.
If $\Phi_i({\tt walk}(\sigma);|{\tt walk}(\sigma)|-c)$ is undefined, then $\Phi_{{\tt hyp}(i,j)}(\tau)$ is undefined for any $\tau\supseteq\sigma$.
If the former is not the case (then, $|{\tt tail}^{\tt cut}(\sigma)|=|{\tt tail}^{\tt cut}(\sigma^-)|+1$), then we concatenate the new values of $\Phi_j({\tt tail}^{\tt cut}(\sigma))$ to $\Phi_{{\tt hyp}(i,j)}(\sigma^-)$ if it belongs to $T_{P_0}$.
Formally, if $\Phi_j({\tt tail}^{\tt cut}(\sigma^-))\subsetneq\Phi_j({\tt tail}^{\tt cut}(\sigma))\in T_{P_0}$, say $\Phi_j({\tt tail}^{\tt cut}(\sigma))=\Phi_j({\tt tail}^{\tt cut}(\sigma^-))\fr\rho$, then we define $\Phi_{{\tt hyp}(i,j)}(\sigma)=\Phi_{{\tt hyp}(i,j)}(\sigma^-)\fr\rho$.
Otherwise, we set $\Phi_{{\tt hyp}(i,j)}(\sigma)=\Phi_{{\tt hyp}(i,j)}(\sigma^-)$.

Now assume that $P_0\leq^1_\omega P_1$ via a computable function $\Phi_e$, and $Q_0\leq^\omega Q_1$ via an eventually Lipschitz learner $\Psi$ with a constant $c$.
We construct a learner $\Delta$ witnessing $Q_0\htie P_0\leq^1_\omega Q_1\htie P_1$.
At first the learner $\Delta$ guesses the index $\Delta(\lrangle{})={\tt hyp}(\Psi(\lrangle{}),e)$.
Fix $\sigma\in\nn^{<\nn}$, and assume that $\Delta(\sigma^-)$ has been already defined.
If $\Psi({\tt walk}(\sigma))\not=\Psi({\tt walk}(\sigma^-))$, then $\Delta$ also changes his mind as $\Delta(\lrangle)={\tt hyp}(\Psi({\tt walk}(\sigma)),e)$.
Assume not.
In the case $|{\tt walk}(\sigma)|>|{\tt walk}(\sigma^-)|$, if either $|{\tt walk}(\sigma)|<c$ or ${\tt walk}(\sigma)\not\in T_{Q_1}^{ext}$ is witnessed, the learner $\Delta$ changes his mind (this situation occurs only finitely often).
Otherwise, the learner $\Delta$ keeps his previous guess, i.e., $\Delta(\sigma)=\Delta(\sigma^-)$.
In this way, it is not hard to see that we may construct a learner $\Delta$ witnessing $Q_0\htie P_0\leq^1_\omega Q_1\htie P_1$.
\end{proof}

\subsection{Nested Infinitary Disjunctions along ill-Founded Trees}

In Part II, we employ finite iterations of the hyperconcatenation $\htie$ to show that some (local) degree structures are not Brouwerian.
Beyond this, we just note that one can iterate the hyperconcatenation $\htie$ along any directed graph $(V,E)$.
However, the iteration of $\htie$ does not represented by our dynamic proof model.
We may introduce another new model called {\em the nested disjunction model}.

\medskip

\noindent
{\bf The nested tape model}: 
\index{model!nested tape}%
As an example, first we consider the nested disjunction $T^*=\bhtie_{\sigma\in T^0}\bhtie_{\tau\in T^1_\sigma}[T^2_{\sigma,\tau}]$ along the graph $G=(\{0,1,2\},\{(0,1),(1,2)\})$ with the initial vertex $\varepsilon=0$, where $T=\{T^0\}\cup\{T^1_\sigma\}_{\sigma\in\nn^{<\nn}}\cup\{T^2_{\sigma,\tau}\}_{\lrangle{\sigma,\tau}\in(\nn^{<\nn})^2}$ is a given collection of subtrees of $\nn^{<\nn}$.
The nested tape model for $T^*$ consists of a collection $\{\Lambda_\square\}\cup\{\Lambda^0\}\cup\{\Lambda^1_\sigma\}_{\sigma\in\nn^{<\nn}}\cup\{\Lambda^2_{\sigma,\tau}\}_{\lrangle{\sigma,\tau}\in(\nn^{<\nn})^2}$ of infinite tapes.

\begin{figure}[t]\centering
\begin{center}
\input{figure/nest01.tex}
\end{center}
 \vspace{-0.5em}
\caption{An example nested tape model when $G$ is a linear order of length $3$: $\lrangle{012}$ is written on $\Lambda_\square$; $\lrangle{101}$ is written on $\Lambda^0$; $\lrangle{1001}$ is written on $\Lambda^1_{101}$; then $\Lambda_\square$, $\Lambda^0$, $\Lambda^1_{101}$, and $\Lambda^2_{101,1001}$ are available.}
  \label{fig:arrow2}
\end{figure}
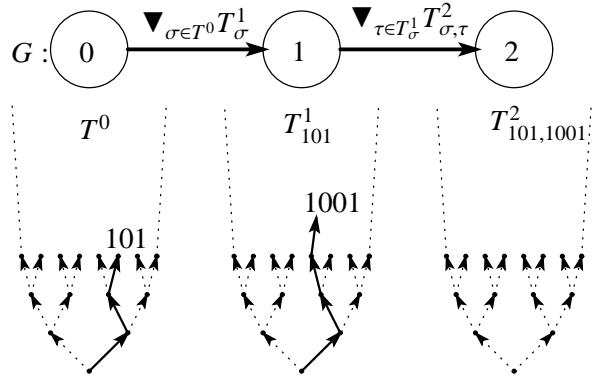

Generally, a {\em nested system $(G,T,\Lambda)$} consists of a graph $G=(V,E)$ with the initial vertex $\varepsilon$, a collection $T=\{T^v_\sigma\}_{v\in V,\sigma\in(\nn^{<\nn})^{<\nn}}$ of (ill-founded) trees, and a collection $\Lambda=\{\Lambda_\square\}\cup\{\Lambda^v_\sigma\}_{v\in V,\sigma\in(\nn^{<\nn})^{<\nn}}$ of infinite tapes.
A verifier $\Psi$ is only allowed to write a letter on tapes which are {\em available}.
Assume that a word ${\tt pr}[v,\sigma]$ is written on $\Lambda^v_\sigma$ for each $v\in V$ and $\sigma\in(\nn^{<\nn})^{<\nn}$.
Then, the availability conditions are given as follows.
\begin{itemize}
\item $\Lambda_\square$ and $\Lambda^\varepsilon_{\lrangle{}}$ are available at each stage.
\item If a finite word $v=\lrangle{v[0],v[1],\dots,v[l]}$ is written on the tape $\Lambda_\square$, then the following tapes are available.
\[\Lambda^{v[1]}_{{\tt pr}[v[0],\lrangle{}]}, \Lambda^{v[2]}_{{\tt pr}[v[1],{\tt pr}[v[0],\lrangle{}]]},\dots, \Lambda^{v[i]}_{{\tt pr}[v[i-1],{\tt pr}[v[i-2],\dots,{\tt pr}[v[1],{\tt pr}[v[0],\lrangle{}]]]]}.\]
\end{itemize}

Here, on the tape $\Lambda_\square$, the verifier $\Psi$ is only allowed to write a path starting from the initial vertex $\varepsilon$ within the graph $G=(V,E)$.

\begin{example}
On the nested tape model for $T^*$, let $\alpha\in((I\cup\{\square\})\times\nn)^{<\nn}$ be the record of a proof process of $\Psi$ by some stage, i.e., ${\tt pr}_\square(\alpha)$ and ${\tt pr}_{(v,\sigma)}(\alpha)$, for each $(v,\sigma)\in I^{<\nn}$, represent the words written on $\Lambda_\square$ and $\Lambda^v_\sigma$, respectively.
Here, $I$ denotes $V\times(\nn^{<\nn})^{<\nn}$.
If the letter $1$ representing the vertex $1\in V$ has been written on $\Lambda_\square$ (i.e., ${\tt pr}_\square(\alpha)\supseteq\lrangle{01}$), then the three tapes $\Lambda_\square$, $\Lambda^0$, and $\Lambda^1_{p}$ are available, where $p={\tt pr}_0(\alpha)$.
\end{example}

The verifier $\Psi$ {\em succeeds} if he eventually writes a correct solution on some tape from $\Lambda$ (i.e., some solution $f\in[T^v_\sigma]$ is eventually written on $\Lambda^v_\sigma$ for some $(v,\sigma)\in V\times(\nn^{<\nn})^{<\nn}$, or otherwise, some infinite path though $G$ is written on $\Lambda_{\square}$).
For each $u,v\in V$ and $(v,\sigma)\in V\times(\nn^{<\nn})^{<\nn}$, the tuple  $\lrangle{T^v_\sigma,\Lambda^v_\sigma,T^{u}_{\sigma,\tau},\Lambda^{u}_{\sigma,\tau}}_{\tau\in T^v_{\sigma}}$ is called {\em the $(\sigma,v,u)$-component of $(G,T,\Lambda)$}.
The $(\sigma,v,u)$-component of our nested system consists of an infinite disjunction along an ill-founded tree, $\bhtie_{\tau\in T^v_\sigma}[T^{u}_{\sigma,\tau}]$.
In other words, on the $(\sigma,v,u)$-component of the system $(I,\Lambda,T,G)$, the set $\Lambda^v_\sigma$ plays the role of the declaration $\square$, and $\Lambda^{u}_{\sigma,\tau}$ plays the role of the working tape for each $\tau\in T^v_\sigma$, as in the dynamic tape model.

\begin{definition}
Fix a directed graph $G=(V,E)$, and we denotes $V\times(\nn^{<\nn})^{<\nn}$ by $I$.
Assume that a collection $\{T_{(v,\sigma)}\}_{(v,\sigma)\in I}$ of subtrees of $\nn^{<\nn}$ are given.
For $\alpha\in((I\cup\{\square\})\times \nn)^{<\nn}$, we inductively define {\em the $n$-th available index along $\alpha$}, $p(\alpha,n)\in I$, for each $n\leq |{\tt pr}_{\square}(\alpha)|$, as follows.
\index{N-th available index@$n$-th available index}\index{$p(\alpha,i)$}%
\[p(\alpha,0)=(\varepsilon,\lrangle{}),\quad p(\alpha,i+1)=({\tt pr}_{\square}(\alpha)(i),(p(\alpha,i))_1\fr\lrangle{{\tt pr}_{p(\alpha,i)}(\alpha)}).\]
Then we define the set of all indices of {\em available tapes along $\alpha$} by $A(\alpha)=\{p(\alpha,n):n\leq |{\tt pr}_{\square}(\alpha)|\}$.
The set $S(\alpha)$ of {\em successors of $\alpha$} is defined as follows:
\begin{align*}
S(\alpha)=\{(p,n)\in(I\cup\{\square\})\times \nn&:p\in A(\alpha)\;\&\;{\tt pr}_p(\alpha)\fr n\in T_p\}\\
&\cup\{(\square,v):({\tt pr}_{\square}(\alpha)(|{\tt pr}_{\square}(\alpha)|-1),v)\in E\}.
\end{align*}
Then {\em the nested infinitary disjunction $\mathbf{W}_{\sigma\in I}[T_\sigma]\subseteq((I\cup\{\square\})\times \nn)^\nn$ of $\{T^v_{\sigma}\}_{(v,\sigma)\in I}$} is defined by
\index{disjunction!infinitary!nested}\index{$\mathbf{W}_{\sigma\in I}[T_\sigma]$}%
\[\mathbf{W}_{\sigma\in I}[T_\sigma]=\{f\in((I\cup\{\square\})\times \nn)^\nn:(\forall n\in\nn)\;f(n)\in S(f\res n)\}.\]
We can also define $\mathbb{W}_{\sigma\in I}[T_\sigma]=\{f\in\mathbf{W}_{\sigma\in I}[T_\sigma]:|{\tt pr}_{\square}(f)|<\infty\}$.
\index{$\mathbb{W}_{\sigma\in I}[T_\sigma]$}%
\end{definition}

\begin{prop}\label{prop:3b:pi01eq}
Assume that $G=(V,E)$ is a computable directed graph, and $\{T_\sigma\}_{\sigma\in I}$ is a computable collection of computable subtrees of $\nn^{<\nn}$, where $I=V\times(\nn^{<\nn})^{<\nn}$.
Then, $\mathbf{W}_{\sigma\in I}[T_\sigma]$ is $\Pi^0_1$.
Moreover, if $G$ and $T_\sigma$ are subtrees of $2^{<\nn}$ for each $\sigma\in I$, then $\mathbf{W}_{\sigma\in I}[T_\sigma]$ is $(1,1)$-equivalent to a $\Pi^0_1$ subset of $2^\nn$.
\end{prop}

\begin{proof}\upshape
Note that $\alpha\mapsto A(\alpha)$ is computable.
Therefore, $\alpha\mapsto S(\alpha)$ is also computable.
Thus, $\mathbf{W}_{\sigma\in I}[T_\sigma]$ is $\Pi^0_1$.

Assume that $G=(V,E(V))$ and $T_\sigma$ are subtrees of $2^{<\nn}$ for each $\sigma\in I$.
Fix new symbols $+,-$ which does not belong to $\nn$.
To construct a $\Pi^0_1$ subset of $(\{+,-\}\cup 2)^\nn$ which is $(1,1)$-equivalent to $\mathbf{W}_{\sigma\in I}[T_\sigma]$, we inductively define a computable function ${\rm head}:(\{+,-\}\cup 2)^{<\nn}\to\mathbb{Z}$.
Fix $\alpha=\alpha^-\fr\lrangle w\in(\{+,-\}\cup 2)^{<\nn}$.
Put ${\rm head}(\lrangle{})=0$,
Put ${\rm head}(\alpha)={\rm head}(\alpha^-)+1$ if $w=+$; put ${\rm head}(\alpha)={\rm head}(\alpha^-)$ if $w\not\in\{+.-\}$; and put ${\rm head}(\alpha)={\rm head}(\alpha^-)-1$ if $w=-$.
If $\alpha=\alpha^{--}\fr\lrangle{+,+}$ and ${\rm head}(\alpha)=\max\{{\rm head}(\beta):\beta\subsetneq\alpha\}+2$, or if ${\rm head}(\alpha)=-1$, then we say that $\alpha$ is {\em overflowing}.
If $\alpha$ has an overflowing initial segment $\beta\subseteq\alpha$, then we also say that $\alpha$ is overflowing.
Let ${\rm Rule}$ denote the set of all non-overflowing strings $\alpha\in(\{+,-\}\cup 2)^{<\nn}$ which has neither $\lrangle{+,-}$ nor $\lrangle{-,+}$ as substrings.
Note that ${\rm Rule}$ is computable.

We now inductively define $\tilde{\tt pr}_\square$, $\tilde{p}$, and $\tilde{\tt pr}_\sigma$ for each $\sigma\in V$.
Put $\tilde{\tt pr}_\square(\lrangle{})$, and $\tilde{p}=\lrangle{\lrangle{}}$.
Fix $\alpha=\alpha^-\fr w\in{\rm Rule}$.
Assume that $\tilde{\tt pr}_\square(\alpha^-)$, and $\tilde{p}(\alpha^-)$ have been already defined.
If $w\in\{+,-\}$, then $\tilde{\tt pr}_\square(\alpha)=\tilde{\tt pr}_\square(\alpha^-)$ and $\tilde{p}(\alpha)=\tilde{p}(\alpha^-)$.
Assume $w\not\in\{+,-\}$.
Then, if ${\rm head}(\alpha)>\max\{{\rm head}(\beta):\beta\subsetneq\alpha\}$, then we define $\tilde{\tt pr}_\square(\alpha)=\tilde{\tt pr}_\square(\alpha^-)\fr w$.
Otherwise, set $\tilde{\tt pr}_\square(\alpha)=\tilde{\tt pr}_\square(\alpha^-)$.
If $\tilde{\tt pr}_\square(\alpha)\not=\tilde{\tt pr}_\square(\alpha^-)$, then $\tilde{p}(\alpha)=\tilde{p}(\alpha^-)\fr\lrangle{\lrangle{}}$.
Otherwise, define $\tilde{p}(\alpha)\in(2^{<\nn})^{|V(\alpha)|}$ as follows.
\[
(\tilde{p}(\alpha))(n)=
\begin{cases}
(\tilde{p}(\alpha^-))(n), & \mbox{ if }n<{\rm head}(\alpha);\\
(\tilde{p}(\alpha^-))(n)\fr w, & \mbox{ if }n={\rm head}(\alpha);\\
\lrangle{}, & \mbox{ if }h(\alpha)<n\leq|\tilde{\tt pr}_\square(\alpha)|.
\end{cases}
\]
Then, for each $\sigma\in V$, we define $\tilde{\tt pr}_\sigma(\alpha)=(\tilde{p}(\beta))(|\sigma|)$ for the greatest $\beta\subseteq\alpha$ such that $\sigma\subseteq\tilde{p}(\beta)$.
Set ${\rm Rule}_{\forall}=\{f\in(\{+,-\}\cup 2)^\nn:(\forall n\in\nn)\;f\res n\in{\rm Rule}\}$.
Note that any $g\in{\rm Rule}_{\forall}$ has no infinite $\{+,-\}$-sequence; otherwise $g\res s$ for some $s\in\nn$ is overflowing or has a substring $\lrangle{+,-}$ or $\lrangle{-,+}$, and hence $g\res s$ must go against ${\rm Rule}$.
Then $P$ is defined as follows.
\[P=\{f\in{\rm Rule}_{\forall}:(\forall n\in\nn)\;(\tilde{\tt pr}_\square(f\res n)\in V\;\&\;(\forall\sigma\in I)\;\tilde{\tt pr}_\sigma(f\res n)\in T_\sigma)\}.\]
Clearly, $P$ is computably bounded, and $\Pi^0_1$.
It remains to show that $P\equiv^1_1\mathbf{W}_{\sigma\in I}[T_\sigma]$.
We first inductively define a computable function $\Phi$ witnessing $P\geq^1_1\mathbf{W}_{\sigma\in I}[T_\sigma]$.
Set $\Phi(\lrangle{})=\lrangle{}$, fix $\alpha=\alpha^-\fr w\in{\rm Rule}$, and assume that $\Phi(\alpha^-)$ has been already defined.
If $w\in\{+.-\}$, then set $\Phi(\alpha)=\Phi(\alpha^-)$.
Assume $w\not\in\{+,-\}$.
If ${\rm head}(\alpha)>\max\{{\rm head}(\beta):\beta\subsetneq\alpha\}$, then we set $\Phi(\alpha)=\Phi(\alpha^-)\fr\lrangle{(\square,w)}$.
Otherwise, we set $\Phi(\alpha)=\Phi(\alpha^-)\fr\lrangle{((\tilde{\tt pr}_\square(\alpha),\tilde{p}(\alpha)\res h(\alpha)),w)}$.
It is not hard to check $P\geq^1_1\mathbf{W}_{\sigma\in I}[T_\sigma]$ via $\Phi$.

To prove $P\geq^1_1\mathbf{W}_{\sigma\in I}[T_\sigma]$, we first define a computable function ${\rm head}^*$.
Firstly put ${\rm head}^*(\lrangle{})=0$.
Fix $\alpha=\alpha^-\fr\lrangle{(\sigma,w)}\in((I\cup\{\square\})\cup\nn)^{<\nn}$.
If $\sigma=\square$, then we set ${\rm head}^*(\alpha)=|{\tt pr}_\square(\alpha)|$.
If $\sigma\in I$, then we set ${\rm head}^*(\alpha)=|(\sigma)_1|$.
Set $\Phi(\lrangle{})=\lrangle{}$, and assume that $\Phi(\alpha^-)$ has already been defined.
Put $d={\rm head}^*(\alpha)-{\rm head}^*(\alpha^-)$.
If $d\geq 0$, then $\Phi(\alpha)=\Phi(\alpha^-)\fr +^d\fr w$.
If $d<0$, then $\Phi(\alpha)=\Phi(\alpha^-)\fr -^{-d}\fr w$.
It is not hard to check $P\leq^1_1\mathbf{W}_{\sigma\in I}[T_\sigma]$ via $\Phi$.
\end{proof}

If $T^v_\sigma$ only depends on $v\in V$, i.e., $T^{v}_\sigma=T_v$, then the nested system $(I,\Lambda,T,G)$ can be viewed as the iteration of the hyperconcatenation $\htie$ along the graph $G$.
In this case, we write $\mathbf{W}_{v\in(V,E)}P_v$ for this notion.

\begin{prop}
Let $(V,E)$ be a computable directed graph, and $\{P_v\}_{v\in V}$ be a computable collection of $\Pi^0_1$ subsets of $\nn^\nn$.
Then, $\mathbf{W}_{v\in(V,E)}P_v\leq^1_1\bhtie_{v\in(V,E)}P_v$.
\end{prop}

\begin{proof}\upshape
We inductively define a computable function $\Phi$ which witnesses the condition $\mathbf{W}_{v\in(V,E)}P_v\leq^1_1\bhtie_{v\in(V,E)}P_v$.
Set $\Phi(\lrangle{})=\lrangle{}$.
Fix $\alpha=\alpha^{--}\fr\lrangle{(u,i),(v,j)}\in(V\times\nn)^{<\nn}$.
Assume that $\Phi(\alpha^-)$ has already been defined, and $\Phi(\alpha^-)$ is of the form $\Phi(\alpha^-)=\beta\fr\lrangle{(\sigma,k)}$ for some $\beta\in((I\cup\{\square\})\times\nn)^{<\nn}$, $\sigma\in I\cup\{\square\}$, and $k\in\nn$.
If $v=u$, then we set $\Phi(\alpha)=\Phi(\alpha^-)\fr\lrangle{(\sigma,j)}$.
If $v\not=u$, then we set $\Phi(\alpha)=\Phi(\alpha^-)\fr\lrangle{(\square,v),((v,(\sigma)_1\fr{\tt pr}_u(\alpha)),j)}$.
Fix $g\in\bhtie_{v\in(V,E)}P_v$.
By induction, we can show ${\tt pr}_{v[n]}(g\res n+1)={\tt pr}_{\sigma[n]}(\Phi(g\res n+1))$, where $g(n)=(v[n],j)$ and $\Phi(g\res n+1)=\beta\fr\lrangle{(\sigma[n],j)}$.
Then, $(\sigma[n])_1=(\sigma[n]^-)_1\fr{\tt pr}_{\sigma[n]^-}(\Phi(g\res n+1))$, by our definition of $\Phi$.
Therefore, $\sigma[n]$ is available whenever $\sigma[n]^-$ is available.
By induction, $\sigma[n]$ is available at $g\res n$, for any $n\in\nn$.
Moreover, ${\tt pr}_{\sigma[n]}(\Phi(g))={\tt pr}_{v[n]}(g)\in T_{v[n]}=T_{\sigma[n]}$, and ${\tt pr}_\square(\Phi(g))={\rm walk}(g)$.
Here ${\rm walk}(g)$ is inductively defined as follows.
Set ${\rm walk}(g\res 1)=(g(0))_0$.
If $(g(n+1))_0=(g(n))_0$, then ${\rm walk}(g\res n+1)={\rm walk}(g\res n)$.
If $(g(n+1))_0\not=(g(n))_0$, then ${\rm walk}(g\res n+1)={\rm walk}(g\res n)\fr(g(n+1))_0$.
Note that $\lrangle{{\rm walk}(g;n),{\rm walk(g;n+1)}}\in E$ for each $n<|{\rm walk}(g)|-1$.
Thus, $\Phi(g;s)\in S(\Phi(g)\res s)$ for any $s\in\nn$.
Consequently, $\Phi(g)\in\mathbf{W}_{v\in(V,E)}P_v$.
\end{proof}

If $G=(V,E)$ is linearly ordered, then we have no choice of the next vertex at each stage.
In this case, to simplify our argument, we assume that only $\{\Lambda_\sigma\}_{\lrangle{v,\sigma}\in I}$ is given, i.e., the $(v,\sigma)$-th tape $\Lambda^v_\sigma$ does not depend on the vertex $v\in V$, and.
Moreover, if $T_\sigma=T_\tau$ for any $\sigma,\tau\in I$, then we only require $\{\Lambda_{|\sigma|}\}_{\lrangle{v,\sigma}\in I}$.
We will use the simplest depth $n$ nested system.
The system $(G,T,\Lambda)$ is {\em an $\nn^{<n}$-nested system} if $G=(n,S)$ and $T_\sigma=T_\tau$ for any $\sigma,\tau\in I$.
\index{N-nested system@$\nn^{<n}$-nested system}%
This system is equivalent to the $n$-th iteration of $\bhtie$.
Let $P^{\hjump{0}}=P$, and $P^{\hjump{n+1}}=P\htie P^{\hjump{n}}$.
\index{$P^{\hjump{n}}$}%
We also write $\bhtie P$ for $\bigcup_{n\in\nn}P^{\hjump{n}}$.
\index{$\bhtie P$}%

\begin{prop}
Let $G=(n+2,S)$, where $n+2=\{m\in\nn:m<n+2\}$ and $S=\{(m,m+1):m\leq n\}$, and $\{P^v_\sigma\}_{\lrangle{v,\sigma}\in I}$ be a computable collection of $\Pi^0_1$ subsets of $\nn^\nn$.
Let $T^v_\sigma$ denote the corresponding tree of $P^v_\sigma$ for each $\lrangle{v,\sigma}\in I$.
Then $\mathbf{W}_{\lrangle{v,\sigma}\in I}P^v_\sigma$ is $(1,1)$-equivalent to the following set:
\[Q=\bhtie_{\sigma(0)\in T^0_{\lrangle{}}}\left(\bhtie_{\sigma(1)\in T^1_{\sigma(0)}}\left(\dots\left(\bhtie_{\sigma(n)\in T^n_{\sigma(0),\dots,\sigma(n-1)}}P^{n+1}_{\sigma(0),\dots,\sigma(n)}\right)\dots\right)\right).\]
In particular, $\mathbf{W}_{v\in(n,S)}P=P^{\hjump{n}}$ for any $\Pi^0_1$ subset of $\nn^\nn$.
\end{prop}

\begin{proof}\upshape
Straightforward.
\end{proof}

\begin{remark}
We may introduce a transfinite iteration $P^{\hjump{a}}$ of hyperconcatenation as in Definition \ref{def:div:trans-mc}, or equivalently, as a nested infinitary disjunction $\mathbf{W}_{\sigma\in (T_a,E(T_a))}P$ along the well-founded tree $T_a$.
Recall from Corollary \ref{cor:1-2c:computable-along} that the hyperconcatenation $\htie$ induces ${\rm dec}^{<\omega}_{\rm d}[\Pi^0_2]{\rm dec}^\omega_{\rm p}[\Pi^0_1]$.
The induced piecewise computability concept becomes the $a$-indexed version of ${\rm dec}^{<\omega}_{\rm d}[\Pi^0_2]{\rm dec}^\omega_{\rm p}[\Pi^0_1]$.
\end{remark}

\begin{remark}
We may introduce the ``{\em nested nested}'' model, the ``{\em nested nested nested}'' model, and so on.
Let $Q\mathbf{w}P$ be $\mathbf{W}_{v\in(T_Q,E(T_Q))}P_v$, where $P_v=P$ for each $v\in T_Q$.
\index{$Q\mathbf{w}P$}%
Then, for example, the nested nested model can be introduced as the iteration of $\mathbf{w}$ along any directed graph $(V,E)$.
Therefore, inside the $(\omega,1)$-degree of any $\Pi^0_1$ set $P\subseteq 2^\nn$, one may iterate this procedure as ``{\em nested nested nested $\dots$ nested nested $\dots$}''
Actually one may iterate ``{\em nested nested nested $\dots$ nested nested $\dots$}'' along any directed graph, for example, along the corresponding tree of $P$.
If we call it a ``{\em hypernested}'' model, then, of course, we may introduce models which are ``{\em hypernested hypernested}'', and ``{\em hypernested hypernested hypernested}'', and so on.
By iterating this notion along the corresponding tree of $P$, we obtain a ``{\em hyperhypernested}'' model.
Iterating this procedure, of course, we have the iteration of ``{\em hyper}'' along the correspoding tree of $P$.
This observation reveals to us that there are a fine structure, a deep hierarchy, and a morass inside each $(\omega,1)$-degree (or equivalently, each Turing upward closure) of a $\Pi^0_1$ subset of $2^\nn$.
\end{remark}

%% file: figure/nest01.tex
\unitlength 0.1in
\begin{picture}( 30.0000, 19.0000)(  4.0000,-22.0000)
%
\special{pn 8}%
\special{sh 1}%
\special{ar 800 2200 10 10 0  6.28318530717959E+0000}%
\special{sh 1}%
\special{ar 600 2000 10 10 0  6.28318530717959E+0000}%
\special{sh 1}%
\special{ar 1000 2000 10 10 0  6.28318530717959E+0000}%
\special{sh 1}%
\special{ar 500 1800 10 10 0  6.28318530717959E+0000}%
\special{sh 1}%
\special{ar 700 1800 10 10 0  6.28318530717959E+0000}%
\special{sh 1}%
\special{ar 900 1800 10 10 0  6.28318530717959E+0000}%
\special{sh 1}%
\special{ar 1100 1800 10 10 0  6.28318530717959E+0000}%
\special{sh 1}%
\special{ar 450 1600 10 10 0  6.28318530717959E+0000}%
\special{sh 1}%
\special{ar 550 1600 10 10 0  6.28318530717959E+0000}%
\special{sh 1}%
\special{ar 650 1600 10 10 0  6.28318530717959E+0000}%
\special{sh 1}%
\special{ar 750 1600 10 10 0  6.28318530717959E+0000}%
\special{sh 1}%
\special{ar 850 1600 10 10 0  6.28318530717959E+0000}%
\special{sh 1}%
\special{ar 950 1600 10 10 0  6.28318530717959E+0000}%
\special{sh 1}%
\special{ar 1050 1600 10 10 0  6.28318530717959E+0000}%
\special{sh 1}%
\special{ar 1150 1600 10 10 0  6.28318530717959E+0000}%
\special{sh 1}%
\special{ar 1150 1600 10 10 0  6.28318530717959E+0000}%
%
\special{pn 8}%
\special{pa 800 2200}%
\special{pa 600 2000}%
\special{dt 0.045}%
\special{sh 1}%
\special{pa 600 2000}%
\special{pa 634 2062}%
\special{pa 638 2038}%
\special{pa 662 2034}%
\special{pa 600 2000}%
\special{fp}%
\special{pa 600 2000}%
\special{pa 500 1800}%
\special{dt 0.045}%
\special{sh 1}%
\special{pa 500 1800}%
\special{pa 512 1870}%
\special{pa 524 1848}%
\special{pa 548 1852}%
\special{pa 500 1800}%
\special{fp}%
\special{pa 500 1800}%
\special{pa 450 1600}%
\special{dt 0.045}%
\special{sh 1}%
\special{pa 450 1600}%
\special{pa 448 1670}%
\special{pa 464 1652}%
\special{pa 486 1660}%
\special{pa 450 1600}%
\special{fp}%
\special{pa 500 1800}%
\special{pa 550 1600}%
\special{dt 0.045}%
\special{sh 1}%
\special{pa 550 1600}%
\special{pa 514 1660}%
\special{pa 538 1652}%
\special{pa 554 1670}%
\special{pa 550 1600}%
\special{fp}%
\special{pa 600 2000}%
\special{pa 700 1800}%
\special{dt 0.045}%
\special{sh 1}%
\special{pa 700 1800}%
\special{pa 652 1852}%
\special{pa 676 1848}%
\special{pa 688 1870}%
\special{pa 700 1800}%
\special{fp}%
\special{pa 700 1800}%
\special{pa 650 1600}%
\special{dt 0.045}%
\special{sh 1}%
\special{pa 650 1600}%
\special{pa 648 1670}%
\special{pa 664 1652}%
\special{pa 686 1660}%
\special{pa 650 1600}%
\special{fp}%
\special{pa 700 1800}%
\special{pa 750 1600}%
\special{dt 0.045}%
\special{sh 1}%
\special{pa 750 1600}%
\special{pa 714 1660}%
\special{pa 738 1652}%
\special{pa 754 1670}%
\special{pa 750 1600}%
\special{fp}%
\special{pa 800 2200}%
\special{pa 1000 2000}%
\special{dt 0.045}%
\special{sh 1}%
\special{pa 1000 2000}%
\special{pa 940 2034}%
\special{pa 962 2038}%
\special{pa 968 2062}%
\special{pa 1000 2000}%
\special{fp}%
\special{pa 1000 2000}%
\special{pa 900 1800}%
\special{dt 0.045}%
\special{sh 1}%
\special{pa 900 1800}%
\special{pa 912 1870}%
\special{pa 924 1848}%
\special{pa 948 1852}%
\special{pa 900 1800}%
\special{fp}%
\special{pa 1000 2000}%
\special{pa 1100 1800}%
\special{dt 0.045}%
\special{sh 1}%
\special{pa 1100 1800}%
\special{pa 1052 1852}%
\special{pa 1076 1848}%
\special{pa 1088 1870}%
\special{pa 1100 1800}%
\special{fp}%
\special{pa 1100 1800}%
\special{pa 1050 1600}%
\special{dt 0.045}%
\special{sh 1}%
\special{pa 1050 1600}%
\special{pa 1048 1670}%
\special{pa 1064 1652}%
\special{pa 1086 1660}%
\special{pa 1050 1600}%
\special{fp}%
\special{pa 1100 1800}%
\special{pa 1150 1600}%
\special{dt 0.045}%
\special{sh 1}%
\special{pa 1150 1600}%
\special{pa 1114 1660}%
\special{pa 1138 1652}%
\special{pa 1154 1670}%
\special{pa 1150 1600}%
\special{fp}%
\special{pa 900 1800}%
\special{pa 950 1600}%
\special{dt 0.045}%
\special{sh 1}%
\special{pa 950 1600}%
\special{pa 914 1660}%
\special{pa 938 1652}%
\special{pa 954 1670}%
\special{pa 950 1600}%
\special{fp}%
\special{pa 900 1800}%
\special{pa 850 1600}%
\special{dt 0.045}%
\special{sh 1}%
\special{pa 850 1600}%
\special{pa 848 1670}%
\special{pa 864 1652}%
\special{pa 886 1660}%
\special{pa 850 1600}%
\special{fp}%
%
\special{pn 13}%
\special{pa 800 2200}%
\special{pa 1000 2000}%
\special{fp}%
\special{sh 1}%
\special{pa 1000 2000}%
\special{pa 940 2034}%
\special{pa 962 2038}%
\special{pa 968 2062}%
\special{pa 1000 2000}%
\special{fp}%
\special{pa 1000 2000}%
\special{pa 900 1800}%
\special{fp}%
\special{sh 1}%
\special{pa 900 1800}%
\special{pa 912 1870}%
\special{pa 924 1848}%
\special{pa 948 1852}%
\special{pa 900 1800}%
\special{fp}%
\special{pa 900 1800}%
\special{pa 950 1600}%
\special{fp}%
\special{sh 1}%
\special{pa 950 1600}%
\special{pa 914 1660}%
\special{pa 938 1652}%
\special{pa 954 1670}%
\special{pa 950 1600}%
\special{fp}%
%
\special{pn 8}%
\special{pa 450 1600}%
\special{pa 400 800}%
\special{dt 0.045}%
\special{pa 1150 1600}%
\special{pa 1200 800}%
\special{dt 0.045}%
%
\special{pn 8}%
\special{ar 800 520 200 200  0.0000000 6.2831853}%
\put(7.5000,-5.9000){\makebox(0,0)[lb]{$0$}}%
%
\special{pn 8}%
\special{sh 1}%
\special{ar 1900 2200 10 10 0  6.28318530717959E+0000}%
\special{sh 1}%
\special{ar 1700 2000 10 10 0  6.28318530717959E+0000}%
\special{sh 1}%
\special{ar 2100 2000 10 10 0  6.28318530717959E+0000}%
\special{sh 1}%
\special{ar 1600 1800 10 10 0  6.28318530717959E+0000}%
\special{sh 1}%
\special{ar 1800 1800 10 10 0  6.28318530717959E+0000}%
\special{sh 1}%
\special{ar 2000 1800 10 10 0  6.28318530717959E+0000}%
\special{sh 1}%
\special{ar 2200 1800 10 10 0  6.28318530717959E+0000}%
\special{sh 1}%
\special{ar 1550 1600 10 10 0  6.28318530717959E+0000}%
\special{sh 1}%
\special{ar 1650 1600 10 10 0  6.28318530717959E+0000}%
\special{sh 1}%
\special{ar 1750 1600 10 10 0  6.28318530717959E+0000}%
\special{sh 1}%
\special{ar 1850 1600 10 10 0  6.28318530717959E+0000}%
\special{sh 1}%
\special{ar 1950 1600 10 10 0  6.28318530717959E+0000}%
\special{sh 1}%
\special{ar 2050 1600 10 10 0  6.28318530717959E+0000}%
\special{sh 1}%
\special{ar 2150 1600 10 10 0  6.28318530717959E+0000}%
\special{sh 1}%
\special{ar 2250 1600 10 10 0  6.28318530717959E+0000}%
\special{sh 1}%
\special{ar 2250 1600 10 10 0  6.28318530717959E+0000}%
%
\special{pn 8}%
\special{pa 1900 2200}%
\special{pa 1700 2000}%
\special{dt 0.045}%
\special{sh 1}%
\special{pa 1700 2000}%
\special{pa 1734 2062}%
\special{pa 1738 2038}%
\special{pa 1762 2034}%
\special{pa 1700 2000}%
\special{fp}%
\special{pa 1700 2000}%
\special{pa 1600 1800}%
\special{dt 0.045}%
\special{sh 1}%
\special{pa 1600 1800}%
\special{pa 1612 1870}%
\special{pa 1624 1848}%
\special{pa 1648 1852}%
\special{pa 1600 1800}%
\special{fp}%
\special{pa 1600 1800}%
\special{pa 1550 1600}%
\special{dt 0.045}%
\special{sh 1}%
\special{pa 1550 1600}%
\special{pa 1548 1670}%
\special{pa 1564 1652}%
\special{pa 1586 1660}%
\special{pa 1550 1600}%
\special{fp}%
\special{pa 1600 1800}%
\special{pa 1650 1600}%
\special{dt 0.045}%
\special{sh 1}%
\special{pa 1650 1600}%
\special{pa 1614 1660}%
\special{pa 1638 1652}%
\special{pa 1654 1670}%
\special{pa 1650 1600}%
\special{fp}%
\special{pa 1700 2000}%
\special{pa 1800 1800}%
\special{dt 0.045}%
\special{sh 1}%
\special{pa 1800 1800}%
\special{pa 1752 1852}%
\special{pa 1776 1848}%
\special{pa 1788 1870}%
\special{pa 1800 1800}%
\special{fp}%
\special{pa 1800 1800}%
\special{pa 1750 1600}%
\special{dt 0.045}%
\special{sh 1}%
\special{pa 1750 1600}%
\special{pa 1748 1670}%
\special{pa 1764 1652}%
\special{pa 1786 1660}%
\special{pa 1750 1600}%
\special{fp}%
\special{pa 1800 1800}%
\special{pa 1850 1600}%
\special{dt 0.045}%
\special{sh 1}%
\special{pa 1850 1600}%
\special{pa 1814 1660}%
\special{pa 1838 1652}%
\special{pa 1854 1670}%
\special{pa 1850 1600}%
\special{fp}%
\special{pa 1900 2200}%
\special{pa 2100 2000}%
\special{dt 0.045}%
\special{sh 1}%
\special{pa 2100 2000}%
\special{pa 2040 2034}%
\special{pa 2062 2038}%
\special{pa 2068 2062}%
\special{pa 2100 2000}%
\special{fp}%
\special{pa 2100 2000}%
\special{pa 2000 1800}%
\special{dt 0.045}%
\special{sh 1}%
\special{pa 2000 1800}%
\special{pa 2012 1870}%
\special{pa 2024 1848}%
\special{pa 2048 1852}%
\special{pa 2000 1800}%
\special{fp}%
\special{pa 2100 2000}%
\special{pa 2200 1800}%
\special{dt 0.045}%
\special{sh 1}%
\special{pa 2200 1800}%
\special{pa 2152 1852}%
\special{pa 2176 1848}%
\special{pa 2188 1870}%
\special{pa 2200 1800}%
\special{fp}%
\special{pa 2200 1800}%
\special{pa 2150 1600}%
\special{dt 0.045}%
\special{sh 1}%
\special{pa 2150 1600}%
\special{pa 2148 1670}%
\special{pa 2164 1652}%
\special{pa 2186 1660}%
\special{pa 2150 1600}%
\special{fp}%
\special{pa 2200 1800}%
\special{pa 2250 1600}%
\special{dt 0.045}%
\special{sh 1}%
\special{pa 2250 1600}%
\special{pa 2214 1660}%
\special{pa 2238 1652}%
\special{pa 2254 1670}%
\special{pa 2250 1600}%
\special{fp}%
\special{pa 2000 1800}%
\special{pa 2050 1600}%
\special{dt 0.045}%
\special{sh 1}%
\special{pa 2050 1600}%
\special{pa 2014 1660}%
\special{pa 2038 1652}%
\special{pa 2054 1670}%
\special{pa 2050 1600}%
\special{fp}%
\special{pa 2000 1800}%
\special{pa 1950 1600}%
\special{dt 0.045}%
\special{sh 1}%
\special{pa 1950 1600}%
\special{pa 1948 1670}%
\special{pa 1964 1652}%
\special{pa 1986 1660}%
\special{pa 1950 1600}%
\special{fp}%
%
\special{pn 8}%
\special{pa 1550 1600}%
\special{pa 1500 800}%
\special{dt 0.045}%
\special{pa 2250 1600}%
\special{pa 2300 800}%
\special{dt 0.045}%
%
\special{pn 13}%
\special{pa 1900 2200}%
\special{pa 2100 2000}%
\special{fp}%
\special{sh 1}%
\special{pa 2100 2000}%
\special{pa 2040 2034}%
\special{pa 2062 2038}%
\special{pa 2068 2062}%
\special{pa 2100 2000}%
\special{fp}%
\special{pa 2100 2000}%
\special{pa 2000 1800}%
\special{fp}%
\special{sh 1}%
\special{pa 2000 1800}%
\special{pa 2012 1870}%
\special{pa 2024 1848}%
\special{pa 2048 1852}%
\special{pa 2000 1800}%
\special{fp}%
\special{pa 2000 1800}%
\special{pa 1950 1600}%
\special{fp}%
\special{sh 1}%
\special{pa 1950 1600}%
\special{pa 1948 1670}%
\special{pa 1964 1652}%
\special{pa 1986 1660}%
\special{pa 1950 1600}%
\special{fp}%
\special{pa 1950 1600}%
\special{pa 1976 1400}%
\special{fp}%
\special{sh 1}%
\special{pa 1976 1400}%
\special{pa 1948 1464}%
\special{pa 1968 1454}%
\special{pa 1988 1470}%
\special{pa 1976 1400}%
\special{fp}%
%
\special{pn 8}%
\special{ar 1900 520 200 200  0.0000000 6.2831853}%
\put(18.5000,-5.8000){\makebox(0,0)[lb]{$1$}}%
\put(10.7000,-4.7000){\makebox(0,0)[lb]{$\bhtie_{\sigma\in T^0}T^1_\sigma$}}%
\put(21.5000,-4.7000){\makebox(0,0)[lb]{$\bhtie_{\tau\in T^1_{\sigma}}T^2_{\sigma,\tau}$}}%
%
\special{pn 8}%
\special{ar 3000 520 200 200  0.0000000 6.2831853}%
\put(29.5000,-5.8000){\makebox(0,0)[lb]{$2$}}%
%
\special{pn 8}%
\special{sh 1}%
\special{ar 3000 2200 10 10 0  6.28318530717959E+0000}%
\special{sh 1}%
\special{ar 2800 2000 10 10 0  6.28318530717959E+0000}%
\special{sh 1}%
\special{ar 3200 2000 10 10 0  6.28318530717959E+0000}%
\special{sh 1}%
\special{ar 2700 1800 10 10 0  6.28318530717959E+0000}%
\special{sh 1}%
\special{ar 2900 1800 10 10 0  6.28318530717959E+0000}%
\special{sh 1}%
\special{ar 3100 1800 10 10 0  6.28318530717959E+0000}%
\special{sh 1}%
\special{ar 3300 1800 10 10 0  6.28318530717959E+0000}%
\special{sh 1}%
\special{ar 2650 1600 10 10 0  6.28318530717959E+0000}%
\special{sh 1}%
\special{ar 2750 1600 10 10 0  6.28318530717959E+0000}%
\special{sh 1}%
\special{ar 2850 1600 10 10 0  6.28318530717959E+0000}%
\special{sh 1}%
\special{ar 2950 1600 10 10 0  6.28318530717959E+0000}%
\special{sh 1}%
\special{ar 3050 1600 10 10 0  6.28318530717959E+0000}%
\special{sh 1}%
\special{ar 3150 1600 10 10 0  6.28318530717959E+0000}%
\special{sh 1}%
\special{ar 3250 1600 10 10 0  6.28318530717959E+0000}%
\special{sh 1}%
\special{ar 3350 1600 10 10 0  6.28318530717959E+0000}%
\special{sh 1}%
\special{ar 3350 1600 10 10 0  6.28318530717959E+0000}%
%
\special{pn 8}%
\special{pa 3000 2200}%
\special{pa 2800 2000}%
\special{dt 0.045}%
\special{sh 1}%
\special{pa 2800 2000}%
\special{pa 2834 2062}%
\special{pa 2838 2038}%
\special{pa 2862 2034}%
\special{pa 2800 2000}%
\special{fp}%
\special{pa 2800 2000}%
\special{pa 2700 1800}%
\special{dt 0.045}%
\special{sh 1}%
\special{pa 2700 1800}%
\special{pa 2712 1870}%
\special{pa 2724 1848}%
\special{pa 2748 1852}%
\special{pa 2700 1800}%
\special{fp}%
\special{pa 2700 1800}%
\special{pa 2650 1600}%
\special{dt 0.045}%
\special{sh 1}%
\special{pa 2650 1600}%
\special{pa 2648 1670}%
\special{pa 2664 1652}%
\special{pa 2686 1660}%
\special{pa 2650 1600}%
\special{fp}%
\special{pa 2700 1800}%
\special{pa 2750 1600}%
\special{dt 0.045}%
\special{sh 1}%
\special{pa 2750 1600}%
\special{pa 2714 1660}%
\special{pa 2738 1652}%
\special{pa 2754 1670}%
\special{pa 2750 1600}%
\special{fp}%
\special{pa 2800 2000}%
\special{pa 2900 1800}%
\special{dt 0.045}%
\special{sh 1}%
\special{pa 2900 1800}%
\special{pa 2852 1852}%
\special{pa 2876 1848}%
\special{pa 2888 1870}%
\special{pa 2900 1800}%
\special{fp}%
\special{pa 2900 1800}%
\special{pa 2850 1600}%
\special{dt 0.045}%
\special{sh 1}%
\special{pa 2850 1600}%
\special{pa 2848 1670}%
\special{pa 2864 1652}%
\special{pa 2886 1660}%
\special{pa 2850 1600}%
\special{fp}%
\special{pa 2900 1800}%
\special{pa 2950 1600}%
\special{dt 0.045}%
\special{sh 1}%
\special{pa 2950 1600}%
\special{pa 2914 1660}%
\special{pa 2938 1652}%
\special{pa 2954 1670}%
\special{pa 2950 1600}%
\special{fp}%
\special{pa 3000 2200}%
\special{pa 3200 2000}%
\special{dt 0.045}%
\special{sh 1}%
\special{pa 3200 2000}%
\special{pa 3140 2034}%
\special{pa 3162 2038}%
\special{pa 3168 2062}%
\special{pa 3200 2000}%
\special{fp}%
\special{pa 3200 2000}%
\special{pa 3100 1800}%
\special{dt 0.045}%
\special{sh 1}%
\special{pa 3100 1800}%
\special{pa 3112 1870}%
\special{pa 3124 1848}%
\special{pa 3148 1852}%
\special{pa 3100 1800}%
\special{fp}%
\special{pa 3200 2000}%
\special{pa 3300 1800}%
\special{dt 0.045}%
\special{sh 1}%
\special{pa 3300 1800}%
\special{pa 3252 1852}%
\special{pa 3276 1848}%
\special{pa 3288 1870}%
\special{pa 3300 1800}%
\special{fp}%
\special{pa 3300 1800}%
\special{pa 3250 1600}%
\special{dt 0.045}%
\special{sh 1}%
\special{pa 3250 1600}%
\special{pa 3248 1670}%
\special{pa 3264 1652}%
\special{pa 3286 1660}%
\special{pa 3250 1600}%
\special{fp}%
\special{pa 3300 1800}%
\special{pa 3350 1600}%
\special{dt 0.045}%
\special{sh 1}%
\special{pa 3350 1600}%
\special{pa 3314 1660}%
\special{pa 3338 1652}%
\special{pa 3354 1670}%
\special{pa 3350 1600}%
\special{fp}%
\special{pa 3100 1800}%
\special{pa 3150 1600}%
\special{dt 0.045}%
\special{sh 1}%
\special{pa 3150 1600}%
\special{pa 3114 1660}%
\special{pa 3138 1652}%
\special{pa 3154 1670}%
\special{pa 3150 1600}%
\special{fp}%
\special{pa 3100 1800}%
\special{pa 3050 1600}%
\special{dt 0.045}%
\special{sh 1}%
\special{pa 3050 1600}%
\special{pa 3048 1670}%
\special{pa 3064 1652}%
\special{pa 3086 1660}%
\special{pa 3050 1600}%
\special{fp}%
%
\special{pn 8}%
\special{pa 2650 1600}%
\special{pa 2600 800}%
\special{dt 0.045}%
\special{pa 3350 1600}%
\special{pa 3400 800}%
\special{dt 0.045}%
\put(7.5000,-9.9000){\makebox(0,0)[lb]{$T^0$}}%
\put(18.0500,-10.1000){\makebox(0,0)[lb]{$T^1_{101}$}}%
\put(28.7500,-10.1000){\makebox(0,0)[lb]{$T^2_{101,1001}$}}%
%
\special{pn 20}%
\special{pa 1000 520}%
\special{pa 1700 520}%
\special{fp}%
\special{sh 1}%
\special{pa 1700 520}%
\special{pa 1634 500}%
\special{pa 1648 520}%
\special{pa 1634 540}%
\special{pa 1700 520}%
\special{fp}%
\special{pa 2100 520}%
\special{pa 2800 520}%
\special{fp}%
\special{sh 1}%
\special{pa 2800 520}%
\special{pa 2734 500}%
\special{pa 2748 520}%
\special{pa 2734 540}%
\special{pa 2800 520}%
\special{fp}%
\put(4.0000,-5.9000){\makebox(0,0)[lb]{$G:$}}%
\put(8.8000,-15.6000){\makebox(0,0)[lb]{$101$}}%
\put(19.1000,-13.7000){\makebox(0,0)[lb]{$1001$}}%
\end{picture}%

%% file: NRMP_fullproof5.tex
\section{Weihrauch Degrees and Wadge Games}

\subsection{Weihrauch Degrees and Constructive Principles}

\subsubsection{Basic Notation}

We can also give a characterization of our nonuniformly computable functions in the context of the Weihrauch degrees which is a generalization of the Medvedev degrees.
Then, our results could be translated into the results on the Weihrauch degrees.
A partial function $P:\subseteq\nn^\nn\to\mathcal{P}(\nn^\nn)$ is called {\em a multi-valued function}.
\index{multi-valued function}%
Then $P$ is also written as $P:\nn^\nn\rightrightarrows\nn^\nn$.
One can think of each multi-valued function $P$ as a collection $\{P(x)\}_{x\in{\rm dom}(P)}$ of mass problems $P(x)\subseteq\nn^\nn$, or a {\em $\Pi_2$-theorem} $(\forall x\in{\rm dom}(P))(\exists y)\;y\in P(x)$.

\begin{definition}[\cite{BMP,BGa,BG,BGM}]
Let $P:\subseteq\nn^\nn\rightrightarrows\nn^\nn$ and $Q:\subseteq\nn^\nn\rightrightarrows\nn^\nn$ be multi-valued partial functions.
\index{realizer}\index{Weihrauch reducibility}\index{$P\leq_WQ$}%
\begin{enumerate}
\item A single-valued function $q:\subseteq\nn^\nn\to\nn^\nn$ is said to be {\em a realizer of $Q$} if $q(x)\in Q(x)$ for any $x\in{\rm dom}(Q)$.
\item We say that $P$ {\em is Weihrauch reducible to} $Q$ (written $P\leq_WQ$) if there are partial computable functions $H,K$ such that $K(x,q\circ H(x))$ for any $x\in{\rm dom}(P)$ and any realizer $q$ of $Q$.
\end{enumerate}
\end{definition}

\begin{remark}
If we think of the values $P(x)$ and $Q(x)$ as {\em relativized mass problems} $P^x$ and $Q^x$, then $P\leq_WQ$ can be represented as the existence of partial computable functions $\Phi,\Delta:\subseteq\nn^\nn\to\nn^\nn$ satisfying $\Phi^x:Q^{\Delta(x)}\to P^x$ for any $x\in{\rm dom}(Q)$, where $\Phi^x$ is the $x$-computable function mapping $y\in\nn^\nn$ to $\Phi(x\oplus y)$.
For any subset $P$ of Baire space $\nn^\nn$, we define $\iota(P):\nn^\nn\rightrightarrows\nn^\nn$ by $\iota(P)(x)=P$ for any $x\in\nn^\nn$.
Then, the map $\iota$ provides an embedding of the Medvedev degrees into the Weihrauch degrees, i.e., $P\leq^1_1Q$ if and only if $\iota(P)\leq_W\iota(Q)$.
\end{remark}

\begin{definition}[\cite{BMP,BGa,BG,BGM}]
Let $P,Q:\subseteq\nn^\nn\rightrightarrows\nn^\nn$ be partial multi-valued functions.
\index{$\lrangle{P,Q}$}\index{$P\times Q$}%
\index{$P\coprod Q$}\index{$P\circ Q$}%
\index{parallelization}\index{$\widehat{P}$}%
\begin{enumerate}
\item (Pairing) $\lrangle{P,Q}(x)=P(x)\times Q(x)$.
\item (Product) $(P\times Q)(\lrangle{x,y})=P(x)\times Q(y)$.
\item (Coproduct) $(P\coprod Q)(0,x)=\{0\}\times P(x)$; and $(P\coprod Q)(1,x)=\{1\}\times Q(x)$.
\item (Composition) $(P\circ Q)(x)=\bigcup\{P(y):y\in Q(x)\}$, where $x\in{\rm dom}(P\circ Q)$ if $x\in{\rm dom}(Q)$ and $Q(x)\subseteq{\rm dom}(P)$.
\item (Parallelization) $\widehat{P}(\lrangle{x_i:i\in\nn})=\prod_{i\in\nn}P(x_i)$.
\end{enumerate}
\end{definition}

Let $X$ be a computable metric space (for definition, see Weihrauch \cite{Wei}).
Then, $\mathcal{A}_-(X)$ denotes the hyperspace of closed subsets of $X$ with the upper Fell representation $\psi_-$ (see \cite{BMP}).
\index{$\mathcal{A}_-(X)$}%
For example, $P$ is a computable point in the hyperspace $\mathcal{A}_-(\nn^\nn)$ (resp.\ $\mathcal{A}_-(2^\nn)$) if and only if $P$ is a $\Pi^0_1$ subset of Baire space $\nn^\nn$ (resp.\ of Cantor space $2^\nn$).

\begin{definition}[Closed Choice \cite{BMP,BGa,BG,BGM}]
Let $X$ be a computable metric space.
Then, the {\em closed choice} operation of $X$ is defined as the following partial function.
\index{closed choice}\index{${\sf C}_X$}%
\[{\sf C}_X:\subseteq\mathcal{A}_-(X)\rightrightarrows X,\quad A\mapsto A\]
Here, ${\rm dom}({\sf C}_X)=\{A\in\mathcal{A}_-(X):A\not=\emptyset\}$.
\end{definition}

\subsubsection{Principles of Omniscience}

\begin{definition}
\index{tame formula}\index{${\rm TameForm}$}%
A formula is {\em tame} if it is well-formed formula constructed from symbols $\{\top,\bot,\wedge,\vee,\neg,\forall n,\exists n\}_{n\in\nn}$ and one variable symbol $\mathbf{V}(n)$ with a number parameter $n\in\nn$.
For any tame formula $A$ and $p\in\nn^\nn$, let $A[\mathbf{V}/p]$ denote the new formula obtained from $A$ by replacing $\mathbf{V}(n)$ with $\top$ if $p(n)=0$ and $\mathbf{V}(n)$ with $\bot$ if $p(n)\not=0$.
Then, let ${\rm TameForm}$ denote the class of formulas of the form $A\longrightarrow B$ for some tame formulas $A$ and $B$.
\end{definition}

\begin{example}\label{exa:5:intui}
The following formulas are contained in ${\rm TameForm}$.
\index{$\Sigma^0_1\text{-}{\sf LEM}$}\index{$\Sigma^0_2\text{-}{\sf LEM}$}%
\index{$\Sigma^0_2\text{-}{\sf DNE}$}\index{$\Sigma^0_1\text{-}{\sf LLPO}$}%
\index{$\Sigma^0_3\text{-}{\sf DNE}$}\index{$\Sigma^0_2\text{-}{\sf LLPO}$}%
\begin{enumerate}
\item $\Sigma^0_1\text{-}{\sf LEM}:\;\top\longrightarrow\exists n\mathbf{V}(n)\vee\neg\exists n\mathbf{V}(n)$.
\item $\Sigma^0_2\text{-}{\sf LEM}:\;\top\longrightarrow\exists m\forall n\mathbf{V}(\lrangle{m,n})\vee\neg\exists m\forall n\mathbf{V}(\lrangle{m,n})$.
\item $\Sigma^0_2\text{-}{\sf DNE}:\;\neg\neg\exists m\forall n\mathbf{V}(\lrangle{m,n})\longrightarrow\exists m\forall n\mathbf{V}(\lrangle{m,n})$.
\item $\Sigma^0_3\text{-}{\sf DNE}:\;\neg\neg\exists k\forall m\exists n\mathbf{V}(\lrangle{k,m,n})\longrightarrow\exists k\forall m\exists n\mathbf{V}(\lrangle{k,m,n})$.
\item $\Sigma^0_1\text{-}{\sf LLPO}:\;\neg(\exists n\mathbf{V}(\lrangle{0,n})\wedge\exists n\mathbf{V}(\lrangle{1,n}))\longrightarrow\neg\exists n\mathbf{V}(\lrangle{0,n})\vee\neg\exists n\mathbf{V}(\lrangle{1,n})$.
\item $\Sigma^0_2\text{-}{\sf LLPO}:\;\neg(\exists m\forall n\mathbf{V}(\lrangle{0,m,n})\wedge\exists m\forall n\mathbf{V}(\lrangle{1,m,n}))\longrightarrow\neg\exists m\forall n\mathbf{V}(\lrangle{0,m,n})\vee\neg\exists m\forall n\mathbf{V}(\lrangle{1,m,n})$.
\end{enumerate}
\end{example}

\begin{remark}
\index{excluded middle!law of}\index{double negation elimination}%
\index{principle of omniscience!lessor limited}\index{de Morgan's law}%
The symbols ${\sf LEM}$, ${\sf DNE}$, ${\sf LLPO}$ express the {\em law of excluded middle}, the {\em double negation elimination}, and the {\em lessor limited principle of omniscience} (i.e., {\em de Morgan's law}), respectively.
\end{remark}

\begin{definition}
\index{$\mathbb{F}_{A\longrightarrow B}$}%
Given any $A\longrightarrow B\in{\rm TameForm}$, we define a partial multivalued function $\mathbb{F}_{A\longrightarrow B}:\subseteq\nn^\nn\rightrightarrows\nn^\nn$ as follows:
\begin{align*}
{\rm dom}(\mathbb{F}_{A\longrightarrow B})&=\{p\oplus q\in\nn^\nn:q\in\bhk{A[\mathbf{V}/p]}\},\\
\mathbb{F}_{A\longrightarrow B}(p\oplus q)&=\bhk{B[\mathbf{V}/p]},
\end{align*}
where $\bhk{\cdot}:{\sf Form}\to\mathcal{P}(\nn^\nn)$ is a unique Medvedev interpretation in Definition \ref{def:1-2:Med-interpre} with $\bhk{\top}=\nn^\nn$.
\end{definition}

One can easily see that either $\bhk{\neg\varphi}=\nn^\nn$ or $\bhk{\neg\varphi}=\emptyset$ holds for every arithmetical sentence $\varphi$ in any Medvedev interpretation.
Therefore, for every principle $A\longrightarrow B$ in Example \ref{exa:5:intui}, its domain is $\{p\oplus q\in\nn^\nn:\bhk{A[\mathbf{V}/p]}\not=\emptyset\}$, that is, we need not to use the information on $q$.
This observation immediately implies the following proposition.

\begin{prop}\label{def:5:intpris}
The induced function $\mathbb{F}_{A\longrightarrow B}$ from a principle $A\longrightarrow B$ in Example \ref{exa:5:intui} is Weihrauch equivalent to the following associated partial multi-valued function $A\longrightarrow B$ on Baire space.
\index{$\Sigma^0_1\text{-}{\sf LEM}$}\index{$\Sigma^0_2\mbox{-}{\sf LEM}$}%
\index{$\Sigma^0_2\mbox{-}{\sf DNE}$}\index{$\Sigma^0_3\mbox{-}{\sf DNE}$}%
\index{$\Sigma^0_1\text{-}{\sf LLPO}$}\index{$\Sigma^0_2\mbox{-}{\sf LLPO}$}%
\begin{align*}
\Sigma^0_1\text{-}{\sf LEM}&:\nn^\nn\to 2, &
\Sigma^0_1\text{-}{\sf LEM}(p)&=
\begin{cases}
0, \mbox{ if }(\exists n\in\nn)\;p(n)=0,\\
1, \mbox{ otherwise.}
\end{cases}
\\
\Sigma^0_2\mbox{-}{\sf LEM}&:\nn^\nn\rightrightarrows 2\times\nn,&
\Sigma^0_2\mbox{-}{\sf LEM}(p)&\ni
\begin{cases}
(0,s), \mbox{ if } (\forall m\in\nn)(\exists n>m)\;p(n)=0,\\
(1,s), \mbox{ if } (\forall n>s)\;p(n)\not=0.
\end{cases}
\\
\Sigma^0_2\mbox{-}{\sf DNE}&:\subseteq\nn^\nn\rightrightarrows\nn,&
\Sigma^0_2\mbox{-}{\sf DNE}(p)&=\{m\in\nn:(\forall n>m)\;p(n)\not=0\}.
\\
\Sigma^0_3\mbox{-}{\sf DNE}&:\subseteq\nn^\nn\rightrightarrows\nn,&
\Sigma^0_3\mbox{-}{\sf DNE}(p)&=\{k:(\forall m\in\nn)(\exists n\geq m)\;p(\lrangle{k,n})=0\}.
\\
\Sigma^0_1\text{-}{\sf LLPO}&:\subseteq(\nn^\nn)^2\rightrightarrows 2,&
\Sigma^0_1\mbox{-}{\sf LLPO}(p_0,p_1)&\ni
\begin{cases}
0, \mbox{ if } (\forall n\in\nn)\;p_0(n)=0,\\
1, \mbox{ if } (\forall n\in\nn)\;p_1(n)=0.
\end{cases}
\\
\Sigma^0_2\mbox{-}{\sf LLPO}&:\subseteq(\nn^\nn)^2\rightrightarrows 2,&
\Sigma^0_2\mbox{-}{\sf LLPO}(p_0,p_1)&\ni
\begin{cases}
0, \mbox{ if } (\forall m)(\exists n>m)\;p_0(n)=0,\\
1, \mbox{ if } (\forall m)(\exists n>m)\;p_1(n)=0.
\end{cases}
\end{align*}
Here, their domains are given as follows.
\begin{align*}
{\rm dom}(\Sigma^0_2\mbox{-}{\sf DNE})&=\{p\in\nn^\nn:(\exists m\in\nn)(\forall n>m)\;p(n)\not=0\}.\\
{\rm dom}(\Sigma^0_3\mbox{-}{\sf DNE})&=\{p\in \nn^\nn:(\exists k\in\nn)(\forall m\in\nn)(\exists n\geq m)\;p(\lrangle{k,n})=0\}.\\
{\rm dom}(\Sigma^0_1\text{-}{\sf LLPO})&=\{(p_0,p_1)\in(\nn^\nn)^2:(\exists i<2)(\forall n\in\nn)\;p_i(n)=0\}.\\
{\rm dom}(\Sigma^0_2\mbox{-}{\sf LLPO})&=\{(p_0,p_1)\in(\nn^\nn)^2:(\exists i<2)(\forall m)(\exists n>m)\;p_i(n)=0\}.
\end{align*}
\end{prop}

\begin{remark}
\begin{enumerate}
\item
\index{principle of omniscience!limited}\index{{\sf LPO}}%
\index{computability!with finitely many mind changes}%
The single-valued function $\Sigma^0_1$-${\sf LEM}$ is usually called {\em the limited principle of omniscience} ({\sf LPO}).
Brattka-de Brecht-Pauly \cite{BMP} showed that a single-valued partial function $f:\subseteq\nn^\nn\to\nn^\nn$ is $(1,\omega)$-computable if and only if $f$ is Weihrauch reducible to the closed choice ${\sf C}_\nn$ for the discrete space $\nn$.
Here, in their term, the $(1,\omega)$-computability is called {\em the computability with finitely many mind changes}.
\item
$\Sigma^0_2$-{\sf LLPO} is Weihrauch equivalent to the {\em jump ${\sf LLPO}'$ of} {\sf LLPO} in the sense of Brattka-Gherardi-Marcone \cite{BGM}.
They also showed that ${\sf LLPO}'$ is Weihrauch equivalent to the Borzano-Weierstrass Theorem ${\sf BWT}_2$ for the discrete space $\{0,1\}$.
Brattka-Gherardi-Marcone \cite{BGM} also pointed out that the $n$-th jump of ${\sf LLPO}$ and ${\sf LPO}$ correspond to $\Sigma^0_{n+1}$-${\sf LLPO}$ (that is, the lessor limited principle of omniscience for $\Sigma^0_{n+1}$-formulas) and $\Sigma^0_{n+1}$-${\sf LEM}$ (the law of excluded middle for $\Sigma^0_{n+1}$-formulas), respectively.
\item
The study of arithmetical hierarchy of semiclassical principles such as $\Sigma^0_n$-{\sf LEM}, $\Sigma^0_n$-{\sf LLPO}, and $\Sigma^0_n$-{\sf DNE} was initiated by Akama-Berardi-Hayashi-Kohlenbach \cite{ABHK}.
In particular, on the study of the second level of arithmetical hierarhcy for semiclassical principles, see also Berardi \cite{Ber06} and Toftdal \cite{Toftdal04}.
The relationship between the learnability and $\Sigma^0_2$-{\sf DNE} has been also studied by Nakata-Hayashi \cite{Nakata-Hayashi} in the context of a realizability interpretation of limit computable mathematics.
\end{enumerate}
\end{remark}

\begin{definition}[Unique variant \cite{BGM}]
\index{unique variant}\index{${\sf Unique}P$}%
Let $P:X\rightrightarrows Y$ be a multi-valued function.
Then ${\sf Unique}P:X\rightrightarrows Y$ is defined as the restriction of $P$ up to ${\rm dom}({\sf Unique}P)=\{x\in{\rm dom}(P):\# P(x)=1\}$.
\end{definition}

\begin{definition}
\index{$\Delta^0_2\text{-}{\sf LEM}$}%
We define the partial multi-valued function $\Delta^0_2\text{-}{\sf LEM}$ as follows.
\begin{align*}
\Delta^0_2\text{-}{\sf LEM}:\subseteq\nn^2\times\nn^\nn\to 2, & & 
\Delta^0_2\text{-}{\sf LEM}(i,j,p)=
\begin{cases}
0, \mbox{ if }p\in{\rm Tot}_i,\\
1, \mbox{ otherwise.}
\end{cases}
\end{align*}
Here, ${\rm dom}(\Delta^0_2\mbox{-}{\sf LEM})=\{(i,j,p)\in\nn^2\times\nn^\nn:{\rm Tot}_i=\nn^\nn\setminus{\rm Tot}_j\}$, where ${\rm Tot}_e$ denotes the set of all oracles $\alpha\in\nn^\nn$ such that $\Phi_e(\alpha;n)$ converges for all inputs $n\in\nn$.
\end{definition}

\begin{prop}\label{prop:5:del-uniq}
$\Delta^0_2$-{\sf LEM} is Weihrauch reducible to ${\sf Unique}\Sigma^0_2\text{-}{\sf LLPO}$.
\end{prop}

\begin{proof}
To see $\Delta^0_2\text{-}{\sf LEM}\leq_W{\sf Unique}\Sigma^0_2\text{-}{\sf LLPO}$, given $(e_0,e_1,p)\in\nn^2\times\nn^\nn$, define $H(e_0,e_1,p)$ to be a pair $(x_0,x_1)$, where $x_i(s)=0$ if and only if the computation $\Phi_{e_i,s+1}(p)$ at stage $s+1$ properly extends $\Phi_{e_i,s}(p)$ at the previous stage.
Then $x_i$ contains infinitely many $0$'s if and only if $p$ is contained in ${\rm Tot}_{e_i}$.
Note that, whenever $(e_0,e_1,p)$ is contained in the domain of $\Delta^0_2$-{\sf LEM}, $H(e_0,e_1,p)$ is also contained in the domain of ${\sf Unique}\Sigma^0_2\text{-}{\sf LLPO}$, since ${\rm Tot}_{e_0}=\nn^\nn\setminus{\rm Tot}_{e_1}$.
Therefore, ${\sf Unique}\Sigma^0_2\text{-}{\sf LLPO}\circ H(e_0,e_1,p)=\Delta^0_2\text{-}{\sf LEM}(e_0,e_1,p)$.
%
\end{proof}

\begin{theorem}\label{learn-conpri}
Let $f:\subseteq\nn^\nn\to\nn^\nn$ be a single-valued partial function.
\begin{enumerate}
\item $f$ is $(1,2)$-computable if and only if $f\leq_W\Sigma^0_1$-${\sf LEM}$.
\item $f$ is $(1,\omega|2)$-computable if and only if $f\leq_W\Delta^0_2\mbox{-}{\sf LEM}$.
\item $f$ is $(1,\omega)$-computable if and only if $f\leq_W\Sigma^0_2\mbox{-}{\sf DNE}$.
\end{enumerate}
\end{theorem}

\begin{proof}\upshape
(1)
Let $f$ be a $(1,2)$-computable function.
By Theorem \ref{thm:5:red-eq-dis} (1), we have $f\in{\rm dec}^2_{\rm d}[\Pi^0_1]$.
Then, there is a $\Sigma^0_1$ set $S\subseteq\nn^\nn$ such that $f_0=f\res S$ and $f_1=f\res\nn^\nn\setminus S$ is computable.
Put $U=\{p\in\nn^\nn:(\exists n)\;p(n)=0\}$.
Note that $\Sigma^0_1\mbox{-}{\sf LEM}$ is the characteristic function $\mathbf{1}_U$ of $U$.
By $\Sigma^0_1$ completeness of $U$, we can find a Wadge reduction (indeed, a computable function) $H$ such that $\mathbf{1}_S=\mathbf{1}_U\circ H$.
Put $K(x,i)=f_i(x)$ for every $i<2$ and $x\in\nn^\nn$.
Then, for every $x\in{\rm dom}(f)$,
\[
K(x,\mathbf{1}_U\circ H(x))=K(x,\mathbf{1}_S(x))=
\begin{cases}
K(x,0)=f_0(x)&\mbox{ if }x\in S,\\
K(x,1)=f_1(x)&\mbox{ if }x\not\in S.
\end{cases}
\]

Conversely, we have $\Sigma^0_1\mbox{-}{\sf LEM}=\mathbf{1}_U\in{\rm dec}^2_{\rm d}[\Pi^0_1]$ since $U$ is $\Sigma^0_1$.
This implies that $H\circ\lrangle{id,\mathbf{1}_U\circ H}\in{\rm dec}^2_{\rm d}[\Pi^0_1]$ for every partial computable functions $H$ and $K$.

\medskip

(2)
Let $f$ be a $(1,\omega|2)$-computable function.
By Theorem \ref{thm:5:red-eq-dis} (2), we have $f\in{\rm dec}^2_{\rm d}[\Delta^0_2]$.
Then, there are $\Pi^0_2$ sets $P_0,P_1\subseteq\nn^\nn$ with $P_0=\nn^\nn\setminus P_1$ such that $f\res P_0$ and $f\res P_1$ are computable.
Then, we can find indices $i,j$ such that $P_0={\rm Tot}_i$ and $P_1={\rm Tot}_j$.
Let $H$ be the function sending $p\in\nn^\nn$ to $(i,j,p)$.
Put $K(x,i)=f_i(x)$ for every $i<2$ and $x\in\nn^{\nn}$.
It is not hard to see that $K(x,\Delta^0_2\mbox{-}{\sf LEM}\circ H(x))=f(x)$ for every $x\in{\rm dom}(f)$.

We show the converse implication.
By Proposition \ref{prop:5:del-uniq}, we have $f\leq_W\Delta^0_2\mbox{-}{\sf LEM}\leq_W{\sf Unique}\Sigma^0_2\text{-}{\sf LLPO}$.
Assume that $f\leq_W{\sf Unique}\Sigma^0_2\text{-}{\sf LLPO}$ via partial computable functions $K:\subseteq\nn^\nn\times 2\to\nn^\nn$ and $H:\subseteq\nn^\nn\to(\nn^\nn)^2$.
Let $e_i$ be an index of $\lambda x.K(x,i)$ for each $i<2$.
We first compute $h(\sigma,i)=\#\{n<|H_i(\sigma)|:H_i(\sigma;n)=0\}$, where $H(\sigma)=\lrangle{H_i(\sigma)}_{i<2}$.
Then let $c(\sigma)$ be the least $i<2$ such that $h(\sigma,k)\leq h(\sigma,i)$ for any $k<2$.
Let us consider a learner $\Psi:\nn^{<\nn}\to\{e_i\}_{i<m}$ defined by $\Psi(\sigma)=e_{c(\sigma)}$.
For any $x\in{\rm dom}(f)$, we have $H(x)\in{\rm dom}({\sf Unique}\Sigma^0_2\text{-}{\sf LLPO})$, and then $\lim_nh(x\res n,i)=\infty$ for just one $i<2$.
Then, $\lim_nc(x\res n)$ also converges to such $i<2$.
Moreover, for any $x\in{\rm dom}(f)$, ${\sf Unique}\Sigma^0_2\text{-}{\sf LLPO}(H(x))=\{i\}$ if and only if $\lim_nh(x\res n,i)=\infty$.
We fix a realizer $U$ of ${\sf Unique}\Sigma^0_2\text{-}{\sf LLPO}$, i.e., $U(x)\in{\sf Unique}\Sigma^0_2\text{-}{\sf LLPO}(x)$ for any $x\in{\rm dom}({\sf Unique}\Sigma^0_2\text{-}{\sf LLPO})$.
Then, $\lim_nc(x\res n)=U\circ H(x)$ for any $x\in{\rm dom}(f)$
Therefore, the limit $\lim_n\Psi(x\res n)$ converges to $e_{U\circ H(x)}$, and $\#{\tt indx}_\Psi(x)\leq\#\{e_i:i<2\}\leq 2$.
Thus, $\Phi_{\lim_n\Psi(x\res n)}(x)=\Phi_{e_{U\circ H(x)}}(x)=K(x,U\circ H(x))=f(x)$ for any $x\in{\rm dom}(f)$.
Hence, $f$ is $(1,\omega|2)$-computable.

\medskip

(3)
Clearly, $\Sigma^0_2$-{\sf DNE} is Weihrauch equivalent to the closed choice ${\sf C}_\nn$ for discrete space $\nn$.
Therefore, the desired condition follows from Brattka-Brecht-Pauly \cite{BMP}.
\end{proof}

\begin{definition}\label{def:5:nonconst}
\index{computable!relative to a nonconstructive principle $\mathsf{\Theta}$}\index{$\mathfrak{C}_{\mathsf{\Theta}}$}%
\index{$P\leq_{\mathsf{\Theta}}Q$}%
Let $\mathsf{\Theta}:\subseteq\nn^\nn\rightrightarrows\nn^\nn$ be a partial multi-valued function.
A partial {\em single}-valued function $f:\subseteq\nn^\nn\to\nn^\nn$ is {\em $\mathsf{\Theta}$-computable} if $f$ is Weihrauch reducible to $\mathsf{\Theta}$.
By $\mathfrak{C}_{\mathsf{\Theta}}$, we denote the least class containing all ${\mathsf{\Theta}}$-computable functions and closed under composition.
Then, for subsets $P,Q$ of $\nn^\nn$, we write $P\leq_{\mathsf{\Theta}}Q$ if $f:Q\to P$ for some $f\in\mathfrak{C}_{\mathsf{\Theta}}$.
\end{definition}

\begin{theorem}\label{learn-conpri2}
Let $P$ be a $\Pi^0_2$ subset of $\nn^\nn$, and $Q$ be any subset of $\nn^\nn$.
\begin{enumerate}
\item $P\leq^{<\omega}_1Q$ if and only if $P\leq_{\Sigma^0_2\mbox{-}{\sf LLPO}}Q$.
\item $P\leq^{\omega}_1Q$ if and only if $P\leq_{\Sigma^0_3\mbox{-}{\sf DNE}}Q$.
\end{enumerate}
\end{theorem}

\begin{proof}\upshape
(1)
If $P\leq^{<\omega}_1Q$ via two algorithms, we have a function $f:Q\to P$ with $f\in{\rm dec}^2_{\rm d}[\Pi^0_2]$ by Proposition \ref{prop:5:refref1} (3).
Then, $f_0=f\res Q_0$ and $f_1=f\res\nn^\nn\setminus Q_0$ are computable for some $\Pi^0_2$ set $Q_0\subseteq\nn^\nn$.
Since $f_1$ is computable, we can extend the domain of $f_0$ to a $\Pi^0_2$ set $Q^+$ including $\nn^\nn\setminus Q_0$.
Then $Q_1=Q^+\cap f_1^{-1}[P]$ is $\Pi^0_2$ since $P$ is $\Pi^0_2$ and $f_1$ is computable.
It is easy to see that $Q_0\cup Q_1$ includes $Q$.
Since $Q_0$ and $Q_1$ are $\Pi^0_2$, they are (computably) Wadge reducible to the $\Pi^0_2$ complete set $U=\{x\in\nn^\nn:(\exists^\infty n)\;x(n)=0\}$.
That is, for every $i<2$, there is a computable functions $H_i$ such that $\mathbf{1}_{Q_i}=\mathbf{1}_U\circ H_i$.
Let $H$ be a computable function sending $x\in\nn^\nn$ to the pair $(H_0(x),H_1(x))$, and put $K(x,i)=f_i(x)$.
We can easily see that
\[x\in Q_i\;\leftrightarrow\;\mathbf{1}_{Q_i}(x)=1\;\leftrightarrow\;\mathbf{1}_U(H_i(x))=1\;\leftrightarrow\;i\in\Sigma^0_2\mbox{-}{\sf LLPO}(H(x)).\]

Thus, for every realizer $G:\subseteq\nn^\nn\to 2$ of $\Sigma^0_2\mbox{-}{\sf LLPO}$, we have $K(x,G\circ H(x))=f_{G\circ H(x)}(x)\in P$.

If $f:Q\to P$ for a single-valued function $f\leq_W\Sigma^0_2\mbox{-}{\sf LLPO}$, then there are computable functions $H:\nn^\nn\to(\nn^\nn)^2$ and $K:\nn^\nn\times 2\to\nn^\nn$ such that $K(x,G\circ H(x))\in P$ for any realizer $G$ of $\Sigma^0_2\mbox{-}{\sf LLPO}$ and any element $x\in Q$.
Then $K(x,i)\in P$ for some $i<2$, since $G\circ H(x)<2$.
Set $\Phi_{e(i)}(x)=K(x,i)$ for each $i<2$.
Then $P\leq^{<\omega}_1Q$ via $\{\Phi_{e(i)}\}_{i<2}$.

\medskip

(2)
Assume that $P\leq^{\omega}_1Q$.
It suffices to show that $f:Q\to P$ for some $f\leq_W\Sigma^0_3\mbox{-}{\sf DNE}$.
Note that the condition $\Phi_e(x)$ is total and belongs to $P$ is $\Pi^0_2$, uniformly in $e\in\nn$ and $x\in\nn^\nn$.
Thus, there is a computable function $H:\nn^\nn\to\nn^\nn$ satisfying that $H(x;e,n)=0$ for infinitely many $n\in\nn$ if and only if $\Phi_e(x)$ is total and belongs to $P$.
By our assumption, there is $e\in\nn$ such that $H(x;e,n)=0$ for infinitely many $n\in\nn$, for any $x\in Q$.
Therefore, $H(x)\in{\rm dom}(\Sigma^0_3\mbox{-}{\sf DNE})$ for any $x\in Q$, and, for any realizer $G$ of $\Sigma^0_3\mbox{-}{\sf DNE}$, $G\circ H(x)$ chooses $e<b$ such that $\Phi_e(x)\in P$.
Then, for a computable function $K:\nn^\nn\times\nn\to\nn^\nn$ mapping $(x,e)$ to $\Phi_e(x)$, we have $K(x,G\circ H(x))=\Phi_e(x)\in P$.

If $f:Q\to P$ for a single-valued function $f\leq_W\Sigma^0_3\mbox{-}{\sf DNE}$, then there are computable functions $H:\nn^\nn\to\nn^\nn$ and $K:\nn^\nn\times\nn\to\nn^\nn$ such that $K(x,G\circ H(x))\in P$ for any realizer $G$ of $\Sigma^0_3\mbox{-}{\sf DNE}$ and any element $x\in Q$.
Then $K(x,i)\in P$ for some $i\in\nn$, since $G\circ H(x)<m$.
Set $\Phi_{e(i)}(x)=K(x,i)$ for each $i\in\nn$.
Then $P\leq^{\omega}_1Q$ via $\{\Phi_{e(i)}\}_{i\in\nn}$.
\end{proof}

\begin{theorem}
Let $P$ and $Q$ be $\Pi^0_1$ subsets of $\nn^\nn$.
Then, $P\leq^{<\omega}_{tt,1}Q$ if and only if $P\leq_{\Sigma^0_1\mbox{-}{\sf LLPO}}Q$.
\end{theorem}

\begin{proof}
We assume that $P\leq^{<\omega}_{tt,1}Q$ via two truth-table functionals $f_0$ and $f_1$.
Note that $f^{-1}(P)$ is $\Pi^0_1$ whenever $f$ is total computable, and $P$ is $\Pi^0_1$.
Then, for $Q_i=Q\cap\Theta_i^{-1}(P)$, the domain $Q$ is covered by $Q_0\cup Q_1$.
By $\Pi^0_1$ completeness of $U=\{x:(\forall n)\;x(n)\not=0\}$, for every $i<2$, we have a computable function $H_i$ such that $\mathbf{1}_{Q_i}=\mathbf{1}_U\circ H_i$.
As in the proof of Theorem \ref{learn-conpri2} (2), we set $H:x\mapsto(H_0(x),H_1(x))$ and $K:(x,i)\mapsto f_i(x)$.
Then, it is not hard to see that the condition $P\leq_{\Sigma^0_1\mbox{-}{\sf LLPO}}Q$ is witnessed by $H$ and $K$

If $f:Q\to P$ for a single-valued function $f\leq_W\Sigma^0_1\mbox{-}{\sf LLPO}$, then there are computable functions $H:\nn^\nn\to(\nn^\nn)^2$ and $K:\nn^\nn\times 2\to\nn^\nn$ such that $K(x,G\circ H(x))\in P$ for any realizer $G$ of $\Sigma^0_1\mbox{-}{\sf LLPO}$ and any element $x\in Q$.
Then $K(x,i)\in P$ for some $i<2$, since $G\circ H(x)<2$.
For $U=\{x:(\forall n)\;x(n)\not=0\}$, define $D_i=H_i^{-1}[U]$, where $H(x)=(H_0(x),H_1(x))$.
The computability of $H_i$ implies that $D_i$ is $\Pi^0_1$.
Define $f_i:D_i\to\nn^\nn$ by $f_i(x)=K(x,i)$ on $D_i$.
Since $D_i$ is $\Pi^0_1$, $f_i$ has a total computable extension $\Phi_{e(i)}$.
Therefore, $P\leq^{<\omega}_{tt,1}Q$ via $\{\Phi_{e(i)}\}_{i<2}$.
\end{proof}

Recall from Remark after Theorem \ref{theorem:12c:charact} that $\leq_{\Sigma^0_2}$ is the reducibility relation induced by the disjunction operation $\bhk{\cdot\vee\cdot}_{\Sigma^0_2}$.

\begin{theorem}\label{learn-conpri3}
Let $P$ and $Q$ be any subsets of $\nn^\nn$.
Then, $P\leq_{\Sigma^0_2}Q$ if and only if $P\leq_{\Sigma^0_2\mbox{-}{\sf LEM}}Q$.
\end{theorem}

\begin{proof}\upshape
Assume that there are two computable functions $H:\nn^\nn\to\nn^\nn$ and $K:\nn^\nn\times 2\times\nn\to\nn^\nn$ such that $K(x,G\circ H(x))\in P$ for any $x\in Q$ and any realizer $G:\nn^\nn\to 2\times\nn$ of $\Sigma^0_2\mbox{-}{\sf LEM}$.
Then the $\Sigma^0_2$ sentence $(\exists v)\theta(v,x)$ is given by $(\exists v)(\forall n>v)H(x;n)\not=0$.
We also define $\Delta(x)=K(x,\lrangle{0,0})$, and $\Gamma_v(x)=K(x,\lrangle{1,v})$, for any $x\in\nn^\nn$.
Fix $x\in Q$.
If $\theta(v,x)$ is true, then there is a realizer $G$ of $\Sigma^0_2\mbox{-}{\sf LEM}$ mapping $H(x)$ to $(1,v)$.
Therefore, $\Gamma_v(x)=K(x,\lrangle{1,v})=K(x,G\circ H(x))\in P$.
If $(\forall v)\neg\theta(v,x)$ is true, then there is a realizer $G$ of $\Sigma^0_2\mbox{-}{\sf LEM}$ mapping $H(x)$ to $(0,0)$.
Therefore, $\Delta(x)=K(x,\lrangle{0,0})=K(x,G\circ H(x))\in P$.
Hence, by Theorem \ref{thm:5:lemiddle}, we obtain $\bhk{P\vee P}_{\Sigma^0_2}\leq^1_1Q$.

Conversely, we assume that $\bhk{P\vee P}_{\Sigma^0_2}\leq^1_1Q$.
Then, there are computable collection $\Delta$, $\{\Gamma_v\}_{v\in}$ of computable functions, and a $\Sigma^0_2$ sentence $\exists v\theta(v,x)$, as in Theorem \ref{thm:5:lemiddle}.
By analyzing the proof of Theorem \ref{thm:5:lemiddle}, we may assume that this $\Sigma^0_2$ sentence has an additional property that, if $\theta(v,x)$ is true and $v\leq u$, then $\theta(u,x)$ is also true.
For any $x\in\nn^\nn$, put $K(x,\lrangle{0,n})=\Delta(x)$ for each $n\in\nn$, and $K(x,\lrangle{1,v})=\Gamma_v(x)$.
From the $\Sigma^0_2$ sentence $\exists v\theta(v,x)$, we can easily construct a computable function $H:\nn^\nn\to\nn^\nn$ satisfying that $\theta(v,x)$ is true if and only if $H(x;n)\not=0$ for any $n>v$.
Fix $x\in Q$
If $\exists v\theta(v,x)$ is true, then any realizer $G$ of $\Sigma^0_2\mbox{-}{\sf LEM}$ mapping $H(x)$ to some $(1,v)$ witnessing $\theta(v,x)$.
Then, $K(x,G\circ H(x))=\Gamma_v(x)\in P$.
If $\forall v\neg\theta(v,x)$ is true, then any realizer $G$ of $\Sigma^0_2\mbox{-}{\sf LEM}$ mapping $H(x)$ to $(0,s)$ for some $s\in\nn$.
Then, $K(x,G\circ H(x))=\Delta(x)\in P$.
\end{proof}

\begin{cor}
Let $P$ and $Q$ be subsets of $\nn^\nn$, where $P$ is $\Pi^0_2$.
Then, $P\leq^{<\omega}_\omega Q$ if and only if $P\leq_{\Sigma^0_2\mbox{-}{\sf LEM}}Q$.
\end{cor}

\begin{proof}
By Proposition \ref{prop:5:refref1} (2) and Theorem \ref{learn-conpri3}.
\end{proof}

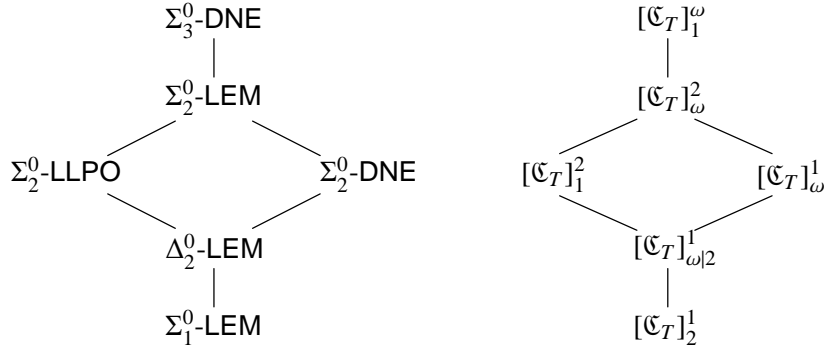
\begin{figure}[t]\centering
\begin{center}
\input{figure/LEM.tex}
\end{center}
\caption{Constructive principles, and nonuniform computability.}
  \label{fig:LEM}
\end{figure}

\subsection{Duality between Dynamic Operations and Nonconstructive Principles}

We now interpret our results in Section 4 in context of the Weihrauch degrees.

\begin{definition}[Le Roux-Pauly \cite{RP13}]
\index{$F\star G$}%
Let $F,G:\subseteq\nn^\nn\rightrightarrows\nn^\nn$ be any multi-valued functions.
Then, $F\star G=\max_{\leq_W}\{F^*\circ G^*:F^*\leq_WF\;\&\;G^*\leq_WG\}$.\qed
\end{definition}

If multi-valued functions $C,D:\subseteq\nn^\nn\rightrightarrows\nn^\nn$ satisfy the condition 
\[D\circ E\leq_WF\;\Longleftrightarrow\;E\leq_WC\star F\]
for any multi-valued functions $E,F:\subseteq\nn^\nn\rightrightarrows\nn^\nn$, then we may think of $D$ as the {\em inverse} of $C$.
One could think of our disjunction operators as inverse operators of various constructive principles.

\begin{definition}
Fix $x\in\nn^\nn$.
\index{$\tie(x)$}\index{$\tie_\omega(x)$}\index{$\tie_\infty(x)$}\index{$\widehat{\deg}_T(x)$}%
\begin{enumerate}
\item $\tie(x)=\{y\in(\nn\cup\{\sharp\})^\nn:\#\{n\in\nn:y(n)=\sharp\}\leq 1\;\&\;{\tt tail}(y)=x\}$.
\item $\tie_\omega(x)=\{y\in(2\times\nn)^\nn:(\exists i<2)\;{\tt pr}_i(y)=x\;\&\;{\tt mc}(y)<\infty\}$.
\item $\tie_\infty(x)=\{y\in(2\times\nn)^\nn:(\exists i<2)\;{\tt pr}_i(y)=x\}$.
\item $\widehat{\deg}_T(x)=\{y\in\nn^\nn:x\leq_Ty\}$.
\end{enumerate}
\end{definition}

\index{$\tie^{(n)}$}\index{$\tie_\omega^{(n)}$}\index{$\tie_\infty^{(n)}$}%
The $n$-th iteration of $\tie$ ($\tie_\omega$ and $\tie_\infty$) is denoted by $\tie^{(n)}$ ($\tie_\omega^{(n)}$ and $\tie_\infty^{(n)}$).
Here, recall from Remark below Definition \ref{def:12b:backtrack} that the symbol $\sharp$ is supposed to be updated each time.
For instance, $\tie^{(2)}$ refers to two special symbols $\sharp_0$ and $\sharp_1$, and then $\tie^{(n)}(x)$ can be identified with the set of all sequences $y$ such that $y$ contains at most $n$ many $\sharp$'s and ${\tt tail}(y)=x$.
More precisely, given a partial multi-valued function $E$, every element of $\tie^{(n)}\circ E(x)$ is of the form $\sigma_1\sharp\sigma_2\sharp\dots\sharp\sigma_n\sharp y$ with $y\in E(x)$.
Thus, $\tie^{(n)}\circ \Sigma^0_1\text{-}{\sf LEM}^n(x)$ has a computable realizer, and indeed, $\tie^{(n)}\circ E$ has a computable realizer for every $E\leq_W\Sigma^0_1\text{-}{\sf LEM}^n(x)$.
We will see more general results in Proposition \ref{prop:5:inverse-dis}.

A multi-valued function $P:\subseteq\nn^\nn\rightrightarrows\nn^\nn$ is {\em Popperian} if there is a computable function $r:\subseteq\nn^\nn\times\nn^\nn\to\nn^\nn$ satisfying $\Sigma^0_1\text{-}{\sf LEM}\circ r(x,y)=\mathbf{1}_{P(x)}(y)$, for any $x\in{\rm dom}(P)$ and $y\in\nn^\nn$, where $\mathbf{1}_{P(x)}$ denotes the characteristic function of $P(x)$.
In other words, $P$ is Popperian if and only if the condition $y\in P(x)$ is $\Pi^0_1$, uniformly in $x\in{\rm dom}(P)$ and $y\in\nn^\nn$.
Every Popperian multi-valued function is clearly Weihrauch reducible to the closed choice ${\sf C}_{\nn^\nn}$ of Baire space $\nn^\nn$.
\index{multi-valued function!Popperian}%

\begin{prop}\label{prop:5:inverse-dis}
Let $E,F:\nn^\nn\rightrightarrows\nn^\nn$ be any multi-valued functions.
\begin{enumerate}
\item $\tie^{(n)}\circ E\leq_WF$ if and only if $E\leq_W\Sigma^0_1\text{-}{\sf LEM}^n\star F$.
\item $\tie_{\omega}^{(n)}\circ E\leq_WF$ if and only if $E\leq_W{\sf Unique}\Sigma^0_2\mbox{-}{\sf LLPO}_{n}\star F$.
\item $\tie\circ E\leq_WF$ if and only if $E\leq_W\Sigma^0_2\text{-}{\sf DNE}\star F$, where $\tie=\bigcup_{n\in\nn}\tie^{(n)}$.
\end{enumerate}
Moreover, if $E$ is Popperian, then we also have the following conditions.
\begin{enumerate}
\item[4.] $\tie^{(n)}_\infty\circ E\leq_WF$ if and only if $E\leq_W\Sigma^0_2\mbox{-}{\sf LLPO}_n\star F$.
\item[5.] $\tie^{(n)}_\infty\circ\tie\circ E\leq_WF$ if and only if $E\leq_W(\Sigma^0_2)_2\mbox{-}{\sf LLPO}_n\star F$.
\item[6.] $\widehat{\deg}_T\circ E\leq_WF$ if and only if $E\leq_W\Sigma^0_3\mbox{-}{\sf DNE}\star F$.
\end{enumerate}
\end{prop}

\begin{proof}\upshape
(1) Assume that there are partial computable functions $H:\subseteq\nn^\nn\to\nn^\nn$ and $K:\subseteq\nn^\nn\times\nn^\nn\to\nn^\nn$ such that $K(x,f\circ H(x))\in\tie^{(n)}\circ E(x)$ for any $x\in{\rm dom}(\tie^{(n)}\circ E)$ and any realizer $f$ of $F$.
Then, for any realizer $f$ of $F$, we have the following condition for any $x\in{\rm dom}(E)$.
\[
K\circ({\rm id}\times f)\circ\lrangle{{\rm id},H}(x)=K(x,f\circ H(x))\in \tie^{(n)}\circ E(x)=\mbox{$\btie^1_n$} E(x).
\]

Note that $H^*=\lrangle{{\rm id},H}:\nn^\nn\to\nn^\nn\times\nn^\nn$ is computable, $F^*=K\circ({\rm id}\times F):\nn^\nn\times\nn^\nn\rightrightarrows\nn^\nn$ is Weihrauch reducible to $F$.
As in the proof of Theorem \ref{thm:5:red-eq-dis}, we can construct an $(1,n)$-computable function $\gamma:\btie^1_nE(x)\to E(x)$, uniformly in $x\in{\rm dom}(E)$.
Therefore, by Theorem \ref{learn-conpri}, we have a function $\gamma\leq_W\Sigma^0_1\text{-}{\sf LCM}^n$ satisfying $\gamma\circ f^*\circ H^*(x)\in E(x)$ for any $x\in{\rm dom}(E)$ and any realizer $f^*$ of $F^*$.
Consequently, $E\leq_W\Sigma^0_1\text{-}{\sf LEM}^n\star F$.

Conversely, we assume that $E\leq_WS^*\circ F^*$ for some $S^*\leq_W\Sigma^0_1\text{-}{\sf LEM}^n$ and $F^*\leq_WF$.
Then there are computable functions $H^*,K^*$ such that $K^*(x,H^*\circ f(x))\in F^*(x)$ for any realizer $f$ of $F$.
From any single valued function $f:\nn^\nn\to\nn^\nn$, we can effectively obtain $f^*(x)=K^*(x,H^*\circ f(x))$.
Assume that $S^*\leq_W\Sigma^0_1\text{-}{\sf LEM}^n$ via $\tilde{H}$ and $\tilde{K}$, and $E\leq_WS^*\circ F^*$ via $H$ and $K$.
We consider $H_f(x)=\tilde{H}\circ f^*\circ H(x)$ and $K_f(x,i)=K(x,\tilde{K}(f^*\circ H(x),i))$.
Then, we have the following condition for any $x\in{\rm dom}(E)$.
\[K_f(x,\Sigma^0_1\text{-}{\sf LCM}^n\circ H_f(x))\in E(x).\]

By calculating $H_f(x)=\tilde{H}\circ f^*\circ H(x)$, we can approximate $i(f;x)=\Sigma^0_1\text{-}{\sf LEM}^n\circ H_f(x)$ uniformly in $f$.
Therefore, we can construct $F^+_f$ to show $\tie^{(n)}\circ E\leq_WF$ by the following way.
Set $F^+_f(\lrangle{})=\lrangle{}$, fix $\sigma\in\nn^{<\nn}$, and assume that $F^+_f(\sigma^-)$ has been already defined.
If $i(f;\sigma)\not=i(f;\sigma^-)$, we put $F^+_f(\sigma)=F^+_f(\sigma^-)\fr\sharp\fr K_f(\sigma,i(f;\sigma))$.
Otherwise, $F^+_f$ continues the approximation of $K_f(\sigma,i(f;\sigma))$.
It is not hard to see that $F^+_f(x)\in\tie^{(n)}\circ E(x)$ for any $x\in{\rm dom}(E)$ and any realizer $f$ of $F$.
Then, $F^+_f$ is Weihrauch reducible to $\lrangle{K_f,H_f}$, and $\lrangle{K_f,H_f}$ is Weihrauch reducible to $f$.
Moreover, these reduction are not depend on $f$.
Hence, $\tie^{(n)}\circ E\leq_WF$.

(2,3) By the same argument as in the proof of the item (1).

(4)
Assume that $E:\subseteq\nn^\nn\rightrightarrows\nn^\nn$ is Popperian, and there are partial computable functions $H:\subseteq\nn^\nn\to\nn^\nn$ and $K:\subseteq\nn^\nn\times\nn^\nn\to\nn^\nn$ such that $K(x,f\circ H(x))\in\tie_\infty^{(n)}\circ E(x)$ for any $x\in{\rm dom}(\tie_\infty^{(n)}\circ E)$ and any realizer $f$ of $F$.
Then, for any realizer $f$ of $F$, we have the following condition for any $x\in{\rm dom}(E)$.
\[
K\circ({\rm id}\times f)\circ\lrangle{{\rm id},H}(x)=K(x,f\circ H(x))\in \tie_\infty^{(n)}\circ E(x)=\mbox{$\left[\btie_\infty\right]^1_n$}E(x).
\]

As in the proof of Theorem \ref{thm:5:red-eq-dis}, we can construct an $(n,1)$-computable function $\gamma:\left[\btie\right]^1_nE(x)\to E(x)$, uniformly in $x\in{\rm dom}(E)$.
Here, note that $E(x)$ is a $\Pi^0_1(x)$ subset of Cantor space, uniformly in $x$.
Therefore, by relativizing Theorem \ref{learn-conpri2}, we have a function $\gamma\leq_W\Sigma^0_2\text{-}{\sf LLPO}^n$ satisfying $\gamma\circ({\rm id}\times f)\circ\lrangle{{\rm id},H}(x)\in E(x)$ for any $x\in{\rm dom}(E)$ and any realizer $f$ of $F$.
Consequently, $E\leq_W\Sigma^0_1\text{-}{\sf LLPO}^n\star F$.

(5,6) By the same argument as in the proof of the item (4).
\end{proof}

\subsection{Borel Measurability, and Backtrack Games}

Berardi-Coquand-Hayashi \cite{BCH} showed that a {\em $1$-backtrack Tarski game} provides a semantics of positive arithmetical fragment of Limit Computable Mathematics (i.e., $\Delta^0_2$-mathematics, in the sense of Kleene realizability).
A positive arithmetical formula $A$ is true in the Limit Realizability Interpretation if and only if the $\exists$-player has a computable winning strategy in the {\em $1$-backtracking game} ${\sf bck}(\mathcal{G}(A))$ associated with the Tarski game for $A$ (for notations, see \cite{BCH}).
Meanwhile, Van Wesep \cite{VW} introduced {\em backtrack game} to study Wadge degrees, and Andretta \cite{And} used this game to characterize the $\mathbf{\Delta}^0_2$-measurable functions (also called the first level Borel functions) on Baire space $\nn^\nn$.
Motto Ros \cite{MR2} and Semmes \cite{Sem} studied more general games to study the Baire hierarchy of Borel measurable functions.
The hierarchy of Borel measurable functions are deeply studied in descriptive set theory \cite{Kec}.
We consider the following notions for a function $f$ on Baire space $\nn^\nn$ and a countable ordinal $\xi<\omega_1$.
\index{Borel function at level $\xi$}\index{S-function@$\mathbf{\Sigma}^0_{\xi+1,\xi+1}$ function}%
\index{S-measurable@$\mathbf{\Sigma}^0_{\xi+1}$-measurable}\index{Baire class $\xi$}%
\begin{enumerate}
\item $f$ is a {\em Borel function at level $\xi$} (or a $\mathbf{\Sigma}^0_{\xi+1,\xi+1}$ function; see \cite{Jay74,Jay79,MRpre,Sem}) if the preimage $f^{-1}(A)$ is $\mathbf{\Sigma}^0_{\xi+1}$ for every $\mathbf{\Sigma}^0_{\xi+1}$ set $A\subseteq\nn^\nn$.
\item $f$ is {\em $\mathbf{\Sigma}^0_{\xi+1}$-measurable} (or equivalently, of {\em Baire class $\xi$}; see for instance, Kechris \cite{Kec}) if the preimage $f^{-1}(A)$ is $\mathbf{\Sigma}^0_{\xi+1}$ for every open set $A\subseteq\nn^\nn$.
\end{enumerate}
Clearly, every level $\xi$ Borel function on Baire space $\nn^\nn$ is $\mathbf{\Sigma}^0_{\xi+1}$-measurable. 
The effective hierarchy of Borel measurable functions is studied by Brattka \cite{Bra1} and developed by many researchers (see \cite{deBre13,Kihta}).
An effective $\Sigma^0_\xi$ measurable function maps a computable point to a point of Turing degree $\dg{0}^{(n)}$.
Additionally, the class of (effectively) $\Sigma^0_\xi$-measurable functions does not closed under composition, in general, whereas the class of the level $\xi$ Borel functions must be closed under composition.
Our results (Theorem \ref{thm:5:red-eq-dis}) suggests that our notions of learnability is not like the effective $\Sigma^0_\xi$-measurability but more like effective versions of the level $\xi$ Borel functions, because of some results from descriptive set theory.

Recall from Definition \ref{def:1-2a:piecewise-computability} that ${\rm dec}_{\rm p}^\omega[\Gamma]\mathcal{F}$ denotes the class of $\Gamma$-piecewise $\mathcal{F}$ functions.
If $\mathcal{F}$ is the class of all partial continuous functions on Baire space, we abbreviate it as ${\bf dec}_{\rm p}^\omega[\Gamma]$.
\index{${\bf dec}_{\rm p}^\omega[\Gamma]$}%
Jayne-Rogers \cite{JR} proved that ${\bf dec}_{\rm p}^\omega[\mathbf{\Pi}^0_1]$ is exactly the class of the first level Borel functions, and Semmes \cite{Sem} showed that $f$ is ${\bf dec}_{\rm p}^\omega[\mathbf{\Pi}^0_2]$ is exactly the class of the second level Borel functions.

As shown in Theorem \ref{thm:5:red-eq-dis} and Proposition \ref{prop:5:refref1}, ${\rm dec}_{\rm p}^\omega[\Pi^0_1]$ is exactly the class of the learnable functions, and the degree structure $\mathcal{P}/{\rm dec}_{\rm p}^\omega[\Pi^0_2]$ is exactly the degree structure $\mathcal{P}^\omega_1$ induced from nonuniform computability.
Actually, our dynamic models directly fit into the backtrack and multitape game characterization of subclasses of Borel measurable functions.
We now introduce various games based on {\em the Wadge game}, {\em the backtrack game}, and {\em the multitape game}, 

\begin{definition}[see also Motto Ros \cite{MR2} and Semmes \cite{Sem}]
\index{$G(f,X)$}%
Fix a partial function $f$ on $\nn^\nn$, and a set $X$ which has no intersection with $\nn$.
The set $X$ may contain {\tt pass}, {\tt back}$\sharp$, $({\tt move},i)$ for each $i\in\nn$.
Then, we introduce various two-players games on $f$ as follows.
At every round $n\in\nn$, Player I chooses an element $x_n\in\nn$, and Player II chooses an element $y_n\in\nn\cup X$.

\begin{center}
\begin{tabular}{lccccccc}
I: & $x_0$ & & $x_1$ & & $x_2$ & & $\dots$ \\
II: &  & $y_0$ & & $y_1$ & & $y_2$ & $\dots$ \\
\end{tabular}
\end{center}

A pair of infinite sequences $\lrangle{x,y}\in\nn^\nn\times(\nn\cup X)^\nn$ is called {\em a play}.
Fix a play $\lrangle{x,y}$, where $x=\lrangle{x_n}_{n\in\nn}$ and $y=\lrangle{y_n}_{n\in\nn}$.
Player I constructs an input $x\in{\rm dom}(f)$ step by step, and Player II try to write a collect output $f(x)$ on some tape, where there may be infinitely many tapes $\{\Lambda_i\}_{i\in\nn}$.
Here, Player II can select a special symbol contained in $X$ at each step.
\begin{itemize}
\item $({\tt move},i)$ indicates the instruction to move the head on the $i$-th tape $\Lambda_i$.
\item {\tt pass} indicates that Player II writes no letter at this step.
\item {\tt back}$\sharp$ indicates the instruction to delete all words on the tape under the head.
\end{itemize}

Formally, we define the following notions.
For each $i\in\nn$, {\em the $i$-th content} of the play $y$ of Player II is a function ${\tt content}_i:(\nn\cup X)^\nn\to\nn^\nn$ which is inductively defined as follows.
Set ${\tt content}_i(\lrangle{})=\lrangle{}$ and ${\tt tape}(\lrangle{})=0$.
Assume that ${\tt content}_i(y\res n)$ and ${\tt tape}(y\res n)$ have been already defined for each $i\in\nn$.
\index{${\tt content}_i(y\res n)$}\index{${\tt tape}(y\res n)$}%
\begin{align*}
{\tt content}_i(y\res n+1)&=
\begin{cases}
{\tt content}_i(y\res n)\fr\lrangle{y_n} & \mbox{ if } y_n\in\nn\;\&\;i={\tt tape}(y\res n),\\
\lrangle{} & \mbox{ if } y_n={\tt back}\sharp\;\&\;i={\tt tape}(y\res n),\\
{\tt content}_i(y\res n) & \mbox{ otherwise.}
\end{cases}
\\
{\tt tape}(y\res n+1)&=
\begin{cases}
i & \mbox{ if } y_n=({\tt move},i),\\
{\tt tape}(y\res n) & \mbox{ otherwise.}
\end{cases}
\end{align*}
Then, for each $i\in\nn$, we define ${\tt content}_i(y)=\lim_{n\in\nn}{\tt content}_i(y\res n)$ for any $y\in(\nn\cup X)^\nn$.
\index{${\tt content}_i(y)$}%
We consider the following special {\em rules} for this game.
\index{rule}\index{rule!basic}\index{rule!rule $m$}\index{rule!rule $*$}%
\begin{itemize}
\item Player I {\em violates the basic rule} if $x\not\in{\rm dom}(f)$.
\item Player II {\em violates the basic rule} if either $y_n\in\{{\tt pass},({\tt move},i):i\in\nn\}$ for almost all $n\in\nn$, or $y_n={\tt back}\sharp$ for infinitely many $n\in\nn$.
\item Player II {\em violates the rule $m$} if $y$ contains at least $m$ many ${\tt back}\sharp$'s.
\item Player II {\em violate the rule $*$} if $y_n\in\{({\tt move},i): i\in\nn\}$ for infinitely many $n\in\nn$.
\end{itemize}
We say that Player II {\em wins} (resp.\ {\em is winnable}) on the play $\lrangle{x,y}\in\nn^\nn\times(\nn\cup X)^\nn$ of the game $G(f,X)$ if either Player II does not violate the basic rule, and $f(x)={\tt content}_i(y)$ for the least $i\in\nn$ with ${\tt content}_i(y)$ being total (resp.\ for some $i\in\nn$), or Player I violates the basic rule.
We also say that Player II {\em wins} (resp.\ {\em is winnable}) on the play $\lrangle{x,y}$ of the game $G_m(f,X)$ if Player II wins (resp.\ is winnable) the game $G(f,X)$ and does not violate the rule $m$, and that Player II {\em wins} (resp.\ {\em is winnable}) the game $G_*(f,X)$ if Player II wins (resp.\ is winnable) the game $G(f,X)$ and does not violate the rule $*$.

\index{strategy}\index{strategy!winning}%
\index{strategy!winnable}%
A {\em strategy} of Player II is a function $\psi:\nn^{<\nn}\to(\nn\cup X)^{<\nn}$ such that $|\psi(\sigma)|=|\sigma|$ for each $\sigma\in\omega^{<\omega}$, and $\psi(\sigma)\subseteq\psi(\tau)$ whenever $\sigma\subseteq\tau$.
A strategy $\psi$ of Player II is {\em winning} (resp.\ {\em winnable}) in the game $G$ if Player II wins (resp.\ is winnable) the game $G$ on the play $\lrangle{x,\bigcup_{n\in\nn}\psi(x\res n)}$ for any $x\in\nn^\nn$.

We write ${\sf P}$, ${\sf B}$, and ${\sf M}_\alpha$ for $\{{\tt pass}\}$, $\{{\tt back}\sharp\}$, and $\{({\tt move},i):i<\alpha\}$, respectively, for each $\alpha\leq\omega$.
Then, for ${\sf S},{\sf T},{\sf U}\in\{{\sf P},{\sf B},{\sf M}_\alpha\}_{\alpha\leq\omega}$, the union ${\sf S}\cup{\sf T}\cup{\sf U}$ is denoted by ${\sf STU}$.
\end{definition}

\begin{remark}
\index{game!Wadge}\index{game!backtrack}\index{game!multitape}%
The games $G(f,{\sf P})$, $G(f,{\sf PB})$, and $G(f,{\sf PM}_\omega)$ are essentially same as {\em the Wadge game}, {\em the backtrack game}, and {\em the multitape game}, respectively.
See also Motto Ros \cite{MR2} and Semmes \cite{Sem}.
\end{remark}

Let $f$ be a partial function on Baire space $\nn^\nn$.
\begin{enumerate}
\item {\rm (Wadge \cite{WWW})} $f$ is continuous if and only if Player II has a winning strategy in the game $G(f,{\sf P})$.
\item {\rm (Andretta \cite{And})} $f$ is $\mathbf{\Delta}^0_2$ if and only if Player II has a winning strategy in the game $G(f,{\sf PB})$.
\item {\rm (Andretta, Semmes \cite{Sem1})} $f$ is $\mathbf{\Pi}^0_2$-piecewise continuous if and only if Player II has a winning strategy in the game $G(f,{\sf PM}_\omega)$.
\end{enumerate}

\begin{theorem}[Game representation]\label{thm:5:game}
Let $f$ be a partial function on Baire space $\nn^\nn$.
\begin{enumerate}
\item $f$ is $(1,1)$-computable if and only if Player II has a computable winning strategy in the game $G(f,{\sf P})$.
\item $f$ is $(1,m)$-computable if and only if Player II has a computable winning strategy in the game $G_m(f,{\sf PB})$.
\item $f$ is $(1,\omega|m)$-computable if and only if Player II has a computable winning strategy in the game $G_*(f,{\sf PM}_m)$.
\item $f$ is $(1,\omega)$-computable if and only if Player II has a computable winning strategy in the game $G(f,{\sf PB})$.
\item $f$ is $(m,1)$-computable if and only if Player II has a computable winnable strategy in the game $G(f,{\sf PM}_m)$.
\item $f$ is $(m,\omega)$-computable if and only if Player II has a computable winnable strategy in the game $G(f,{\sf PBM}_m)$.
\item $f$ is $(\omega,1)$-computable if and only if Player II has a computable winnable strategy in the game $G(f,{\sf PM}_\omega)$.
\end{enumerate}
\end{theorem}

\begin{proof}\upshape
(2,4) We need to construct a winning strategy $\psi:\nn^{<\nn}\to(\nn\cup\{{\tt pass},{\tt back}\sharp\})^{<\nn}$ from a given partial $(1,\omega)$-computable function $f:\subseteq\nn^\nn\to\nn^\nn$.
Assume that $f$ is $(1,\omega)$-computable via a learner $\Psi$.
We inductively define a strategy $\psi:\nn^{<\nn}\to(\nn\cup\{{\tt pass},{\tt back}\sharp\})^{<\nn}$ and an auxiliary parameter ${\tt backlog}:\nn^{<\nn}\to(\nn\cup\{{\tt back}\sharp\})^{<\nn}$.
Set $\psi(\lrangle{})={\tt backlog}(\lrangle{})=\lrangle{}$, and assume that $\psi(\sigma^-)$ and ${\tt backlog}(\sigma^-)$ have been already defined.
Then, define $\psi(\sigma)$ and ${\tt backlog}(\sigma)$ as follows:
\begin{align*}
\psi(\sigma)&=
\begin{cases}
\psi(\sigma^-)\fr{\tt pass}&\mbox{ if }{\tt backlog}(\sigma^-)=\lrangle{},\\
\psi(\sigma^-)\fr({\tt backlog}(\sigma^-)(0))&\mbox{ if }{\tt backlog}(\sigma^-)\not=\lrangle{},
\end{cases}
\\
{\tt backlog}(\sigma)&=
\begin{cases}
{\tt backlog}(\sigma^-)^{\shft 1}\fr{\tt new}\Phi_{\Psi(\sigma)}(\sigma)&\mbox{ if }\Psi(\sigma)=\Psi(\sigma^-),\\
{\tt backlog}(\sigma^-)^{\shft 1}\fr{\tt back}\sharp\fr\Phi_{\Psi(\sigma)}(\sigma)&\mbox{ if }\Psi(\sigma)\not=\Psi(\sigma^-).
\end{cases}
\end{align*}

Here, recall the notation ${\tt new}\Phi_{\Psi(\sigma)}(\sigma)$ defined before Theorem \ref{theorem:12c:charact}.
Note that $\{n\in\nn:(\bigcup_k\psi(x\res k))(n)={\tt back}\sharp\}={\tt mcl}_\Psi(x)$ for any $x\in{\rm dom}(f)$.
It is easy to see that $\psi$ is a computable winning strategy in the game $G(f,{\sf PB})$.

Assume that a computable winning strategy $\psi^*$ in the game $G(f,{\sf PB})$ is given.
We consider the computable function $\psi(\sigma)={\tt content}_0(\psi^*(\sigma))$.
Then $\{n\in\nn:\psi(x\res n+1)\not\supseteq\psi(x\res n)\}$ is finite, for any $x\in{\rm dom}(f)$, since $\bigcup_{n\in\nn}\psi(x\res n)$ contains finitely many ${\tt back}\sharp$'s.
Moreover, $f(x)=\lim_n\psi(x\res n)$.
Thus, by Proposition \ref{prop:1-2:characterization}, $f$ is $(1,\omega)$-computable.

\medskip

(3)
Assume that $f$ is $(1,\omega|<\omega)$-computable via a learner $\Psi$.
We inductively define a strategy $\psi:\nn^{<\nn}\to(\nn\cup\{{\tt pass},{\tt back}\sharp\})^{<\nn}$ and an auxiliary parameter ${\tt backlog}:\nn^{<\nn}\to(\nn\cup\{{\tt back}\sharp\})^{<\nn}$.
Set $\psi(\lrangle{})={\tt backlog}(\lrangle{})=\lrangle{}$, and assume that $\psi(\sigma^-)$ and ${\tt backlog}(\sigma^-)$ have been already defined.
Then, define $\psi(\sigma)$ and ${\tt backlog}(\sigma)$ as follows:
\begin{align*}
\psi(\sigma)&=
\begin{cases}
\psi(\sigma^-)\fr{\tt pass}&\mbox{ if }{\tt backlog}(\sigma^-)=\lrangle{},\\
\psi(\sigma^-)\fr({\tt backlog}(\sigma^-)(0))&\mbox{ if }{\tt backlog}(\sigma^-)\not=\lrangle{},
\end{cases}
\\
{\tt backlog}(\sigma)&={\tt backlog}(\sigma^-)^{\shft 1}\fr({\tt move},\Psi(\sigma))\fr{\tt new}^*\Phi_{\Psi(\sigma)}(\sigma)
\end{align*}

Here, recall the notation ${\tt new}^*\Phi_{\Psi(\sigma)}(\sigma)$ defined in the proof of Theorem \ref{theorem:12c:charact} (2).
Note that $\{n\in\nn:(\bigcup_k\psi(x\res k))(n)={\tt back}\sharp\}=\{n\in\nn:\Psi(x\res n+1)\not=\Psi(x\res n)\}$ for any $x\in{\rm dom}(f)$.
It is easy to see that $\psi$ is a computable winning strategy in the game $G(f,{\sf PM}_m)$.
Moreover, since $\#{\tt indx}_\Psi(x)$ is finite, $\psi(x)=\bigcup_n\psi(x\res n)$ contains $({\tt move},i)$ for only finitely many different $i$'s.
Therefore, $\psi$ does not violate the rule $*$.
Hence, $\psi$ is a winning strategy in the game $G_*(f,{\sf PM}_m)$.

Assume that a computable winning strategy $\psi^*$ in the game $G_*(f,{\sf PM}_m)$ is given.
Let $e(i)$ be an index of a partial computable function $x\mapsto{\tt content}_i\circ\psi^*(x)$ for each $i<m$.
Since $\psi^*$ does not violate the rule $\ast$, there is a unique $i<m$ such that $\Phi_{e(i)}={\tt content}_i\circ\psi^*(x)$ is total, for any $x\in{\rm dom}(f)$.
We inductively define a learner $\Psi$.
The learner $\Psi$ first guesses $\Psi(\lrangle{})=e(0)$.
Set $\Psi(\sigma)=\Psi(\sigma^-)$ when there is no $i<m$ such that $|\Phi_{e(i)}(\sigma)|>|\Phi_{e(i)}(\sigma^-)|$.
Otherwise, for the least such $i<m$, the learner guesses $\Psi(\sigma)=e(i)$.
Clearly, $\#\{\Psi(x\res n):n\in\nn\}<m$ for any $x\in\nn^\nn$.
It is easy to check that, for any $x\in{\rm dom}(f)$, $\lim_n\Psi(x\res n)$ converges to $e(i)$ for the unique $i<m$ ensuring the totality of ${\tt content}_i\circ\psi^*(x)$, and, for such $i<m$, we have $\Phi_{\lim_n\Psi(x\res n)}(x)={\tt content}_i\circ\psi^*(x)=f(x)$.
Consequently, $f$ is $(1,\omega|m)$-computable.

\medskip

(5,7)
For a given collection $\{\Phi_i\}_{i\in I}$ of partial computable functions, we can easily construct a strategy $\psi:\nn^{<\nn}\to(\nn\cup\{{\tt pass},({\tt move},i):i\in I\})$ ensuring ${\tt content}_i\circ\psi(x)=\Phi_i(x)$ for any $x\in\nn^\nn$.
Therefore, $f$ is nonuniformly computable via $\{\Phi_i\}_{i\in I}$, then $\psi$ is winnable in $G(f,{\sf PM}_I)$.
Conversely, if a winnable strategy $\psi:\nn^{<\nn}\to(\nn\cup\{{\tt pass},({\tt move},i):i\in I\})$ of the game $G(f,{\sf PM}_I)$ is given.
Then we consider the partial computable function $\Gamma_i$ computing $\Gamma_i(x)={\tt content}_i\circ\psi(x)$ for any $x\in\nn^\nn$.
It is easy to see that $f$ is nonuniformly computable via $\{\Gamma_i\}_{i\in I}$.

(6) By combining the proofs of the items (3) and (4), it is not hard to see the equivalence of the $(m,\omega)$-computability of $f$ and the computable winnability in the game $G(f,{\sf PBM}_m)$.
\end{proof}

\begin{remark}
We may introduce more general multitape games based on our dynamic tape models, and nested (nested nested, nested nested nested, etc.) tape models.
\end{remark}

%% file: figure/LEM.tex
\unitlength 0.1in
\begin{picture}( 39.4000, 16.1000)(  8.0000,-28.3000)
\put(16.0000,-14.0000){\makebox(0,0)[lb]{$\Sigma^0_3$-{\sf DNE}}}%
\put(16.0000,-18.0000){\makebox(0,0)[lb]{$\Sigma^0_2$-{\sf LEM}}}%
\put(8.0000,-22.0000){\makebox(0,0)[lb]{$\Sigma^0_2$-{\sf LLPO}}}%
\put(24.0000,-22.0000){\makebox(0,0)[lb]{$\Sigma^0_2$-{\sf DNE}}}%
\put(16.0000,-26.0000){\makebox(0,0)[lb]{$\Delta^0_2$-{\sf LEM}}}%
\put(16.0000,-30.0000){\makebox(0,0)[lb]{$\Sigma^0_1$-{\sf LEM}}}%
\put(40.4000,-13.9000){\makebox(0,0)[lb]{$[\mathfrak{C}_T]^\omega_1$}}%
\put(40.4000,-17.9000){\makebox(0,0)[lb]{$[\mathfrak{C}_T]^2_\omega$}}%
\put(34.4000,-22.0000){\makebox(0,0)[lb]{$[\mathfrak{C}_T]^2_1$}}%
\put(46.6000,-22.0000){\makebox(0,0)[lb]{$[\mathfrak{C}_T]^1_\omega$}}%
\put(40.2000,-26.0000){\makebox(0,0)[lb]{$[\mathfrak{C}_T]^1_{\omega|2}$}}%
\put(40.2000,-30.0000){\makebox(0,0)[lb]{$[\mathfrak{C}_T]^1_2$}}%
%
\special{pn 8}%
\special{pa 1300 2210}%
\special{pa 1650 2390}%
\special{fp}%
\special{pa 2050 2390}%
\special{pa 2400 2210}%
\special{fp}%
%
\special{pn 8}%
\special{pa 1300 2010}%
\special{pa 1650 1830}%
\special{fp}%
\special{pa 2050 1830}%
\special{pa 2400 2010}%
\special{fp}%
%
\special{pn 8}%
\special{pa 4040 2400}%
\special{pa 3640 2220}%
\special{fp}%
%
\special{pn 8}%
\special{pa 4340 2400}%
\special{pa 4740 2220}%
\special{fp}%
%
\special{pn 8}%
\special{pa 4340 1820}%
\special{pa 4740 2000}%
\special{fp}%
%
\special{pn 8}%
\special{pa 4040 1820}%
\special{pa 3640 2000}%
\special{fp}%
%
\special{pn 8}%
\special{pa 4200 2800}%
\special{pa 4200 2600}%
\special{fp}%
\special{pa 4200 1600}%
\special{pa 4200 1400}%
\special{fp}%
%
\special{pn 8}%
\special{pa 1850 2800}%
\special{pa 1850 2600}%
\special{fp}%
\special{pa 1850 1600}%
\special{pa 1850 1400}%
\special{fp}%
\end{picture}%